\documentclass[11pt]{article}

\usepackage{amsmath, amsfonts, amsthm, amssymb, mathtools}
\usepackage{authblk, color, bm, graphicx, epstopdf, url}
\usepackage[font = small, labelfont = bf, labelsep = period]{caption}
\usepackage{subcaption}
\usepackage{xspace}
\usepackage{enumitem}

\usepackage{booktabs, tabularx, multirow}
\newcolumntype{C}{>{\centering\arraybackslash}X}
\newcolumntype{R}{>{\raggedleft\arraybackslash}X}
\newcolumntype{L}{>{\raggedright\arraybackslash}X}
\let\oldtabularx\tabularx
\renewcommand{\tabularx}{\footnotesize\oldtabularx}

\newtheorem{ex}{Example}

\newtheorem{prop}{Proposition}
\newtheorem{obs}{Observation}

\newtheorem{rem}{Remark}
\newtheorem{thm}{Theorem}

\usepackage{fullpage}[2cm]
\usepackage[page]{appendix}

\title{$K$-Adaptability in Two-Stage Mixed-Integer Robust Optimization}

\author[1]{Anirudh Subramanyam}
\author[1]{Chrysanthos E.~Gounaris}
\author[2]{Wolfram Wiesemann}

\affil[ ]{\texttt{asubramanyam@cmu.edu}, \texttt{gounaris@cmu.edu}, \texttt{ww@imperial.ac.uk}}
\affil[1]{\small Department of Chemical Engineering, Carnegie Mellon University, United States}
\affil[2]{\small Imperial College Business School, Imperial College London, United Kingdom}

\begin{document}

\maketitle

\begin{abstract}
We study two-stage robust optimization problems with mixed discrete-continuous decisions in both stages. Despite their broad range of applications, these problems pose two fundamental challenges: \emph{(i)} they constitute infinite-dimensional problems that require a finite-dimensional approximation, and \emph{(ii)} the presence of discrete recourse decisions typically prohibits duality-based solution schemes. We address the first challenge by studying a $K$-adaptability formulation that selects $K$ candidate recourse policies \emph{before} observing the realization of the uncertain parameters and that implements the best of these policies \emph{after} the realization is known. We address the second challenge through a branch-and-bound scheme that enjoys asymptotic convergence in general and finite convergence under specific conditions. We illustrate the performance of our algorithm in numerical experiments involving benchmark data from several application domains.

\end{abstract}

\section{Introduction}

Dynamic decision-making under uncertainty, where actions need to be taken both in anticipation of and in response to the realization of a priori uncertain problem parameters, arguably forms one of the most challenging domains of operations research and optimization theory. Despite intensive research efforts over the past six decades, many uncertainty-affected optimization problems resist solution, and even our understanding of the complexity of these problems remains incomplete.

In the last two decades, robust optimization has emerged as a promising methodology to counter some of the intricacies associated with decision-making under uncertainty. The rich theory on static robust optimization problems, in which all decisions have to be taken before the uncertainty is resolved, is summarized in \cite{BTEGN09:rob_opt, BBC11:roreview, GMT14:roreview}. However, dynamic robust optimization problems, in which some of the decisions can adapt to the observed uncertainties, are still poorly understood.

This paper is concerned with \emph{two-stage robust optimization problems} of the form
\begin{equation}\label{eq:two_stage_ro}
\inf_{\bm{x} \in \mathcal{X}} \; \sup_{\bm{\xi} \in \Xi} \; \inf_{\bm{y} \in \mathcal{Y}}
\left\{ \bm{c}^\top \bm{x} + \bm{d} (\bm{\xi})^\top \bm{y} \, : \, \bm{T} (\bm{\xi}) \bm{x} + \bm{W} (\bm{\xi}) \bm{y} \leq \bm{h} (\bm{\xi}) \right\},
\end{equation}
where $\mathcal{X} \subseteq \mathbb{R}^{N_1}$, $\mathcal{Y} \subseteq \mathbb{R}^{N_2}$ and $\Xi \subseteq \mathbb{R}^{N_p}$ constitute nonempty and bounded mixed-integer linear programming (MILP) representable sets,
$\bm{c} \in \mathbb{R}^{N_1}$, and the functions $\bm{d} : \Xi \mapsto \mathbb{R}^{N_2}$, $\bm{T} : \Xi \mapsto \mathbb{R}^{L \times N_1}$, $\bm{W} : \Xi \mapsto \mathbb{R}^{L \times N_2}$ and $\bm{h} : \Xi \mapsto \mathbb{R}^L$ are affine.
In problem~\eqref{eq:two_stage_ro}, the vector $\bm{x}$ represents the first-stage (or `here-and-now') decisions which are taken before the value of the uncertain parameter vector $\bm{\xi}$ from within the uncertainty set $\Xi$ is observed. The vector $\bm{y}$, on the other hand, denotes the second-stage (or `wait-and-see') decisions that can adapt to the realized value of $\bm{\xi}$. We emphasize that problem~\eqref{eq:two_stage_ro} can have a random recourse, i.e., the recourse matrix $\bm{W}$ may depend on the uncertain parameters $\bm{\xi}$. Moreover, we do not assume a relatively complete recourse; that is, for some first-stage decisions $\bm{x} \in \mathcal{X}$, there can be parameter realizations $\bm{\xi} \in \Xi$ such that there is no feasible second-stage decision $\bm{y}$. Also, we do not assume that the sets $\mathcal{X}$, $\mathcal{Y}$ or $\Xi$ are convex.

\begin{rem}[Uncertain First-Stage Objective Coefficients]\label{rem:uncertain_first_stage_objective_coefficients}
The assumption that $\bm{c}$ is deterministic does not restrict generality. Indeed, problem~\eqref{eq:two_stage_ro} accounts for uncertain first-stage objective coefficients $\bm{c}' : \Xi \mapsto \mathbb{R}^{N_1}$ if we augment the second-stage decisions $\bm{y}$ to $(\bm{y}, \bm{y}')$, replace the second-stage objective coefficients $\bm{d}$ with $(\bm{d}, \bm{c}')$ and impose the constraint that $\bm{y}' = \bm{x}$.
\end{rem}

Even in the special case where $\mathcal{X}$, $\mathcal{Y}$ and $\Xi$ are linear programming (LP) representable, problem~\eqref{eq:two_stage_ro} involves infinitely many decision variables and constraints, and it has been shown to be NP-hard~\cite{Gus02:aro_thesis}.
Nevertheless, problem~\eqref{eq:two_stage_ro} simplifies considerably if the sets $\mathcal{Y}$ and $\Xi$ are LP representable. For this setting, several approximate solution schemes have been proposed that replace the second-stage decisions with \emph{decision rules}, i.e., parametric classes of linear or nonlinear functions of $\bm{\xi}$~\cite{BTGGN04:adjustable, CZ09:extended_affinely_adj, GWK10:generalized_drs, GS10:dro_tractable_approximations, KWG11:primal_dual}. If we further assume that $\bm{d}$, $\bm{T}$ and $\bm{W}$ are deterministic and $\Xi$ is of simple form (e.g., a budget uncertainty set), a number of exact solution schemes based on Benders' decomposition \cite{BLSZZ13:adaptive_ro_security_constarinted_unit_commitment, JZLG10:twostage_robust_power_grid, TTE10:robust_lo_with_recourse, ZWWG13:multistage_robust_unit_commitment} and semi-infinite programming \cite{AP15:decomposition, ZL13:column_and_constraint} have been developed.

Problem~\eqref{eq:two_stage_ro} becomes significantly more challenging if the set $\mathcal{Y}$ is not LP representable. For this setting, conservative MILP approximations have been developed in \cite{dick13:hints, VKR11:decision_rules} by partitioning the uncertainty set $\Xi$ into hyperrectangles and restricting the continuous and integer recourse decisions to affine and constant functions of $\bm{\xi}$ over each hyperrectangle, respectively. These a priori partitioning schemes have been extended to iterative partitioning approaches in \cite{BD15:multistage_robust_mio, PdH15:multistage_adjustable_robust_mio}. Iterative solution approaches based on decision rules have been proposed in \cite{BG13:multistage_adaptive, Angele:BDRs}. However, to the best of our knowledge, none of these approaches has been shown to converge to an optimal solution of problem~\eqref{eq:two_stage_ro}. For the special case where problem~\eqref{eq:two_stage_ro} has a relatively complete recourse, $\bm{d}$, $\bm{T}$ and $\bm{W}$ are deterministic and the optimal value of the second-stage problem is quasi-convex over $\Xi$, the solution scheme of \cite{ZL13:column_and_constraint} has been extended in \cite{ZZ12:exact_milp_recourse} to a nested semi-infinite approach that can solve instances of problem~\eqref{eq:two_stage_ro} with MILP representable sets $\mathcal{Y}$ and $\Xi$ to optimality in finite time.

Instead of solving problem~\eqref{eq:two_stage_ro} directly, we study its \emph{$K$-adaptability problem}
\begin{equation}\label{eq:k_adaptability}
\inf_{\substack{\bm{x} \in \mathcal{X}, \\ \bm{y} \in \mathcal{Y}^K}} \; \sup_{\bm{\xi} \in \Xi} \; \inf_{k \in \mathcal{K}}
\left\{ \bm{c}^\top \bm{x} + \bm{d} (\bm{\xi})^\top \bm{y}_k \, : \, \bm{T} (\bm{\xi}) \bm{x} + \bm{W} (\bm{\xi}) \bm{y}_k \leq \bm{h} (\bm{\xi}) \right\},
\end{equation}
where $\mathcal{Y}^K = \bigtimes_{k=1}^K \mathcal{Y}$ 
and $\mathcal{K} = \{ 1, \ldots, K \}$. Problem~\eqref{eq:k_adaptability} determines $K$ non-adjustable second-stage policies $\bm{y}_1, \ldots, \bm{y}_K$ here-and-now and subsequently selects the best of these policies in response to the observed value of $\bm{\xi}$. If all policies are infeasible for some realization $\bm{\xi} \in \Xi$, then the solution $(\bm{x}, \bm{y})$ attains the objective value $+ \infty$. By construction, the $K$-adaptability problem~\eqref{eq:k_adaptability} bounds the two-stage robust optimization problem~\eqref{eq:two_stage_ro} from above. 

Our interest in problem~\eqref{eq:k_adaptability} is motivated by two observations. Firstly, problem~\eqref{eq:k_adaptability} has been shown to be a remarkably good approximation of problem~\eqref{eq:two_stage_ro}, both in theory and in numerical experiments~\cite{BC10:finite_adaptability, HKW15:rip}. Secondly, and perhaps more importantly, the $K$-adaptability problem conforms well with human decision-making, which tends to address uncertainty by developing a small number of contingency plans, rather than devising the optimal response for every possible future state of the world. For instance, practitioners may prefer a limited number of contingency plans to full flexibility in the second stage for operational (e.g., in production planning or logistics) or organizational (e.g., in emergency response planning) reasons.

The $K$-adaptability problem was first studied in~\cite{BC10:finite_adaptability}, where the authors reformulate the $2$-adaptability problem as a finite-dimensional bilinear program and solve it heuristically. The authors also show that the $2$-adaptability problem is NP-hard even if $\bm{d}$, $\bm{T}$ and $\bm{W}$ are deterministic, and they develop necessary conditions for the $K$-adaptability problem~\eqref{eq:k_adaptability} to outperform the static robust problem (where all decisions are taken here-and-now).
The relationship between the $K$-adaptability problem~\eqref{eq:k_adaptability} and static robust optimization is further explored in~\cite{BGS11:finite_adaptability} for the special case where $\bm{T}$ and $\bm{W}$ are deterministic. The authors show that the gaps between both problems and the two-stage robust optimization problem~\eqref{eq:two_stage_ro} are intimately related to geometric properties of the uncertainty set $\Xi$. Finite-dimensional MILP reformulations for problem~\eqref{eq:k_adaptability} are developed in~\cite{HKW15:rip} under the additional assumption that both the here-and-now decisions $\bm{x}$ and the wait-and-see decisions $\bm{y}$ are binary. The authors show that both the size of the reformulations as well as their gaps to the two-stage robust optimization problem~\eqref{eq:two_stage_ro} depend on whether the uncertainty only affects the objective coefficients $\bm{d}$, or whether the constraint coefficients $\bm{T}$, $\bm{W}$ and $\bm{h}$ are uncertain as well. Finally, it is shown in \cite{BK16:min_max_max, BK16:min_max_max_MP} that for polynomial time solvable deterministic combinatorial optimization problems, the associated instances of problem~\eqref{eq:k_adaptability} without first-stage decisions $\bm{x}$ can also be solved in polynomial time if all of the following conditions hold: \emph{(i)} $\Xi$ is convex, \emph{(ii)} only the objective coefficients $\bm{d}$ are uncertain, and \emph{(iii)} $K > N_2$ policies are sought. This result has been extended to discrete uncertainty sets in~\cite{BK16:min_max_max_discrete}, in which case pseudo-polynomial solution algorithms can be developed.

In this paper, we expand the literature on the $K$-adaptability problem in two ways. From an \emph{analytical viewpoint}, we compare the two-stage robust optimization problem~\eqref{eq:two_stage_ro} with the $K$-adaptability problem~\eqref{eq:k_adaptability} in terms of their continuity, convexity and tractability. We also investigate when the approximation offered by the $K$-adaptability problem is tight, and under which conditions the two-stage robust optimization and $K$-adaptability problems reduce to single-stage problems.
From an \emph{algorithmic viewpoint}, we develop a branch-and-bound scheme for the $K$-adaptability problem that combines ideas from semi-infinite and disjunctive programming. We establish conditions for its asymptotic and finite time convergence; we show how it can be refined and integrated into state-of-the-art MILP solvers; and, we present a heuristic variant that can address large-scale instances.
In contrast to existing approaches, our algorithm can handle mixed continuous and discrete decisions in both stages as well as discrete uncertainty, and allows for modeling continuous second-stage decisions via a novel class of highly flexible piecewise affine decision rules.
Extensive numerical experiments on benchmark data from various application domains indicate that our algorithm is highly competitive with state-of-the-art solution schemes for problems~\eqref{eq:two_stage_ro} and~\eqref{eq:k_adaptability}.

The remainder of this paper develops as follows. Section~\ref{sec:prob_analysis} analyzes the geometry and tractability of problems~\eqref{eq:two_stage_ro} and~\eqref{eq:k_adaptability}, and it proposes a novel class of piecewise affine decision rules that can be modeled as an instance of problem~\eqref{eq:k_adaptability} with continuous second-stage decisions. Section~\ref{sec:bab} develops a branch-and-bound algorithm for the $K$-adaptability problem and analyzes its convergence. We present numerical results in Section~\ref{sec:num_results}, and we offer concluding remarks in Section~\ref{sec:conclusions}.

\textbf{Notation.} Vectors and matrices are printed in bold lowercase and uppercase letters, respectively, while scalars are printed in regular font. We use $\mathbf{e}_k$ to denote the $k^\text{th}$ unit basis vector and $\mathbf{e}$ to denote the vector whose components are all ones, respectively; their dimensions will be clear from the context. The $i^\text{th}$ row vector of a matrix $\bm{A}$ is denoted by $\bm{a}_i^\top$. For a logical expression $\mathcal{E}$, we define $\mathbb{I}[\mathcal{E}]$ as the indicator function which takes a value of $1$ is $\mathcal{E}$ is true and $0$ otherwise.

\section{Problem Analysis}\label{sec:prob_analysis}
In this section, we analyze the geometry and tractability of the two-stage robust optimization problem~\eqref{eq:two_stage_ro} and its associated $K$-adaptability problem~\eqref{eq:k_adaptability}.
Specifically, we characterize the continuity, convexity and tractability of both problems, as well as their relationship to the static robust optimization problem where all decisions are taken here-and-now. We also show how the $K$-adaptability problem with continuous second-stage decisions enables us to approximate the two-stage robust optimization problem~\eqref{eq:two_stage_ro} through highly flexible piecewise affine decision rules.

\begin{table}[!tb]
\begin{center}
\begin{footnotesize}
\begin{tabular}{c|c|c|c|c}
\multicolumn{2}{c}{} & \multicolumn{1}{c}{\textbf{Continuity}} & \multicolumn{1}{c}{\textbf{Convexity}} & \multicolumn{1}{c}{\textbf{Tractability}} \\ \hline
\textbf{First-stage} & problem~\eqref{eq:two_stage_ro} & if feasible & if $\mathcal{X}$, $\mathcal{Y}$ convex & if $\mathcal{X}$, $\mathcal{Y}$, $\Xi$ convex and OBJ \\
\textbf{problem} & problem~\eqref{eq:k_adaptability} & if feasible & typically not & if $\mathcal{X}$, $\mathcal{Y}$, $\Xi$ convex and OBJ \\ \hline
\textbf{Evaluation} & problem~\eqref{eq:two_stage_ro} & if OBJ & if $\Xi$ convex and OBJ & if $\Xi$, $\mathcal{Y}$ convex and OBJ \\
\textbf{problem} & problem~\eqref{eq:k_adaptability} & if OBJ or CON & if $\Xi$ convex and OBJ & if $\Xi$ convex and OBJ \\ \hline
\textbf{Second-stage} & problem~\eqref{eq:two_stage_ro} & if feasible & if $\mathcal{Y}$ convex & if $\mathcal{Y}$ convex \\
\textbf{problem} & problem~\eqref{eq:k_adaptability} &  if feasible & always & always \\ \hline \hline
\multicolumn{2}{c}{} & \multicolumn{1}{c}{Propositions \ref{prop:continuity_two_stage_ro}, \ref{prop:continuity_k_adaptability}} & \multicolumn{1}{c}{Propositions \ref{prop:convexity_two_stage_ro}, \ref{prop:convexity_k_adaptability}} & \multicolumn{1}{c}{Propositions \ref{prop:tractability_two_stage_ro}, \ref{prop:tractability_k_adaptability}}
\end{tabular}

~\\[3mm]

\begin{tabular}{cc}
\raisebox{2.7mm}{
\begin{minipage}[t]{8.5cm}
\begin{tabular}{c|c}
\multicolumn{1}{c}{} & \textbf{Reduction to static problem} \\ \hline
problem~\eqref{eq:two_stage_ro} & if $\Xi$ convex and OBJ \\
problem~\eqref{eq:k_adaptability} & if $\Xi$ convex, OBJ and $K > \min \{ N_2, N_p \}$ \\ \hline \hline
\multicolumn{1}{c}{} & \multicolumn{1}{c}{Propositions \ref{prop:static_two_stage_ro}, \ref{prop:static_k_adaptability}}
\end{tabular}
\end{minipage}}
&
\begin{minipage}[t]{6cm}
\begin{tabular}{c}
\textbf{Optimality of problem~\eqref{eq:k_adaptability}} \\ \hline
if $\mathcal{Y}$, $\Xi$ convex and OBJ\\
if $\Xi$ convex, OBJ and $K > \min \{ N_2, N_p \}$ \\
if $|\mathcal{Y}|$ finite and $K \geq |\mathcal{Y}|$ \\ \hline \hline
Proposition \ref{prop:optimality_k_adaptability}
\end{tabular}
\end{minipage}
\end{tabular}
\end{footnotesize}
\caption{Summary of theoretical results. Here, ``OBJ'' refers to instances where only the objective coefficients $\bm{d}$ are uncertain, while $\bm{T}$, $\bm{W}$ and $\bm{h}$ are constant. Similarly, ``CON'' refers to instances where only the constraint coefficients $\bm{T}$, $\bm{W}$ and right-hand sides $\bm{h}$ are uncertain, while $\bm{d}$ is constant.}
\label{table:prob_analysis_summary}
\end{center}
\end{table}

Table~\ref{table:prob_analysis_summary} summarizes our theoretical results. 
For the sake of brevity, we present the formal statements of these results and their proofs as supplementary material at the end of the paper.
In the table, the \emph{first-stage problem} refers to the overall problems~\eqref{eq:two_stage_ro} and~\eqref{eq:k_adaptability}, the \emph{evaluation problem} refers to the maximization over $\bm{\xi} \in \Xi$ for a fixed first-stage decision, and the \emph{second-stage problem} refers to the inner minimization over $\bm{y} \in \mathcal{Y}$ or $k \in \mathcal{K}$ for a fixed first-stage decision and a fixed realization of the uncertain problem parameters. The table reveals that despite significant differences in their formulations, the problems~\eqref{eq:two_stage_ro} and~\eqref{eq:k_adaptability} behave very similarly. The most significant difference is caused by the replacement of the optimization over the second-stage decisions $\bm{y} \in \mathcal{Y}$ in problem~\eqref{eq:two_stage_ro} with the selection of a candidate policy $k \in \mathcal{K}$ in problem~\eqref{eq:k_adaptability}. This ensures that the second-stage problem in~\eqref{eq:k_adaptability} is always continuous, convex and tractable, whereas the first-stage problem in~\eqref{eq:k_adaptability} fails to be convex even if $\mathcal{X}$ and $\mathcal{Y}$ are convex. Moreover, in contrast to problem~\eqref{eq:two_stage_ro}, the evaluation problem in~\eqref{eq:k_adaptability} remains continuous as long as either the objective function or the constraints are unaffected by uncertainty. For general problem instances, however, neither of the two evaluation problems is continuous. As we will see in Section~\ref{sec:bab_convergence}, this directly impacts the convergence of our branch-and-bound algorithm, which only takes place asymptotically in general. Note that the convexity of the problems~\eqref{eq:two_stage_ro} and~\eqref{eq:k_adaptability} does not depend on the shape of the uncertainty set $\Xi$.

\subsection{Incorporating Decision Rules in the $K$-Adaptability Problem}\label{sec:prob_analysis_affine_decision_rules}

Although the $K$-adaptability problem~\eqref{eq:k_adaptability} selects the best candidate policy $\bm{y}_k$ in response to the observed parameter realization $\bm{\xi} \in \Xi$, the policies $\bm{y}_1, \ldots, \bm{y}_K$---once selected in the first stage---no longer depend on $\bm{\xi}$. This lack of dependence on the uncertain problem parameters can lead to overly conservative approximations of the two-stage robust optimization problem~\eqref{eq:two_stage_ro} when the second-stage decisions are continuous. In this section, we show how the $K$-adaptability problem~\eqref{eq:k_adaptability} can be used to generalize affine decision rules, which are commonly used to approximate continuous instances of the two-stage robust optimization problem~\eqref{eq:two_stage_ro}.
We note that existing schemes, such as~\cite{BK16:min_max_max_MP,HKW15:rip}, cannot be used for this purpose as they require the wait-and-see decisions $\bm{y}$ to be binary.

Throughout this section, we assume that problem~\eqref{eq:two_stage_ro} has purely continuous second-stage decisions (that is, $\mathcal{Y}$ is LP representable), a deterministic objective function (that is, $\bm{d} (\bm{\xi}) = \bm{d}$ for all $\bm{\xi} \in \Xi$) and fixed recourse (that is, $\bm{W} (\bm{\xi}) = \bm{W}$ for all $\bm{\xi} \in \Xi$). The assumption of continuous second-stage decisions allows us to assume, without loss of generality, that $\mathcal{Y} = \mathbb{R}^{N_2}$ as any potential restrictions can be absorbed in the second-stage constraints.

The affine decision rule approximation to the two-stage robust optimization problem~\eqref{eq:two_stage_ro} is
\begin{equation*}
\inf_{\substack{\bm{x} \in \mathcal{X}, \\ \bm{y} : \Xi \overset{1}{\mapsto} \mathbb{R}^{N_2}}} \; \left\{ \sup_{\bm{\xi} \in \Xi}
\left\{ \bm{c}^\top \bm{x} + \bm{d}^\top \bm{y} (\bm{\xi}) \right\} \, : \, \bm{T} (\bm{\xi}) \bm{x} + \bm{W} \bm{y} (\bm{\xi}) \leq \bm{h} (\bm{\xi}) \;\;\; \forall \bm{\xi} \in \Xi \right\},
\end{equation*}
where $\bm{y} : \Xi \overset{1}{\mapsto} \mathbb{R}^{N_2}$ indicates that $\bm{y} (\bm{\xi}) = \bm{y}^{0} + \bm{Y} \bm{\xi}$ for some $\bm{y}^{0} \in \mathbb{R}^{N_2}$ and $\bm{Y} \in \mathbb{R}^{N_2 \times N_p}$, see~\cite{BTGGN04:adjustable}. This problem provides a conservative approximation to the two-stage robust optimization problem~\eqref{eq:two_stage_ro} since we replace the space of all (possibly non-convex and discontinuous) second-stage policies $\bm{y} : \Xi \mapsto \mathbb{R}^{N_2}$ with the subspace of all affine second-stage policies $\bm{y} : \Xi \overset{1}{\mapsto} \mathbb{R}^{N_2}$. In a similar spirit, we define the subspace of all piecewise affine decision rules $\bm{y} : \Xi \overset{K}{\mapsto} \mathbb{R}^{N_2}$ with $K$ pieces as
\begin{equation*}
\bm{y} : \Xi \overset{K}{\mapsto} \mathbb{R}^{N_2}
\quad \Longleftrightarrow \quad
\left[
\begin{array}{l}
\exists (\bm{y}_k^0, \bm{Y}_k) \in \mathbb{R}^{N_2} \times \mathbb{R}^{N_2 \times N_p}, \, k = 1, \ldots, K, \text{ such that} \\
\forall \bm{\xi} \in \Xi, \; \exists k \in \{ 1, \ldots, K \} \, : \, \bm{y} (\bm{\xi}) = \bm{y}_k^0 + \bm{Y}_k \bm{\xi}
\end{array}
\right].
\end{equation*}
Note that our earlier definition of $\bm{y} : \Xi \overset{1}{\mapsto} \mathbb{R}^{N_2}$ is identical to our definition of $\bm{y} : \Xi \overset{K}{\mapsto} \mathbb{R}^{N_2}$ if $K = 1$. For $K > 1$, the decision rules $\bm{y} : \Xi \overset{K}{\mapsto} \mathbb{R}^{N_2}$ may be non-convex and discontinuous, and the regions where $\bm{y}$ is affine may be non-closed and non-convex. 
We highlight that the points of nonlinearity are determined by the optimization problem. This is in contrast to many existing solution schemes for piecewise affine decision rules, such as \cite{Angele:BDRs, CZ09:extended_affinely_adj, GWK10:generalized_drs, GS10:dro_tractable_approximations}, where these points are specified ad hoc by the decision-maker.

\begin{obs}
The piecewise affine decision rule problem with fixed recourse
\begin{equation}\label{eq:piecewise_affine_decision_rule}
\inf_{\substack{\bm{x} \in \mathcal{X}, \\ \bm{y} : \Xi \overset{K}{\mapsto} \mathbb{R}^{N_2}}} \; \left\{ \sup_{\bm{\xi} \in \Xi}
\left\{ \bm{c}^\top \bm{x} + \bm{d}^\top \bm{y} (\bm{\xi}) \right\} \, : \, \bm{T} (\bm{\xi}) \bm{x} + \bm{W} \bm{y} (\bm{\xi}) \leq \bm{h} (\bm{\xi})  \;\;\; \forall \bm{\xi} \in \Xi \right\}
\end{equation}
is equivalent to the $K$-adaptability problem with random recourse
\begin{equation}\label{eq:piecewise_affine_k_adaptability}
\inf_{\substack{\bm{x} \in \mathcal{X}, \\ (\bm{y}^0, \bm{Y}, \bm{z}) \in \hat{\mathcal{Y}}^K}} \; \sup_{\bm{\xi} \in \Xi} \; \inf_{k \in \mathcal{K}}
\left\{ \bm{c}^\top \bm{x} + \bm{d}^\top \bm{y}_k^0 + \bm{\hat{d}} (\bm{\xi})^\top \bm{z}_{k1} \, : \, \bm{T} (\bm{\xi}) \bm{x} + \bm{W} \bm{y}_k^0 + \bm{\hat{W}} (\bm{\xi}) \bm{z}_{k2} \leq \bm{h} (\bm{\xi}) \right\},
\end{equation}
where $\hat{\mathcal{Y}} = \Big\{(\bm{y}^0, \bm{Y}, \bm{z}) \in \mathbb{R}^{N_2} \times \mathbb{R}^{N_2 \times N_p} \times (\mathbb{R}^{N_p} \times \mathbb{R}^{N_p L}) \,:\, \bm{z} = (\bm{z}_{1}, \bm{z}_{2}) \text{ with } \bm{z}_1 = \bm{Y}^\top \bm{d} \text{ and } \bm{z}_2 = [\bm{w}_1^\top \bm{Y} \; \ldots \; \bm{w}_L^\top \bm{Y} ]^\top \Big\}$, $\bm{\hat{d}} (\bm{\xi}) = \bm{\xi}$ and $\bm{\hat{W}} (\bm{\xi}) = \text{\emph{diag}} \big( \bm{\xi}^\top, \, \ldots, \, \bm{\xi}^\top \big) \in \mathbb{R}^{L \times N_p L}$.
\end{obs}

\begin{proof}
Problem~\eqref{eq:piecewise_affine_k_adaptability} is infeasible if and only if for every $\bm{x} \in \mathcal{X}$ and $(\bm{y}^0, \bm{Y}, \bm{z}) \in \hat{\mathcal{Y}}^K$ there is a $\bm{\xi} \in \Xi$ such that $\bm{T} (\bm{\xi}) \bm{x} + \bm{W} \bm{y}_k^0 + \bm{\hat{W}} (\bm{\xi}) \bm{z}_{k2} \not\leq \bm{h} (\bm{\xi})$ for all $k = 1,\ldots, K$, which in turn is the case if and only if for every $\bm{x} \in \mathcal{X}$ and $(\bm{y}^0_k, \bm{Y}_k) \in \mathbb{R}^{N_2} \times \mathbb{R}^{N_2 \times N_p}$, $k = 1,\ldots, K$ there is a $\bm{\xi} \in \Xi$ such that $\bm{T} (\bm{\xi}) \bm{x} + \bm{W} \bm{y}_k^0 + \bm{W} \bm{Y}_k \bm{\xi} \not\leq \bm{h} (\bm{\xi})$ for all $k = 1,\ldots, K$; that is, if and only if problem~\eqref{eq:piecewise_affine_decision_rule} is infeasible. We thus assume that both~\eqref{eq:piecewise_affine_decision_rule} and~\eqref{eq:piecewise_affine_k_adaptability} are feasible. In this case, we verify that every feasible solution $(\bm{x}, \bm{y}^0, \bm{Y}, \bm{z})$ to problem~\eqref{eq:piecewise_affine_k_adaptability} gives rise to a feasible solution $(\bm{x}, \bm{y})$, where $\bm{y} (\bm{\xi}) = \bm{y}_{k(\bm{\xi})}^0 + \bm{Y}_{k(\bm{\xi})} \bm{\xi}$ and $k(\bm{\xi})$ is any element of $\mathop{\arg\min}\limits_{k \in \mathcal{K}} \Big\{ \bm{c}^\top \bm{x} + \bm{d}^\top \bm{y}_k^0 + \bm{\hat{d}} (\bm{\xi})^\top \bm{z}_{k1} \, : \, \bm{T} (\bm{\xi}) \bm{x} + \bm{W} \bm{y}_k^0 + \bm{\hat{W}} (\bm{\xi}) \bm{z}_{k2} \leq \bm{h} (\bm{\xi}) \Big\}$,  in problem~\eqref{eq:piecewise_affine_decision_rule} that attains the same worst-case objective value. Similarly, every \emph{optimal solution} $(\bm{x}, \bm{y})$ to problem~\eqref{eq:piecewise_affine_decision_rule} gives rise to an optimal solution $(\bm{x}, \bm{y}^0, \bm{Y}, \bm{z})$, where $\bm{z}_k = (\bm{z}_{k1}, \bm{z}_{k2})$ with $\bm{z}_{k1} = \bm{Y}_k^\top \bm{d}$ and $\bm{z}_{k2} = [\bm{w}_1^\top \bm{Y}_k \; \ldots \; \bm{w}_L^\top \bm{Y}_k ]^\top$, $k = 1, \ldots, K$, in problem~\eqref{eq:piecewise_affine_k_adaptability}. Hence,~\eqref{eq:piecewise_affine_decision_rule} and~\eqref{eq:piecewise_affine_k_adaptability} share the same optimal value and the same sets of optimal solutions. 
\end{proof}

We close with an example that illustrates the benefits of piecewise affine decision rules.

\begin{ex}\label{ex:suboptimality_of_LDR}
Consider the following instance of the two-stage robust optimization problem~\eqref{eq:two_stage_ro}, which has been proposed in~\cite[Section~5.2]{Gorissen2013:robust_sums_of_maxima}:
\begin{equation*}
\sup_{\bm{\xi} \in [-1, 1]^2} \; \inf_{\bm{y} \in \mathbb{R}^4_{+}}
\left\{ \bm{e}^\top \bm{y} \, : \, y_1 \geq \xi_1 + \xi_2, \; y_2 \geq \xi_1 - \xi_2, \; y_3 \geq -\xi_1 + \xi_2, \; y_4 \geq -\xi_1 - \xi_2 \right\}.
\end{equation*}
The \emph{optimal second-stage policy} is $\bm{y}^\star (\bm{\xi}) = ([\xi_1 + \xi_2]_+, \, [\xi_1 - \xi_2]_+, \, [-\xi_1 + \xi_2]_+, \, [-\xi_1 - \xi_2]_+)$, where $[ \cdot ]_+ = \max \{ \cdot, 0 \}$, and it results in the optimal second-stage value function $Q^\star (\bm{\xi}) = [\xi_1 + \xi_2]_+ + [\xi_1 - \xi_2]_+ + [-\xi_1 + \xi_2]_+ + [-\xi_1 - \xi_2]_+$ with a worst-case objective value of $2$, see Figure~\ref{figure:optimal_value_function}. The \emph{best affine decision rule} $\bm{y}^1 : \Xi \overset{1}{\mapsto} \mathbb{R}^4_+$ is $\bm{y}^1 (\bm{\xi}) = \left(1 + \xi_2, \, 1 + \xi_1, \, 1 - \xi_1, \, 1 - \xi_2\right)$, and it results in the constant second-stage value function $Q^1 (\bm{\xi}) = 4$. The \emph{best $2$-adaptable affine decision rule} $\bm{y}^2 : \Xi \overset{2}{\mapsto} \mathbb{R}^4_+$, on the other hand, is given by
\begin{equation*}
\bm{y}^2 (\bm{\xi}) = \begin{dcases*}
\left(0,\, 1 + \xi_1,\, 1 + \xi_2,\, -\xi_1 - \xi_2\right) & if $\xi_1 + \xi_2 \leq 0$ \\
\left(\xi_1 + \xi_2,\, 1 - \xi_2,\, 1 - \xi_1,\, 0\right) & otherwise,
\end{dcases*}
\end{equation*}
and it results in the constant second-stage value function $Q^2(\bm{\xi}) = 2$. Thus, $2$-adaptable affine decision rules are optimal in this example. Figure~\ref{figure:suboptimality_of_LDR_y3} illustrates the optimal value, the affine approximation and the $2$-adaptable affine approximation of the decision variable $y_3$.

\begin{figure}[!tb]
\centering
\includegraphics[width=0.3\textwidth]{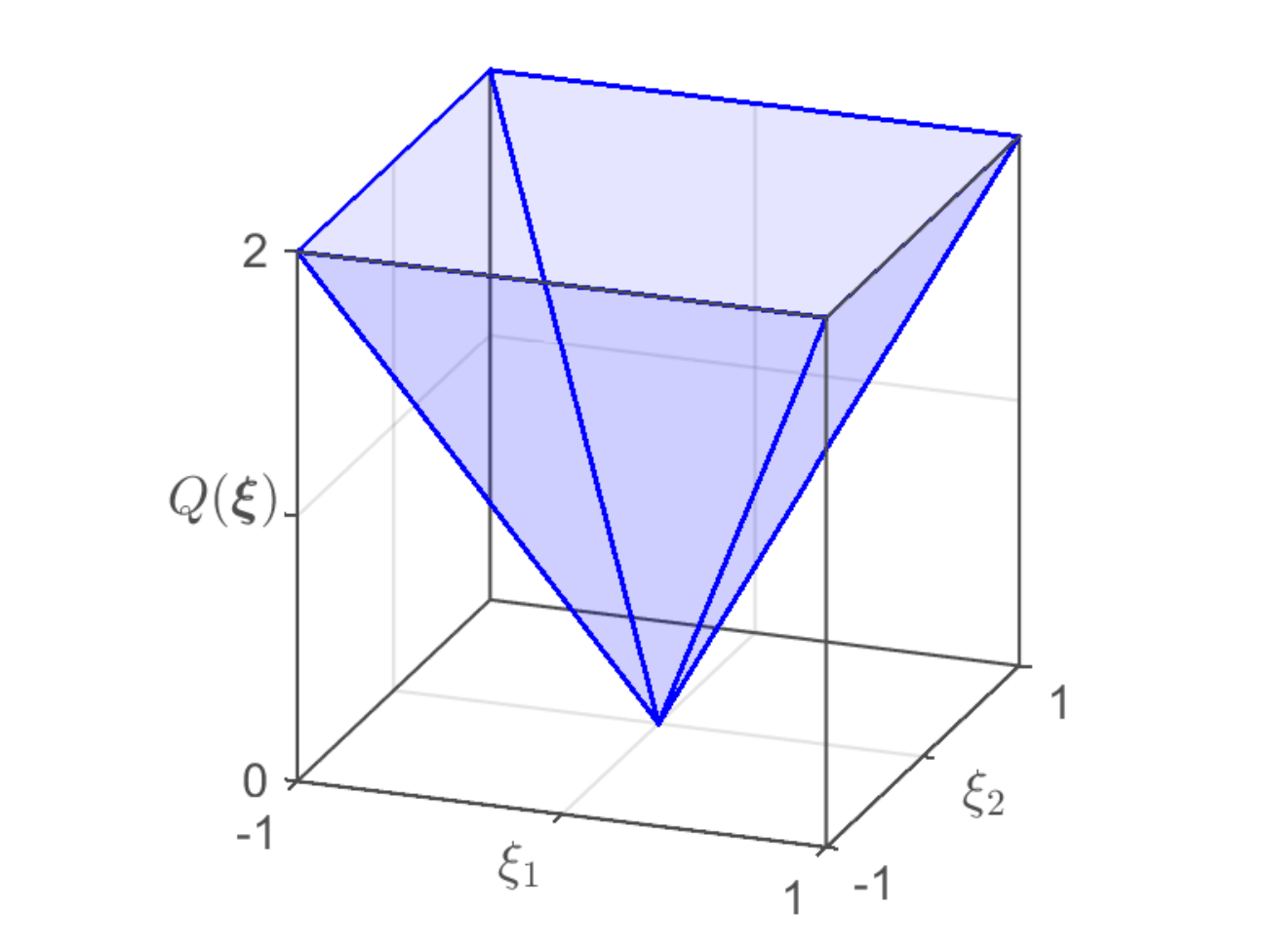}
\caption{The optimal second-stage value function in Example~\ref{ex:suboptimality_of_LDR} is given by the first-order cone $Q^\star (\bm{\xi}) = [\xi_1 + \xi_2]_+ + [\xi_1 - \xi_2]_+ + [-\xi_1 + \xi_2]_+ + [-\xi_1 - \xi_2]_+$.}
\label{figure:optimal_value_function}
\end{figure}

\begin{figure}[!tb]
\captionsetup[subfigure]{belowskip=-24pt}
\captionsetup[subfigure]{labelformat=empty}
\centering
\begin{subfigure}[b]{0.3\textwidth}
\includegraphics[width=\textwidth]{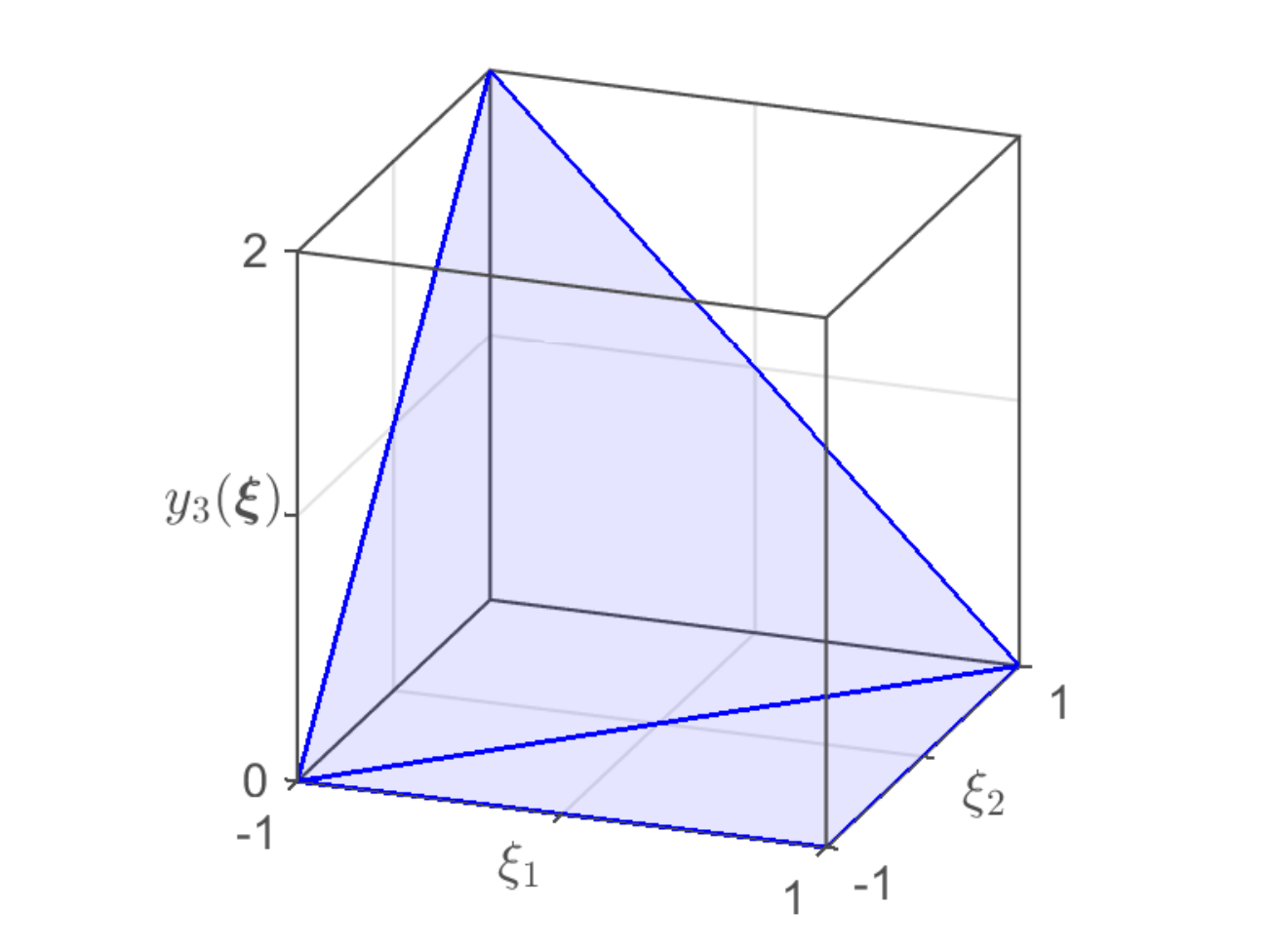}
\caption{}
\label{figure:suboptimality_of_LDR_y3_optimal}
\end{subfigure}~%
\begin{subfigure}[b]{0.3\textwidth}
\includegraphics[width=\textwidth]{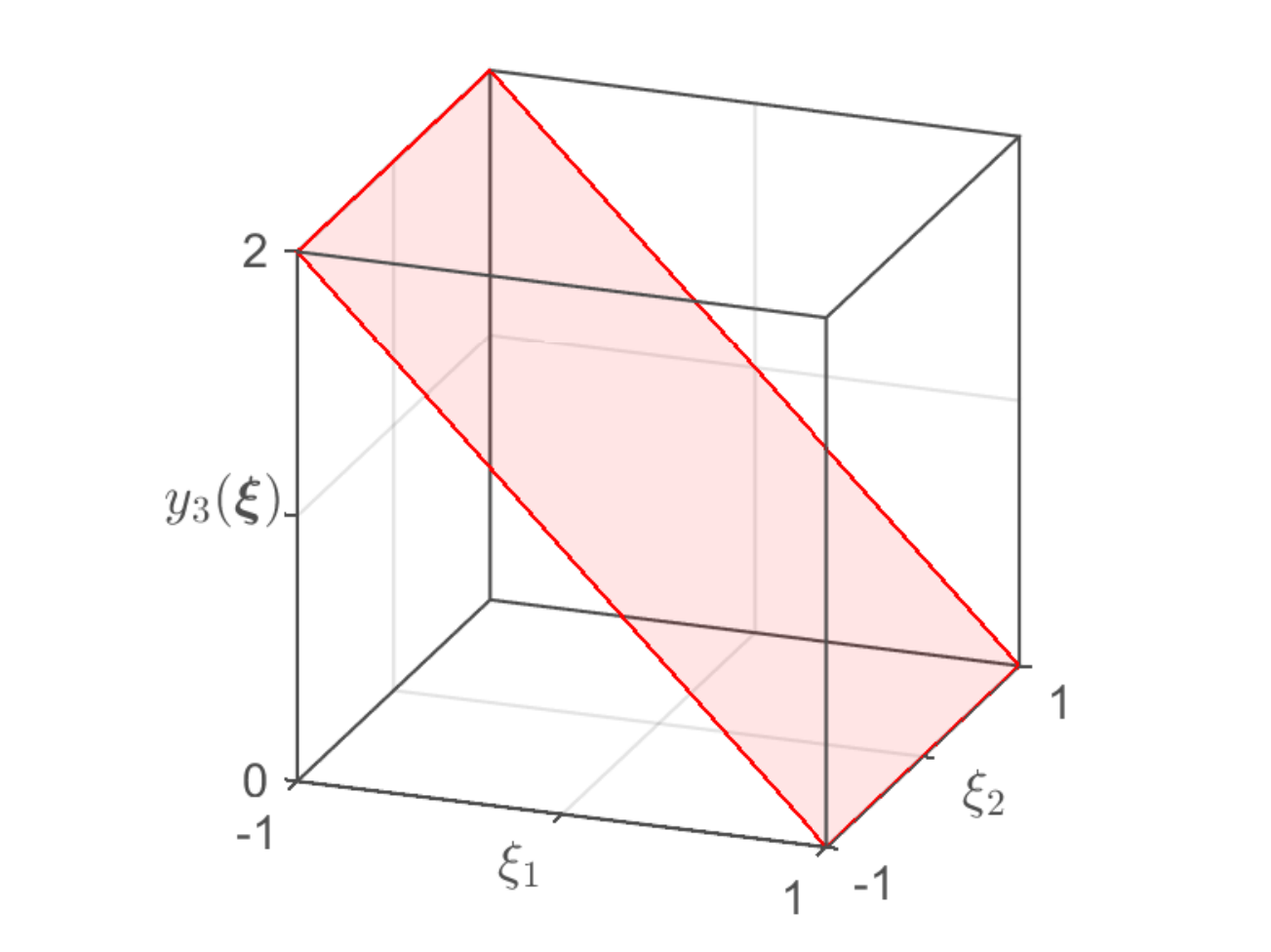}
\caption{}
\label{figure:suboptimality_of_LDR_y3_LDR}
\end{subfigure}~%
\begin{subfigure}[b]{0.3\textwidth}
\includegraphics[width=\textwidth]{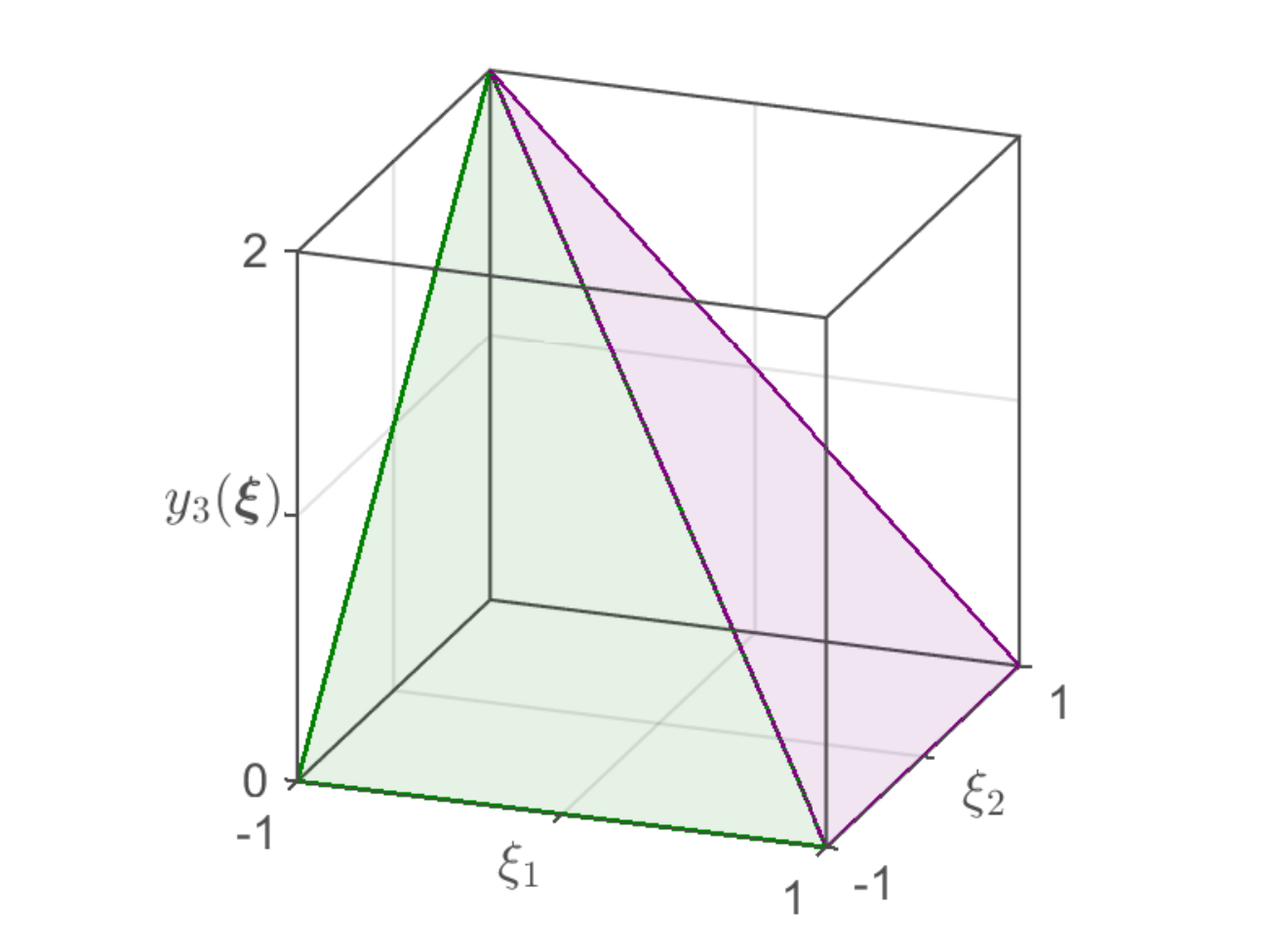}
\caption{}
\label{figure:suboptimality_of_LDR_y3_KAdaptable}
\end{subfigure}
\caption{Plot of the optimal second-stage policy $y_3 (\bm{\xi})$ in the two-stage robust optimization problem (left), the affine decision rule problem (middle) and the $2$-adaptable affine decision rule problem (right).}
\label{figure:suboptimality_of_LDR_y3}
\end{figure}
\end{ex}

The piecewise affine decision rules presented here can be readily combined with discrete second-stage decisions. For the sake of brevity, we omit the details of this straightforward extension.

\section{Solution Scheme}\label{sec:bab}
Our solution scheme for the $K$-adaptability problem~\eqref{eq:k_adaptability} is based on a reformulation as a semi-infinite disjunctive program which we present next.
\begin{obs}\label{obs:semiinfinite_disjunctive_reformulation}
The $K$-adaptability problem~\eqref{eq:k_adaptability} is equivalent to
\begin{equation}\label{eq:k_adaptability_semiinfinite}
\begin{array}{l@{\quad}l}
\text{\emph{minimize}} & \theta \\
\text{\emph{subject to}} & \theta \in \mathbb{R}, \;\; \bm{x} \in \mathcal{X}, \;\; \bm{y} \in \mathcal{Y}^K \\
& \displaystyle \bigvee_{k \in \mathcal{K}} \left[
\begin{array}{l}
\bm{c}^\top \bm{x} + \bm{d} (\bm{\xi})^\top \bm{y}_k \leq \theta \\
\bm{T} (\bm{\xi}) \bm{x} + \bm{W} (\bm{\xi}) \bm{y}_k \leq \bm{h} (\bm{\xi})
\end{array}
\right] \quad \forall\, \bm{\xi} \in \Xi.
\end{array}
\end{equation}
Moreover, if some of the constraints in problem~\eqref{eq:k_adaptability_semiinfinite} are deterministic, i.e., they do not depend on $\bm{\xi}$, then they can be moved outside the disjunction and instead be enforced for all $k \in \mathcal{K}$.
\end{obs}

In the following, we stipulate that the optimal value of~\eqref{eq:k_adaptability_semiinfinite} is $+ \infty$ whenever it is infeasible.

\begin{proof}[Proof of Observation~\ref{obs:semiinfinite_disjunctive_reformulation}]
Problem~\eqref{eq:k_adaptability} is infeasible if and only if (iff) for every $\bm{x} \in \mathcal{X}$ and $\bm{y} \in \mathcal{Y}^K$ there is a $\bm{\xi} \in \Xi$ such that $\bm{T} (\bm{\xi}) \bm{x} + \bm{W} (\bm{\xi}) \bm{y}_k \not\leq \bm{h} (\bm{\xi})$ for all $k \in \mathcal{K}$, which in turn is the case iff for every $\bm{x} \in \mathcal{X}$ and $\bm{y} \in \mathcal{Y}^K$, the disjunction in~\eqref{eq:k_adaptability_semiinfinite} is violated for at least one $\bm{\xi} \in \Xi$; that is, iff problem~\eqref{eq:k_adaptability_semiinfinite} is infeasible. We thus assume that both~\eqref{eq:k_adaptability} and~\eqref{eq:k_adaptability_semiinfinite} are feasible. In this case, one readily verifies that every feasible solution $(\bm{x}, \bm{y})$ to problem~\eqref{eq:k_adaptability} gives rise to a feasible solution $(\theta, \bm{x}, \bm{y})$, where $\theta = \sup\limits_{\bm{\xi} \in \Xi}\; \inf\limits_{k \in \mathcal{K}} \big\{ \bm{c}^\top \bm{x} + \bm{d} (\bm{\xi})^\top \bm{y}_k : \bm{T} (\bm{\xi}) \bm{x} + \bm{W} (\bm{\xi}) \bm{y}_k \leq \bm{h} (\bm{\xi}) \big\}$, in problem~\eqref{eq:k_adaptability_semiinfinite} with the same objective value. Likewise, any \emph{optimal} solution $(\theta, \bm{x}, \bm{y})$ to problem~\eqref{eq:k_adaptability_semiinfinite} corresponds to an optimal solution $(\bm{x}, \bm{y})$ in problem~\eqref{eq:k_adaptability}. Hence,~\eqref{eq:k_adaptability} and~\eqref{eq:k_adaptability_semiinfinite} share the same optimal value and the same sets of optimal solutions. 

We now claim that if $\bm{t}_l (\bm{\xi})^\top = \bm{t}_l^\top$, $\bm{w}_l (\bm{\xi})^\top = \bm{w}_l^\top$ and $h_l (\bm{\xi}) = h_l$ for all $l \in \{ 1, \ldots, L \} \setminus \mathcal{L}$, where $\mathcal{L} \subseteq \{ 1, \ldots, L \}$, then problem~\eqref{eq:k_adaptability_semiinfinite} is equivalent to
\begin{equation}\label{eq:k_adaptability_semiinfinite_reform}
\begin{array}{l@{\quad}l@{\quad}l}
\text{minimize} & \theta \\
\text{subject to} & \theta \in \mathbb{R}, \;\; \bm{x} \in \mathcal{X}, \;\; \bm{y} \in \mathcal{Y}^K \\
& \bm{t}_l^\top \bm{x} + \bm{w}_l^\top \bm{y}_k \leq h_l & \forall\, l \in \{ 1, \ldots, L \} \setminus \mathcal{L}, \; \forall\, k \in \mathcal{K} \\
& \displaystyle \bigvee_{k \in \mathcal{K}} \left[
\begin{array}{l}
\bm{c}^\top \bm{x} + \bm{d} (\bm{\xi})^\top \bm{y}_k \leq \theta \\
\bm{t}_l (\bm{\xi})^\top \bm{x} + \bm{w}_l (\bm{\xi})^\top \bm{y}_k \leq h_l (\bm{\xi})
\quad \forall\, l \in \mathcal{L}
\end{array}
\right] & \forall\, \bm{\xi} \in \Xi.
\end{array}
\tag{$\ref{eq:k_adaptability_semiinfinite}'$}
\end{equation}
By construction, every feasible solution $(\theta, \bm{x}, \bm{y})$ to problem~\eqref{eq:k_adaptability_semiinfinite_reform} is feasible in problem~\eqref{eq:k_adaptability_semiinfinite}. Conversely, fix any feasible solution $(\theta, \bm{x}, \bm{y})$ to problem~\eqref{eq:k_adaptability_semiinfinite} and assume that $\bm{t}_l^\top \bm{x} + \bm{w}_l^\top \bm{y}_k > h_l$ for some $k \in \mathcal{K}$ and $l \in \{ 1, \ldots, L \} \setminus \mathcal{L}$. In that case, the $k^\text{th}$ disjunct in~\eqref{eq:k_adaptability_semiinfinite} is violated for \emph{every} realization $\bm{\xi} \in \Xi$. We can therefore replace $\bm{y}_k$ with a different candidate policy $\bm{y}_{k'}$ that satisfies $\bm{t}_l^\top \bm{x} + \bm{w}_l^\top \bm{y}_{k'} \leq h_l$ for all $l \in \{ 1, \ldots, L \} \setminus \mathcal{L}$ without sacrificing feasibility. (Note that such a candidate policy $\bm{y}_{k'}$ exists since $(\theta, \bm{x}, \bm{y})$ is assumed to be feasible in~\eqref{eq:k_adaptability_semiinfinite}.) Replacing any infeasible policy in this way results in a solution that is feasible in problem~\eqref{eq:k_adaptability_semiinfinite_reform}.
\end{proof}

Problem~\eqref{eq:k_adaptability_semiinfinite} cannot be solved directly as it contains infinitely many disjunctive constraints. Instead, our solution scheme iteratively solves a sequence of (increasingly tighter) relaxations of this problem that are obtained by enforcing the disjunctive constraints over finite subsets of $\Xi$. Whenever the solution of such a relaxation violates the disjunction for some realization $\bm{\xi} \in \Xi$, we create $K$ subproblems that enforce the disjunction associated with $\bm{\xi}$ to be satisfied by the $k^\text{th}$ disjunct, $k = 1, \ldots, K$. Our solution scheme is reminiscent of discretization methods employed in \emph{semi-infinite programming}, which iteratively replace an infinite set of constraints with finite subsets and solve the resulting discretized problems. Indeed, our scheme can be interpreted as a generalization of Kelley's cutting-plane method \cite{Blankenship1976:semi_infinite, Kelley1960:cutting_plane} applied to semi-infinite disjunctive programs. In the special case where $K=1$, our method reduces to the cutting-plane method for (static) robust optimization problems proposed in \cite{MB09:robust_cutting_plane}.

In the remainder of this section, we describe our basic branch-and-bound scheme (Section~\ref{sec:bab_algorithm}), we study its convergence (Section~\ref{sec:bab_convergence}), we discuss algorithmic variants to the basic scheme that can enhance its numerical performance (Section~\ref{sec:bab_extensions}), and we present a heuristic variant that can address problems of larger scale (Section~\ref{sec:bab_heuristic}).

\subsection{Branch-and-Bound Algorithm}\label{sec:bab_algorithm}

Our solution scheme iteratively solves a sequence of scenario-based $K$-adaptability problems and separation problems. We define both problems first, and then we describe the overall algorithm.

\paragraph{The Scenario-Based $K$-Adaptability Problem.}
For a collection $\Xi_1, \ldots, \Xi_K$ of finite subsets of the uncertainty set $\Xi$, we define the \emph{scenario-based $K$-adaptability problem} as
\begin{equation}\label{eq:master_k_adaptability}
\begin{array}{r@{}l@{\quad}l}
\mathcal{M} (\Xi_1, \ldots, \Xi_K) \; = \;\;\;
& \text{minimize} & \theta \\
& \text{subject to} & \theta \in \mathbb{R}, \;\; \bm{x} \in \mathcal{X}, \;\; \bm{y} \in \mathcal{Y}^K \\
&& \left.\begin{array}{l}
\mspace{-11mu} \bm{c}^\top \bm{x} + \bm{d} (\bm{\xi})^\top \bm{y}_k \leq \theta \\
\mspace{-11mu} \bm{T} (\bm{\xi}) \bm{x} + \bm{W} (\bm{\xi}) \bm{y}_k \leq \bm{h} (\bm{\xi})
\end{array}
\right\} \;\; \forall\, \bm{\xi} \in \Xi_k,\; \forall\, k \in \mathcal{K}.
\end{array}
\end{equation}

If $\mathcal{X}$ and $\mathcal{Y}$ are convex, problem~\eqref{eq:master_k_adaptability} is an LP; otherwise, it is an MILP.
The problem is closely related to a relaxation of the semi-infinite disjunctive program~\eqref{eq:k_adaptability_semiinfinite} that enforces the disjunction only over the realizations $\bm{\xi} \in \bigcup_{k \in \mathcal{K}} \Xi_k$. More precisely, problem~\eqref{eq:master_k_adaptability} can be interpreted as a restriction of that relaxation which requires the $k^\text{th}$ candidate policy $\bm{y}_k$ to be worst-case optimal for all realizations $\bm{\xi} \in \Xi_k$, $k \in \mathcal{K}$. We obtain an optimal solution $(\theta, \bm{x}, \bm{y}_k)$ to the relaxed semi-infinite disjunctive program by solving $\mathcal{M} (\Xi_1, \ldots, \Xi_K)$ for all partitions $(\Xi_1, \ldots, \Xi_K)$ of $\bigcup_{k \in \mathcal{K}} \Xi_k$ and reporting the optimal solution $(\theta, \bm{x}, \bm{y}_k)$ of the problem $\mathcal{M} (\Xi_1, \ldots, \Xi_K)$ with the smallest objective value.

If $\Xi_k = \emptyset$ for all $k \in \mathcal{K}$, then problem~\eqref{eq:master_k_adaptability} is unbounded, and we stipulate that its optimal value is $-\infty$ and that its optimal value is attained by any solution $(-\infty, \bm{x}, \bm{y})$ with $(\bm{x}, \bm{y}) \in \mathcal{X} \times \mathcal{Y}^K$. Otherwise, if problem~\eqref{eq:master_k_adaptability} is infeasible for $\Xi_1,\ldots,\Xi_K$, then we define its optimal value to be $+ \infty$. In all other cases, the optimal value of problem~\eqref{eq:master_k_adaptability} is finite and it is attained by an optimal solution $(\theta, \bm{x}, \bm{y})$ since $\mathcal{X}$ and $\mathcal{Y}$ are compact.

\begin{rem}[Decomposability]
For $K$-adaptability problems without first-stage decisions $\bm{x}$, problem~\eqref{eq:master_k_adaptability} decomposes into $K$ scenario-based static robust optimization problems that are only coupled through the constraints referencing the epigraph variable $\theta$. In this case, we can recover an optimal solution to problem~\eqref{eq:master_k_adaptability} by solving each of the $K$ static problems individually and identifying the optimal $\theta$ as the maximum of their optimal values.
\end{rem}

\paragraph{The Separation Problem.}
For a feasible solution $(\theta, \bm{x}, \bm{y})$ to the scenario-based $K$-adaptability problem~\eqref{eq:master_k_adaptability}, we define the \emph{separation problem} as
\begin{equation}\label{eq:separation_explicit}
\begin{aligned}
\mathcal{S}(\theta, \bm{x}, \bm{y}) &= \max_{\bm{\xi} \in \Xi} \; S (\theta, \bm{x}, \bm{y}, \bm{\xi}), \; \text{where} \\
S (\theta, \bm{x}, \bm{y}, \bm{\xi}) &= \min_{k \in \mathcal{K}} \; \max \left\{ \bm{c}^\top \bm{x} + \bm{d} (\bm{\xi})^\top \bm{y}_k - \theta, \; \max_{l \in \{ 1, \ldots, L \}} \; \left\{ \bm{t}_l (\bm{\xi})^\top \bm{x} + \bm{w}_l (\bm{\xi})^\top \bm{y}_k - h_l (\bm{\xi}) \right\} \right\},
\end{aligned}
\end{equation}
for $\mathcal{S} : \mathbb{R} \cup \{- \infty\}\times \mathcal{X} \times \mathcal{Y}^K \mapsto \mathbb{R} \cup \{ + \infty \}$ and $S : \mathbb{R} \cup \{- \infty\}\times \mathcal{X} \times \mathcal{Y}^K \times \Xi \mapsto \mathbb{R} \cup \{ + \infty \}$. Whenever it is positive, the innermost maximum in the definition of $S (\theta, \bm{x}, \bm{y}, \bm{\xi})$ records the maximum constraint violation of the candidate policy $\bm{y}_k$ under the parameter realization $\bm{\xi} \in \Xi$. Likewise, the quantity $\bm{c}^\top \bm{x} + \bm{d} (\bm{\xi})^\top \bm{y}_k - \theta$ denotes the excess of the objective value of $\bm{y}_k$ under the realization $\bm{\xi}$ over the current candidate value of the worst-case objective, $\theta$. Thus, $S (\theta, \bm{x}, \bm{y}, \bm{\xi})$ is strictly positive if and only if every candidate policy $\bm{y}_k$ either is infeasible or results in an objective value greater than $\theta$ under the realization $\bm{\xi} \in \Xi$. Whenever $\theta$ is finite, the separation problem is feasible and bounded, and it has an optimal solution since $\Xi$ is nonempty and compact. Otherwise, we have $\mathcal{S}(\theta, \bm{x}, \bm{y}) = + \infty$, and the optimal value is attained by any $\bm{\xi} \in \Xi$.

\begin{obs}\label{obs:milp_reformulation_sep_problem}
The separation problem~\eqref{eq:separation_explicit} is equivalent to the MILP
\begin{equation}\label{eq:separation}
\begin{array}{l@{\quad}l}
\text{\emph{maximize}} & \zeta \\
\text{\emph{subject to}} & \zeta \in \mathbb{R}, \;\; \bm{\xi} \in \Xi, \;\; z_{kl} \in \{0,1\}, \, (k,l) \in \mathcal{K} \times \{ 0, 1, \ldots, L \} \\
& \left.\begin{array}{l}
\mspace{-11mu} \displaystyle \sum_{l = 0}^{L} z_{kl} = 1 \\
\mspace{-11mu} \displaystyle z_{k0} = 1 \;\; \Rightarrow \;\; \zeta \leq \bm{c}^\top \bm{x} + \bm{d} (\bm{\xi})^\top \bm{y}_k - \theta \\
\mspace{-11mu} \displaystyle z_{kl} = 1 \;\;\, \Rightarrow \;\; \zeta \leq \bm{t}_l (\bm{\xi})^\top \bm{x} + \bm{w}_l (\bm{\xi})^\top \bm{y}_k - h_l (\bm{\xi}) \quad \forall\, l \in \{1, \ldots, L\}
\end{array}
\right\} \;\; \forall\, k \in \mathcal{K}.
\end{array}
\end{equation}
This problem can be solved in polynomial time if $\Xi$ is convex and the number of policies $K$ is fixed.
\end{obs}

\begin{proof}
Fix any feasible solution $(\theta, \bm{x}, \bm{y})$ to the scenario-based $K$-adaptability problem~\eqref{eq:master_k_adaptability}. For every $\bm{\xi} \in \Xi$, we can construct a feasible solution $(\zeta, \bm{\xi}, \bm{z})$ to problem~\eqref{eq:separation} with $\zeta = S (\theta, \bm{x}, \bm{y}, \bm{\xi})$ by setting $z_{k0} = 1$ if $\bm{c}^\top \bm{x} + \bm{d} (\bm{\xi})^\top \bm{y}_k - \theta \geq \bm{t}_l (\bm{\xi})^\top \bm{x} + \bm{w}_l (\bm{\xi})^\top \bm{y}_k - h_l (\bm{\xi})$ for all $l = 1, \ldots, L$ and $z_{kl} = 1$ for $l \in \mathop{\arg \max}\limits_{l \in \{1, \ldots, L\}} \left\{ \bm{t}_l (\bm{\xi})^\top \bm{x} + \bm{w}_l (\bm{\xi})^\top \bm{y}_k - h_l (\bm{\xi}) \right\}$ otherwise (where ties can be broken arbitrarily). We thus conclude that $\mathcal{S} (\theta, \bm{x}, \bm{y})$ is less than or equal to the optimal value of problem~\eqref{eq:separation}. Likewise, every feasible solution $(\zeta, \bm{\xi}, \bm{z})$ to problem~\eqref{eq:separation} satisfies $\zeta \leq \max \big\{ \bm{c}^\top \bm{x} + \bm{d} (\bm{\xi})^\top \bm{y}_k - \theta, \; \max\limits_{l \in \{ 1, \ldots, L \}} \; \{ \bm{t}_l (\bm{\xi})^\top \bm{x} + \bm{w}_l (\bm{\xi})^\top \bm{y}_k - h_l (\bm{\xi}) \} \big\}$ for all $k \in \mathcal{K}$; that is, $\zeta \leq S (\theta, \bm{x}, \bm{y}, \bm{\xi})$. Thus, the optimal value of problem~\eqref{eq:separation} is less than or equal to $\mathcal{S} (\theta, \bm{x}, \bm{y})$ as well.

If the number of policies $K$ is fixed and the uncertainty set $\Xi$ is convex, then problem~\eqref{eq:separation} can be solved by enumerating all $(L + 1)^K$ possible choices for $\bm{z}$, solving the resulting linear programs in $\zeta$ and $\bm{\xi}$ and reporting the solution with the maximum value of $\zeta$.
\end{proof}

\paragraph{The Algorithm.}
Our solution scheme solves a sequence of scenario-based $K$-adaptability problems~\eqref{eq:master_k_adaptability} over monotonically increasing scenario sets $\Xi_k$, $k \in \mathcal{K}$. At each iteration, the separation problem~\eqref{eq:separation} identifies a new scenario $\bm{\xi} \in \Xi$ to be added to these sets.
\begin{enumerate}
\item \textit{Initialize.}\label{algorithm:initialize} Set $\mathcal{N} \gets \{ \tau^0 \}$ (node set), where $\tau^0 = (\Xi_1^0, \ldots, \Xi_K^0)$ with $\Xi_k^0 = \emptyset$ for all $k \in \mathcal{K}$ (root node). Set $(\theta^\text{i}, \bm{x}^\text{i}, \bm{y}^\text{i}) \gets (+\infty, \emptyset, \emptyset)$ (incumbent solution).

\item \textit{Check convergence.}\label{algorithm:check_convergence} If $\mathcal{N} = \emptyset$, then stop and declare infeasibility (if $\theta^\text{i} = +\infty$) or report $(\bm{x}^\text{i}, \bm{y}^\text{i})$ as an optimal solution to problem~\eqref{eq:k_adaptability}.

\item \textit{Select node.}\label{algorithm:node_selection} Select a node $\tau = (\Xi_1, \ldots, \Xi_K)$ from $\mathcal{N}$. Set $\mathcal{N} \gets \mathcal{N} \setminus \{ \tau \}$.

\item \textit{Process node.}\label{algorithm:master_problem} Let $(\theta, \bm{x}, \bm{y})$ be an optimal solution to the scenario-based $K$-adaptability problem~\eqref{eq:master_k_adaptability}. If $\theta \geq \theta^\text{i}$, then go to Step~\ref{algorithm:check_convergence}.

\item \textit{Check feasibility.}\label{algorithm:separation_problem} Let $(\zeta, \bm{\xi}, \bm{z})$ be an optimal solution to the separation problem~\eqref{eq:separation}. If $\zeta \leq 0$, then set $(\theta^\text{i}, \bm{x}^\text{i}, \bm{y}^\text{i}) \gets (\theta, \bm{x}, \bm{y})$ and go to Step~\ref{algorithm:check_convergence}.

\item \textit{Branch.}\label{algorithm:branching} Instantiate $K$ new nodes $\tau_1, \ldots, \tau_K$ as follows: $\tau_k = (\Xi_1, \ldots, \Xi_k \cup \{ \bm{\xi} \}, \ldots, \Xi_K)$ for each $k \in \mathcal{K}$. Set $\mathcal{N} \gets \mathcal{N} \cup \{\tau_1, \ldots, \tau_K\}$ and go to Step~\ref{algorithm:node_selection}.
\end{enumerate}

Our branch-and-bound algorithm can be interpreted as an uncertainty set partitioning scheme. For a solution $(\theta, \bm{x}, \bm{y})$ in Step~\ref{algorithm:master_problem}, the sets
\begin{equation*}
\Xi(\theta, \bm{x}, \bm{y}_k) = \left\{ \bm{\xi} \in \Xi \, : \, \bm{c}^\top \bm{x} + \bm{d}(\bm{\xi})^\top \bm{y}_k \leq \theta,\; \bm{T} (\bm{\xi}) \bm{x} + \bm{W} (\bm{\xi}) \bm{y}_k \leq \bm{h} (\bm{\xi}) \right\}, \;\; k \in \mathcal{K},
\end{equation*}
describe the regions of the uncertainty set $\Xi$ for which at least one of the candidate policies is feasible and results in an objective value smaller than or equal to $\theta$. Step~\ref{algorithm:separation_problem} of the algorithm attempts to identify a realization $\bm{\xi} \in \Xi \setminus \bigcup_{k \in \mathcal{K}} \Xi(\theta, \bm{x}, \bm{y}_k)$ for which every candidate policy either is infeasible or results in an objective value that exceeds $\theta$. If there is no such realization, then the solution $(\bm{x}, \bm{y})$ is feasible in the $K$-adaptability problem~\eqref{eq:k_adaptability}. Otherwise, Step~\ref{algorithm:branching} assigns the realization $\bm{\xi}$ to each scenario subset $\Xi_k$, $k \in \mathcal{K}$, in turn. Figure~\ref{figure:algorithm_example} illustrates our solution scheme.

\begin{figure}[!ht]
\centering
\includegraphics[width=0.7\textwidth]{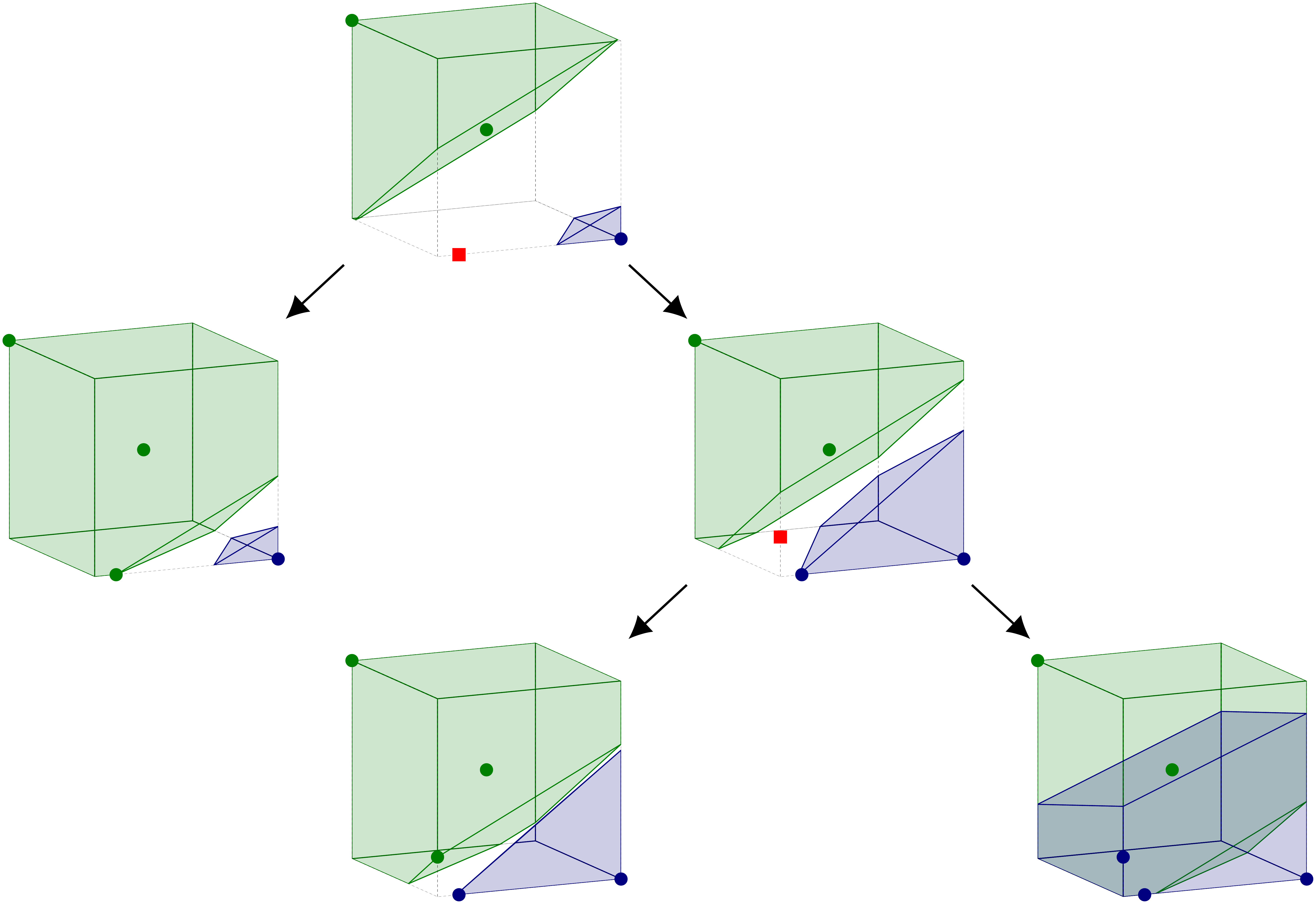}
\caption{An illustrative example with $K = 2$ policies. Each cube represents the uncertainty set $\Xi$ while the shaded regions represent $\Xi(\theta, \bm{x}, \bm{y}_1)$ and $\Xi(\theta, \bm{x}, \bm{y}_2)$. The green and blue dots represent elements of the sets $\Xi_1$ and $\Xi_2$, respectively, while the red squares represent the candidate realizations $\bm{\xi}$ identified in Step~\ref{algorithm:separation_problem} of the algorithm.}
\label{figure:algorithm_example}
\end{figure}

\subsection{Convergence Analysis}\label{sec:bab_convergence}

We now establish the correctness of our branch-and-bound scheme, as well as conditions for its asymptotic and finite convergence.

\begin{thm}[Correctness]\label{prop:algorithm_correctness}
If the branch-and-bound scheme terminates, then it either returns an optimal solution to problem~\eqref{eq:k_adaptability} or correctly identifies the latter as infeasible.
\end{thm}
\begin{proof}
We first show that if the problem instance is infeasible, then the algorithm terminates with the incumbent solution $(\theta^\text{i}, \bm{x}^\text{i}, \bm{y}^\text{i}) = (+\infty, \emptyset, \emptyset)$. Indeed, the algorithm can only update the incumbent solution in Step~\ref{algorithm:separation_problem} if the objective value of the separation problem is non-positive. By construction, this is only possible if the algorithm has determined a feasible solution.

We now show that for feasible problem instances, the algorithm terminates with an optimal solution $(\bm{x}^\text{i}, \bm{y}^\text{i})$ of problem~\eqref{eq:k_adaptability}. To this end, assume that $(\bm{x}^\star, \bm{y}^\star)$ is an optimal solution of problem~\eqref{eq:k_adaptability} with objective value $\theta^\star$. Let $\mathcal{T}$ be the set of all nodes of the branch-and-bound tree for which $(\theta^\star, \bm{x}^\star, \bm{y}^\star)$ is feasible in the corresponding scenario-based $K$-adaptability problem~\eqref{eq:master_k_adaptability}. Note that $\mathcal{T} \neq \emptyset$ since $(\theta^\star, \bm{x}^\star, \bm{y}^\star)$ is feasible in the root node. Let $\mathcal{T}' \subseteq \mathcal{T}$ be the set of those nodes which have children in $\mathcal{T}$ and consider the set $\mathcal{T}'' = \mathcal{T} \setminus \mathcal{T}'$; by construction, we have $\mathcal{T}'' \neq \emptyset$. Consider an arbitrary node $\tau \in \mathcal{T}''$. By definition of $\mathcal{T}''$, our algorithm has not branched $\tau$. Since $\tau$ has been selected in Step~\ref{algorithm:node_selection}, this is only possible if either \emph{(i)} $\tau$ has been fathomed in Step~\ref{algorithm:master_problem} or if \emph{(ii)} $\tau$ has been fathomed in Step~\ref{algorithm:separation_problem}. In the former case, the solution $(\theta^\star, \bm{x}^\star, \bm{y}^\star)$ must have been weakly dominated by the incumbent solution $(\theta^\text{i}, \bm{x}^\text{i}, \bm{y}^\text{i})$, which therefore must be optimal as well. In the latter case, the incumbent solution must have been updated to $(\theta^\star, \bm{x}^\star, \bm{y}^\star)$.
\end{proof}

We now show that our branch-and-bound scheme converges asymptotically to an optimal solution of the $K$-adaptability problem~\eqref{eq:k_adaptability}. Our result has two implications: \emph{(i)} for infeasible problem instances, the algorithm always terminates after finitely many iterations, i.e., infeasibility is detected in finite time; \emph{(ii)} for feasible problem instances, the algorithm eventually only inspects solutions in the neighborhood of optimal solutions.

\begin{thm}[Asymptotic Convergence]\label{thm:algorithm_convergence}
Every accumulation point $(\hat{\theta}, \bm{\hat{x}}, \bm{\hat{y}})$ of the solutions to the scenario-based $K$-adaptability problem~\eqref{eq:master_k_adaptability} in an infinite branch of the branch-and-bound tree gives rise to an optimal solution $(\bm{\hat{x}}, \bm{\hat{y}})$ of the $K$-adaptability problem~\eqref{eq:k_adaptability} with objective value $\hat{\theta}$.
\end{thm}

\begin{proof}
We denote by $(\theta^\ell, \bm{x}^\ell, \bm{y}^\ell)$ and $(\zeta^\ell, \bm{\xi}^\ell, \bm{z}^\ell)$ the sequences of optimal solutions to the scenario-based $K$-adaptability problem in Step~\ref{algorithm:master_problem} and the separation problem in Step~\ref{algorithm:separation_problem} of the algorithm, respectively, that correspond to the node sequence $\tau^\ell$, $\ell = 0, 1, \ldots$, of some infinite branch of the branch-and-bound tree. Since $\mathcal{X}$, $\mathcal{Y}$ and $\Xi$ are compact, the Bolzano-Weierstrass theorem implies that $(\theta^\ell, \bm{x}^\ell, \bm{y}^\ell)$ and $(\zeta^\ell, \bm{\xi}^\ell, \bm{z}^\ell)$ each have at least one accumulation point.

We first show that every accumulation point $(\hat{\theta}, \bm{\hat{x}}, \bm{\hat{y}})$ of the sequence $(\theta^\ell, \bm{x}^\ell, \bm{y}^\ell)$ corresponds to a \emph{feasible} solution $(\bm{\hat{x}}, \bm{\hat{y}})$ of the $K$-adaptability problem~\eqref{eq:k_adaptability} with objective value $\hat{\theta}$. By possibly going over to subsequences, we can without loss of generality assume that the two sequences $(\theta^\ell, \bm{x}^\ell, \bm{y}^\ell)$ and $(\zeta^\ell, \bm{\xi}^\ell, \bm{z}^\ell)$ converge themselves to $(\hat{\theta}, \bm{\hat{x}}, \bm{\hat{y}})$ and $(\hat{\zeta}, \bm{\hat{\xi}}, \bm{\hat{z}})$, respectively. Assume now that $(\bm{\hat{x}}, \bm{\hat{y}})$ does \emph{not} correspond to a feasible solution of the $K$-adaptability problem~\eqref{eq:k_adaptability} with objective value $\hat{\theta}$. Then there is $\bm{\xi}^\star \in \Xi$ such that
$S (\hat{\theta}, \bm{\hat{x}}, \bm{\hat{y}}, \bm{\xi}^\star) \geq \delta$ for some $\delta > 0$. By construction of the separation problem~\eqref{eq:separation}, this implies that
\begin{equation*}
S (\theta^\ell, \bm{x}^\ell, \bm{y}^\ell, \bm{\xi}^\ell)
\;\; = \;\;
\max_{\bm{\xi} \in \Xi} \; S (\theta^\ell, \bm{x}^\ell, \bm{y}^\ell, \bm{\xi})
\;\; \geq \;\;
S (\theta^\ell, \bm{x}^\ell, \bm{y}^\ell, \bm{\xi}^\star)
\;\; \geq \;\;
\delta / 2
\end{equation*}
for all $\ell$ sufficiently large. By taking limits and exploiting the continuity of $S$, we conclude that
\begin{equation*}
S (\hat{\theta}, \bm{\hat{x}}, \bm{\hat{y}}, \bm{\hat{\xi}})
\;\; \geq \;\;
S (\hat{\theta}, \bm{\hat{x}}, \bm{\hat{y}}, \bm{\xi}^\star)
\;\; \geq \;\;
\delta / 2.
\end{equation*}
Note, however, that $S (\theta^{\ell+1}, \bm{x}^{\ell+1}, \bm{y}^{\ell+1}, \bm{\xi}^\ell) \leq 0$ since $\bm{\xi}^\ell \in \Xi^{\ell + 1}_k$ for some $k \in \mathcal{K}$. Since the sequence $(\theta^{\ell + 1}, \bm{x}^{\ell + 1}, \bm{y}^{\ell + 1})$ also converges to $(\hat{\theta}, \bm{\hat{x}}, \bm{\hat{y}})$ and $\bm{\xi}^\ell$ converges to $\bm{\hat{\xi}}$, we thus conclude that $S (\hat{\theta}, \bm{\hat{x}}, \bm{\hat{y}}, \bm{\hat{\xi}}) \leq 0$, which yields the desired contradiction.

We now show that every accumulation point $(\hat{\theta}, \bm{\hat{x}}, \bm{\hat{y}})$ of the sequence $(\theta^\ell, \bm{x}^\ell, \bm{y}^\ell)$ corresponds to an \emph{optimal} solution $(\bm{\hat{x}}, \bm{\hat{y}})$ of the $K$-adaptability problem~\eqref{eq:k_adaptability} with objective value $\hat{\theta}$. Assume to the contrary that $(\hat{\theta}, \bm{\hat{x}}, \bm{\hat{y}})$ is feasible but suboptimal. Then there is a feasible solution $(\theta', \bm{x}', \bm{y}')$ with $\theta' < \hat{\theta}$ that either \emph{(i)} is used to update the incumbent solution after finitely many iterations, or \emph{(ii)} constitutes the accumulation point of another infinite sequence $(\theta'^{,\ell}, \bm{x}'^{,\ell}, \bm{y}'^{,\ell})$. In the first case, the objective values $\theta^\ell$ of the scenario-based $K$-adaptability problems will be arbitrarily close to $\hat{\theta}$ for $\ell$ sufficiently large, which implies that the corresponding nodes $\tau^\ell$ will be fathomed in Step~\ref{algorithm:master_problem}. Similarly, in the second case the objective values $\theta^\ell$ and $\theta'^{,\ell}$ of the scenario-based $K$-adaptability problems will be arbitrarily close to $\hat{\theta}$ and $\theta'$, respectively, for $\ell$ sufficiently large. Since $\theta' < \hat{\theta}$, the algorithm will fathom the tree nodes corresponding to the sequence $(\theta^\ell, \bm{x}^\ell, \bm{y}^\ell)$ in Step~\ref{algorithm:master_problem}. The result now follows since both cases contradict the assumption that $(\hat{\theta}, \bm{\hat{x}}, \bm{\hat{y}})$ is an accumulation point.
\end{proof}

Theorem~\ref{thm:algorithm_convergence} guarantees that after sufficiently many iterations of the algorithm, our scheme  generates feasible solutions that are close to an optimal solution of the $K$-adaptability problem~\eqref{eq:k_adaptability}. In general, our algorithm may not converge after finitely many iterations. In the following, we discuss a class of problem instances for which finite convergence is guaranteed.

\begin{thm}[Finite Convergence]\label{prop:algorithm_finite_time}
The branch-and-bound scheme terminates after finitely many iterations, if $\mathcal{Y}$ has finite cardinality and only the objective function in problem~\eqref{eq:k_adaptability} is uncertain.
\end{thm}

\begin{proof}
If only the objective function in the $K$-adaptability problem~\eqref{eq:k_adaptability} is uncertain, then the corresponding semi-infinite disjunctive program~\eqref{eq:k_adaptability_semiinfinite} can be written as
\begin{equation*}
\begin{array}{l@{\quad}l@{\qquad}l}
\text{minimize} & \theta \\
\text{subject to} & \theta \in \mathbb{R}, \;\; \bm{x} \in \mathcal{X}, \;\; \bm{y} \in \mathcal{Y}^K \\
& \displaystyle \bm{T} \bm{x} + \bm{W} \bm{y}_k \leq \bm{h} & \displaystyle \forall k \in \mathcal{K} \\
& \displaystyle \bigvee_{k \in \mathcal{K}} \left[
\bm{c}^\top \bm{x} + \bm{d} (\bm{\xi})^\top \bm{y}_k \leq \theta
\right] & \displaystyle \forall \bm{\xi} \in \Xi,
\end{array}
\end{equation*}
see Observation~\ref{obs:semiinfinite_disjunctive_reformulation}. Thus, the scenario-based $K$-adaptability problem~\eqref{eq:master_k_adaptability} becomes
\begin{equation*}
\begin{array}{r@{}l@{\quad}l@{\qquad}l}
\mathcal{M} (\Xi_1, \ldots, \Xi_K) \; = \;\;\;
& \text{minimize} & \theta \\
& \text{subject to} & \theta \in \mathbb{R}, \;\; \bm{x} \in \mathcal{X}, \;\; \bm{y} \in \mathcal{Y}^K \\
&& \displaystyle \bm{T} \bm{x} + \bm{W} \bm{y}_k \leq \bm{h} & \displaystyle \forall k \in \mathcal{K} \\
&& \displaystyle \bm{c}^\top \bm{x} + \bm{d} (\bm{\xi})^\top \bm{y}_k \leq \theta & \displaystyle \forall \bm{\xi} \in \Xi_k, \; \forall k \in \mathcal{K},
\end{array}
\end{equation*}
and the separation problem~\eqref{eq:separation_explicit} can be written as
\begin{equation*}
\begin{aligned}
\mathcal{S}(\theta, \bm{x}, \bm{y}) \, &= \, \max_{\bm{\xi} \in \Xi} \; \min_{k \in \mathcal{K}} \left\{ \bm{c}^\top \bm{x} + \bm{d} (\bm{\xi})^\top \bm{y}_k - \theta \right\} \\
\, &= \,
\bm{c}^\top \bm{x} - \theta + \max_{\bm{\xi} \in \Xi} \; \min_{k \in \mathcal{K}} \left\{ \bm{d} (\bm{\xi})^\top \bm{y}_k \right\}.
\end{aligned}
\end{equation*}

We now show that if $\mathcal{Y}$ has finite cardinality, then our branch-and-bound algorithm terminates after finitely many iterations. To this end, assume that this is not the case, and let $\tau^\ell$, $\ell = 0, 1, \ldots$ be some rooted branch of the tree with infinite length. We denote by $(\theta^\ell, \bm{x}^\ell, \bm{y}^\ell)$ and $(\zeta^\ell, \bm{\xi}^\ell, \bm{z}^\ell)$ the corresponding sequences of optimal solutions to the master and the separation problem, respectively. Since $\mathcal{Y}$ has finite cardinality, we must have $\bm{y}^{\ell_1} = \bm{y}^{\ell_2}$ for some $\ell_1 < \ell_2$.

The solution $(\theta^{\ell_2}, \bm{x}^{\ell_2}, \bm{y}^{\ell_2})$ satisfies $\mathcal{S} (\theta^{\ell_2}, \bm{x}^{\ell_2}, \bm{y}^{\ell_2}) > 0$ since $\tau^\ell$, $\ell = 0, 1, \ldots$, is a branch of infinite length. Since $\bm{y}^{\ell_2} = \bm{y}^{\ell_1}$, we thus conclude that
\begin{equation*}
\bm{c}^\top \bm{x}^{\ell_2} -\theta^{\ell_2} + \max_{\bm{\xi} \in \Xi} \; \min_{k \in \mathcal{K}} \; \left\{ \bm{d} (\bm{\xi})^\top \bm{y}_k^{\ell_1} \right\}
\; > \; 0.
\end{equation*}
Since $\bm{\xi}^{\ell_1}$ is optimal in the separation problem $\mathcal{S} (\theta^{\ell_1}, \bm{x}^{\ell_1}, \bm{y}^{\ell_1})$ and $S (\theta^{\ell_1}, \bm{x}^{\ell_1}, \bm{y}^{\ell_1}, \bm{\xi}^{\ell_1}) > 0$, we have
\begin{equation*}
\bm{c}^\top \bm{x}^{\ell_2} -\theta^{\ell_2} + \min_{k \in \mathcal{K}} \; \left\{ \bm{d} (\bm{\xi}^{\ell_1})^\top \bm{y}_k^{\ell_1} \right\}
\; = \;
\bm{c}^\top \bm{x}^{\ell_2} -\theta^{\ell_2} + \max_{\bm{\xi} \in \Xi} \; \min_{k \in \mathcal{K}} \; \left\{ \bm{d} (\bm{\xi})^\top \bm{y}_k^{\ell_1} \right\}
\; > \;  0.
\end{equation*}
However, since the node $\tau^{\ell_2} = (\Xi_1^{\ell_2}, \ldots, \Xi_K^{\ell_2})$ is a descendant of the node $\tau^{\ell_1} = (\Xi_1^{\ell_1}, \ldots, \Xi_K^{\ell_1})$, we must have $\bm{\xi}^{\ell_1} \in \Xi^{\ell_2}_k$ for some $k \in \mathcal{K}$. This, along with the fact that $(\theta^{\ell_2}, \bm{x}^{\ell_2}, \bm{y}^{\ell_2})$ is a feasible solution to the master problem $\mathcal{M} (\Xi_1^{\ell_2}, \ldots, \Xi_K^{\ell_2})$ and that $\bm{y}^{\ell_2} = \bm{y}^{\ell_1}$, implies that
\begin{equation*}
\bm{c}^\top \bm{x}^{\ell_2} -\theta^{\ell_2} + \min_{k \in \mathcal{K}} \; \left\{ \bm{d} (\bm{\xi}^{\ell_1})^\top \bm{y}_k^{\ell_1} \right\}
\; \leq \; 0.
\end{equation*}
This yields the desired contradiction and proves the theorem.
\end{proof}

We note that the assumption of deterministic constraints is critical in the previous statement.

\begin{ex}\label{ex:screw-up}
Consider the following instance of the $K$-adaptability problem~\eqref{eq:k_adaptability}:
\begin{equation*}
\inf_{y_1, y_2 \in \{ 0, 1\}} \; \sup_{\xi \in [0,1]} \; \inf_{k \in \{ 1, 2 \}}
\left\{ (\xi - 1) (1 - 2 y_k) \, : \, y_k \geq \xi \right\}
\end{equation*}
On this instance, our branch-and-bound algorithm generates a tree in which all branches have finite length, except (up to permutations) the sequence of nodes $\tau^\ell = (\Xi_1^\ell, \Xi_2^\ell)$, where $(\Xi_1^0, \Xi_2^0) = (\emptyset, \emptyset)$ and $(\Xi_1^\ell, \Xi_2^\ell) = \left( \big\{ \xi^0 2^{-i} \, : \, i = 0, 1, \ldots, \ell-1 \big\}, \{0\} \right)$, $\ell > 0$, for some $\xi^0 \in (0, 1]$.
For the node $\tau^\ell$, $\ell > 1$, the optimal solution of the scenario-based $K$-adaptability problem~\eqref{eq:master_k_adaptability} is $(\theta^\ell, y_1^\ell, y_2^\ell) = (1 - \xi^0 2^{-\ell+1}, 1, 0)$, while the optimal solution of the separation problem is $(\zeta^\ell, \xi^\ell) = (\xi^0 2^{-\ell}, \xi^0 2^{-\ell})$. Thus, our branch-and-bound algorithm does not terminate after finitely many iterations.
\end{ex}

We note that every practical implementation of our branch-and-bound scheme will fathom nodes in Step~\ref{algorithm:separation_problem} whenever the objective value of the separation problem~\eqref{eq:separation_explicit} is sufficiently close to zero (within some $\epsilon$-tolerance). This ensures that the algorithm terminates in finite time in practice. Indeed, in Example~\ref{ex:screw-up} the objective value of the separation problem is less than $\epsilon$ for all nodes $\tau^\ell$ with $\ell \geq \log_2 (\xi^0  \epsilon^{-1})$, and our branch-and-bound algorithm will fathom the corresponding path of the tree after $\mathcal{O} (\log \epsilon^{-1})$ iterations if we seek $\epsilon$-precision solutions.

\subsection{Improvements to the Basic Algorithm}\label{sec:bab_extensions}
The algorithm of Section~\ref{sec:bab_algorithm} serves as a blueprint that can be extended in multiple ways. In the following, we discuss three enhancements that improve the numerical performance of our algorithm.

\paragraph{Breaking Symmetry.}
For any feasible solution $(\bm{x}, \bm{y})$ of the $K$-adaptability problem~\eqref{eq:k_adaptability}, every solution $(\bm{x}, \bm{y}')$, where $\bm{y}'$ is one of the $K!$ permutations of the second-stage policies $(\bm{y}_1, \ldots, \bm{y}_K)$, is also feasible in~\eqref{eq:k_adaptability} and attains the same objective value. This implies that our branch-and-bound tree is highly isomorphic since the scenario-based problems~\eqref{eq:master_k_adaptability} and~\eqref{eq:separation} are identical (up to a permutation of the policies) across many nodes. We can reduce this undesirable symmetry by modifying Step~\ref{algorithm:branching} of our branch-and-bound scheme as follows:
\begin{enumerate}
\item[\ref{algorithm:branching}$'$.] \textit{Branch.} Let $K' = 1$ if $\Xi_1 = \ldots = \Xi_K = \emptyset$ and let $K' = \min \Big\{K, \, 1 + \max\limits_{k \in \mathcal{K}} \big\{k : \Xi_k \neq \emptyset \big\} \Big\}$ otherwise. Instantiate $K'$ new nodes $\tau_k = (\Xi_1, \ldots, \Xi_k \cup \{ \bm{\xi} \}, \ldots, \Xi_K)$, $k = 1, \ldots, K'$. Set $\mathcal{N} \gets \mathcal{N} \cup \{\tau_1, \ldots, \tau_{K'}\}$ and go to Step~\ref{algorithm:node_selection}.
\end{enumerate}
Despite generating only a subset of the nodes that our original algorithm constructs, the modification above always maintains at least one of the $K!$ solutions symmetric to every feasible solution.

\paragraph{Integration into MILP Solvers.}
Step~\ref{algorithm:master_problem} of our algorithm solves the scenario-based problem~\eqref{eq:master_k_adaptability} from scratch in every node, despite the fact that two successive problems along any branch of the branch-and-bound tree
differ only by the addition of a few constraints. We can leverage this commonality if we integrate our branch-and-bound algorithm into the solution scheme of the MILP solver used for problem~\eqref{eq:master_k_adaptability}. In doing so, we can also exploit the advanced facilities commonly present in the state-of-the-art solvers such as warm-starts and cutting planes, among others. 

In order to integrate our branch-and-bound algorithm into the solution scheme of the MILP solver, we initialize the solver with the scenario-based problem~\eqref{eq:master_k_adaptability} corresponding to the root node $\tau^0$ of our algorithm, see Step~\ref{algorithm:initialize}. The solver then proceeds to solve this problem using its own branch-and-bound procedure. Whenever the solver encounters an \emph{integral solution} $(\theta, \bm{x}, \bm{y}) \in \mathbb{R} \times \mathcal{X} \times \mathcal{Y}^K$, we solve the associated separation problem~\eqref{eq:separation}. If $\mathcal{S}(\theta, \bm{x}, \bm{y}) > 0$, then we execute Step~\ref{algorithm:branching} of our algorithm through a \emph{branch callback}: we report the $K$ new branches to the solver, which will discard the current solution. Otherwise, if $\mathcal{S}(\theta, \bm{x}, \bm{y}) \leq 0$, then we do not create any new branches, and the solver will accept $(\theta, \bm{x}, \bm{y})$ as the new incumbent solution. This ensures that only those solutions which are feasible in problem~\eqref{eq:k_adaptability_semiinfinite} are accepted as incumbent solutions.

Whenever the solver encounters a \emph{fractional solution}, it will by default branch on an integer variable that is fractional in the current solution. However, if $\mathcal{S} (\theta, \bm{x}, \bm{y}) > 0$, it is possible to override this strategy and instead execute Step~\ref{algorithm:branching} of our algorithm. In such cases, a heuristic rule can be used to decide whether to branch on integer variables or to branch as in Step~\ref{algorithm:branching}. In our computational experience, a simple rule that alternates between the default branching rule of the solver and the one defined by Step~\ref{algorithm:branching} appears to perform well in practice.

\subsection{Modification as a Heuristic Algorithm}\label{sec:bab_heuristic}
Whenever the number of policies $K$ is large, the solution of the scenario-based $K$-adaptability problem~\eqref{eq:master_k_adaptability} 
can be time consuming. In such cases, only a limited number of nodes will be explored by the algorithm in a given amount of computation time, and the quality of the final incumbent solution may be poor. As a remedy, we can reduce the size and complexity of the scenario-based $K$-adaptability problem~\eqref{eq:master_k_adaptability} by fixing some of its second-stage policies. In doing so, we obtain a heuristic variant of our algorithm that can scale to large values of $K$.

In our computational experience, a simple heuristic that {sequentially solves the $1$-, $2$-, \ldots, $K$-adaptability problems by fixing in each $K$-adaptability problem} all but one of the second-stage policies, $\bm{y}_1, \ldots, \bm{y}_{K-1}$, to their corresponding values in the $(K-1)$-adaptability problem, performs well in practice. This heuristic is motivated by two observations. First, the resulting scenario-based $K$-adaptability problems~\eqref{eq:master_k_adaptability} have the same size and complexity as the corresponding scenario-based $1$-adaptability problems. Second, in our experiments on instances with uncertain objective coefficients $\bm{d}$, we often found that some optimal second-stage policies of the $(K-1)$-adaptability problem also appear in the optimal solution of the $K$-adaptability problem. In fact, it can be shown that this heuristic can obtain $K$-adaptable solutions that improve upon $1$-adaptable solutions only if the objective coefficients $\bm{d}$ are affected by uncertainty.

\section{Numerical Results}\label{sec:num_results}
We now analyze the computational performance of our branch-and-bound scheme in a variety of problem instances from the literature. We consider a shortest path problem with uncertain arc weights (Section~\ref{sec:num_results_shortest_path}), a capital budgeting problem with uncertain cash flows (Section~\ref{sec:num_results_capital_budgeting}), a variant of the capital budgeting problem with the additional option to take loans (Section~\ref{sec:num_results_capital_budgeting_loans}), a project management problem with uncertain task durations (Section~\ref{sec:num_results_project_management}), and a vehicle routing problem with uncertain travel times (Section~\ref{sec:num_results_vehicle_routing}). Of these, the first two problems involve only binary decisions, and they can therefore also be solved with the approach described in~\cite{HKW15:rip}. In these cases, we show that our solution scheme is highly competitive, and it frequently outperforms the approach of~\cite{HKW15:rip}. In contrast, the third and fourth problems also involve continuous decisions, and there is no existing solution approach for their associated $K$-adaptability problems. However, the project management problem from Section~\ref{sec:num_results_project_management} involves \emph{only} continuous second-stage decisions, and therefore the corresponding two-stage robust optimization problem~\eqref{eq:two_stage_ro} can also be approximated using affine decision rules~\cite{BTGGN04:adjustable}, which represent the most popular approach for such problems. In this case, we elucidate the benefits of $K$-adaptable constant and affine decisions over standard affine decision rules.
Finally, the first and last problems involve only binary second-stage decisions and deterministic constraints, and they can therefore also be addressed with the heuristic approach described in~\cite{BK16:min_max_max_MP}. In these cases, we show that the heuristic variant of our algorithm often outperforms the latter approach in terms of solution quality.

For each problem category, we investigate the tradeoffs between computational effort and improvement in objective value of the $K$-adaptability problem for increasing values of $K$. We demonstrate that \emph{(i)} the $K$-adaptability problem can provide significant improvements over static robust optimization (which corresponds to the case $K=1$), and that \emph{(ii)} our solution scheme can quickly determine feasible solutions of high quality.

We implemented our branch-and-bound algorithm in C++ using the C~callable library of CPLEX~12.7~\cite{CPLEX}. 
We used a constraint feasibility tolerance of $\epsilon = 10^{-4}$ to accept any incumbent solutions, whereas all other solver options were kept at their default values. The experiments were conducted on a single core of an Intel~Xeon~2.8GHz computer with 16GB~RAM. 

\subsection{Shortest Paths}\label{sec:num_results_shortest_path}
We consider the shortest path problem from~\cite{HKW15:rip}. Let $G = (V, A)$ be a directed graph with nodes $V = \{1, \ldots, N\}$, arcs $A \subseteq V \times V$ and arc weights $d_{ij}(\bm{\xi}) = (1 + \xi_{ij}/2) d_{ij}^0$, $(i, j) \in A$. Here, $d_{ij}^0 \in \mathbb{R}_{+}$ represents the nominal weight of the arc $(i, j) \in A$ and $\xi_{ij}$ denotes the uncertain deviation from the nominal weight. The realizations of the uncertain vector $\bm{\xi}$ are known to belong to the set
\begin{equation*}
\Xi = \left\{ \bm{\xi} \in [0,1]^{\lvert A \rvert}: \sum_{(i,j)\in A} \xi_{ij} \leq \Gamma \right\},
\end{equation*}
which stipulates that at most $\Gamma$ arc weights may maximally deviate from their nominal values.

Let $s \in V$ and $t \in V$, $s \neq t$, denote the source and terminal nodes of $G$, respectively. The decision-maker aims to choose $K$ paths from $s$ to $t$ here-and-now, i.e., before observing the actual arc weights, such that the worst-case weight of the shortest among the chosen paths is minimized. This problem can be formulated as an instance of the $K$-adaptability problem~\eqref{eq:k_adaptability}:
\begin{equation*}
\inf_{\bm{y} \in \mathcal{Y}^K} \; \sup_{\bm{\xi} \in \Xi} \; \inf_{k \in \mathcal{K}} \;
\bm{d} (\bm{\xi})^\top \bm{y}_k
\end{equation*}
Here, $\mathcal{Y}$ denotes the set of all $s-t$ paths in $G$; that is,
\begin{equation*}
\mathcal{Y} = \left\{ \bm{y} \in \{0, 1\}^{\lvert A \rvert} : \sum_{(j,l) \in A} y_{jl} - \sum_{(i,j) \in A} y_{ij} \geq \mathbb{I}[j=s] - \mathbb{I}[j=t] \;\; \forall \, j \in V \right\}.
\end{equation*}
Note that this problem only contains second-stage decisions and as such, the corresponding two-stage robust optimization problem~\eqref{eq:two_stage_ro} may be of limited interest in practice. Nevertheless, the $K$-adaptability problem~\eqref{eq:k_adaptability} has important applications in logistics and disaster relief~\cite{HKW15:rip}.

For each graph size $N \in \{20, 25, \ldots, 50\}$, we randomly generate 100 problem instances as follows. We assign the coordinates $(u_i, v_i) \in \mathbb{R}^2$ to each node $i \in V$ uniformly at random from the square $[0, 10]^2$. The nominal weight of the arc $(i,j) \in A$ is defined to be the Euclidean distance between the nodes $i$ and $j$; that is, $d_{ij}^0 = \sqrt{(u_i - u_j)^2 + (v_i - v_j)^2}$. The source node $s$ and the terminal node $t$ are defined to be the nodes with the maximum Euclidean distance between them. The arc set $A$ is obtained by removing from the set of all pairwise links the $\lfloor 0.7  (N^2 - N) \rfloor$ connections with the largest nominal weights. We set the uncertainty budget to $\Gamma = 3$. Further details on the parameter settings can be found in~\cite{HKW15:rip}.

Table~\ref{table:num_results_shortest_path} summarizes the numerical performance of our branch-and-bound scheme for $K \in \{2, 3, 4\}$. Table~\ref{table:num_results_shortest_path} indicates that our scheme is able to reliably compute optimal solutions for small values of $N$ and $K$, while the average optimality gap for large values of $N$ and $K$ is less than 9\%.
The numerical performance is strongly affected by the value of $K$; very few of the $4$-adaptable instances are solved to optimality within the time limit. This decrease in tractability is partly explained in Figure~\ref{figure:spp_improvement}, which shows the improvement in objective value of the $K$-adaptability problem over the static problem (where $K=1$). Figure~\ref{figure:spp_improvement_overall} shows that the computed $4$-adaptable solutions are typically of high quality since they improve upon the static solutions by as much as 13\% for large values of $N$. Moreover, Figure~\ref{figure:spp_improvement_profile} shows that these solutions are obtained within 1 minute~(on average), even for the largest instances. This indicates that the gaps in Table~\ref{table:num_results_shortest_path} are likely to be very conservative since the majority of computation time is spent on obtaining a certificate of optimality for these solutions.

\begin{table}[!htb]
  \centering
  \caption{Results for the shortest path problem. For each value of $K$, the ``Opt (\#)'' column reports the number of instances (out of 100) which were solved to optimality, while the ``Time (s)'' column reports the average time to solve these instances to optimality. For those instances which could not be solved to optimality within the time limit of 7,200s, the average gap $\lvert (\text{ub} - \text{lb})/ \text{ub} \rvert \times 100\%$ between the global lower bound (lb) and global upper bound (ub) of the branch-and-bound tree is reported in the ``Gap (\%)'' column.}
    \begin{tabularx}{\textwidth}{cRRRRRRRRR}
    \toprule
          & \multicolumn{3}{c}{$K = 2$} & \multicolumn{3}{c}{$K = 3$} & \multicolumn{3}{c}{$K = 4$} \\
    \cmidrule(lr){2-4} \cmidrule(lr){5-7} \cmidrule(l){8-10}
    $N$     & Opt (\#) & Time (s) & Gap (\%) & Opt (\#) & Time (s) & Gap (\%) & Opt (\#) & Time (s) & Gap (\%) \\
    \midrule
    20    & 99    & 6     & 1.23  & 97    & 408   & 2.51  & 70    & 539   & 1.74 \\
    25    & 91    & 222   & 4.14  & 64    & 847   & 2.91  & 33    & 885   & 2.89 \\
    30    & 64    & 744   & 4.40  & 31    & 1,237 & 4.10  & 16    & 827   & 4.27 \\
    35    & 37    & 1,083 & 5.36  & 14    & 1,020 & 5.01  & 10    & 896   & 5.23 \\
    40    & 10    & 808   & 6.28  & 6     & 1,670 & 6.43  & 2     & 39    & 6.10 \\
    45    & 9     & 1,152 & 7.70  & 1     & 16    & 7.06  & 1     & 15    & 6.61 \\
    50    & 2     & 3,307 & 8.55  & 1     & 2,308 & 7.90  & 0     & --    & 7.10 \\
    \bottomrule
    \end{tabularx}%
  \label{table:num_results_shortest_path}%
\end{table}%

\begin{figure}[!htb]
\captionsetup[subfigure]{belowskip=0pt}
\centering
\begin{subfigure}[b]{0.48\textwidth}
\includegraphics[width=\textwidth]{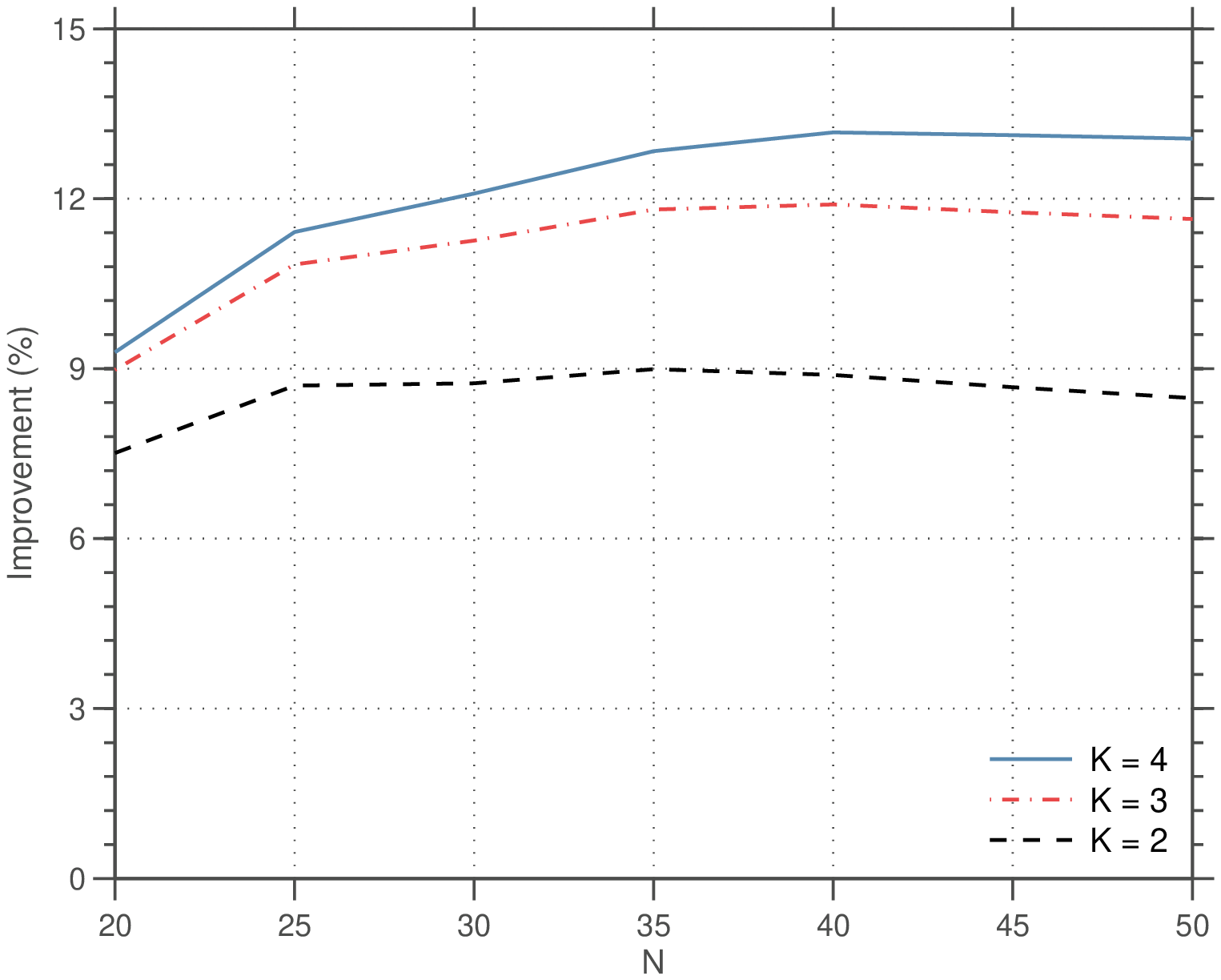}
\caption{}
\label{figure:spp_improvement_overall}
\end{subfigure}~%
\begin{subfigure}[b]{0.48\textwidth}
\includegraphics[width=\textwidth]{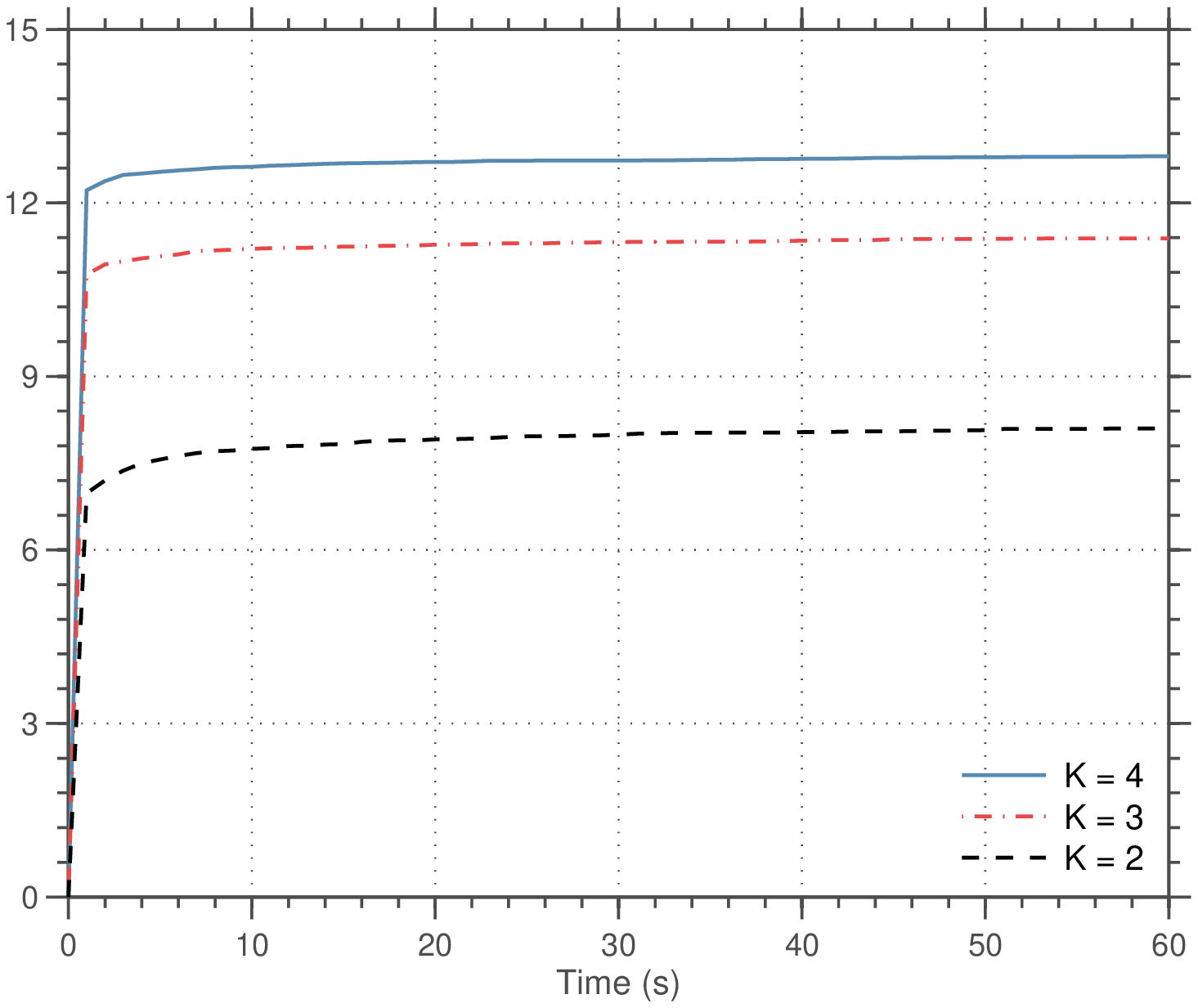}
\caption{}
\label{figure:spp_improvement_profile}
\end{subfigure}
\caption{Results for the shortest path problem using the exact algorithm. The graphs show the average improvement $\lvert(\theta_1 - \theta_K)/\theta_1 \rvert \times 100\%$ of the objective value of the $K$-adaptability problem ($\theta_K$) over the static problem ($\theta_1$). The left graph shows the improvement after 2~hours (for increasing $N$), while the right graph shows the time profile of the improvement of the incumbent solution in the first 60~seconds (for $N = 50$).}
\label{figure:spp_improvement}
\end{figure}

Figure~\ref{figure:spp_improvement_heuristic} illustrates the quality of the solutions obtained using the heuristic variant of our algorithm, described in Section~\ref{sec:bab_extensions}, and contrasts it with the quality of the solutions obtained using the heuristic algorithm described in~\cite{BK16:min_max_max_MP}. Figure~\ref{figure:spp_improvement_heuristic} shows that, after just one minute of computation time, the $2$-, $3$- and $4$-adaptable solutions obtained using our heuristic algorithm are within 0.3\% of known optimal solutions and about 2\% better than those obtained using the heuristic algorithm described in~\cite{BK16:min_max_max_MP}, on average.
The differences in the qualities of the $6$-, $8$- and $10$-adaptable solutions are smaller.
The figure also shows that the marginal gain in objective value decreases rapidly as we increase the number of policies $K$. Indeed, while the $2$-adaptable solutions are about 8.3\% better than the $1$-adaptable (i.e., static) solutions, the $10$-adaptable solutions are only about 0.1\% better than the $8$-adaptable solutions.
This may be explained by the possibility that the objective values of the corresponding $K$-adaptable solutions are very close to the optimal value of the two-stage robust optimization problem~\eqref{eq:two_stage_ro}.

\begin{figure}[!htb]
\captionsetup[subfigure]{belowskip=0pt}
\centering
\begin{subfigure}[b]{0.48\textwidth}
\includegraphics[width=\textwidth]{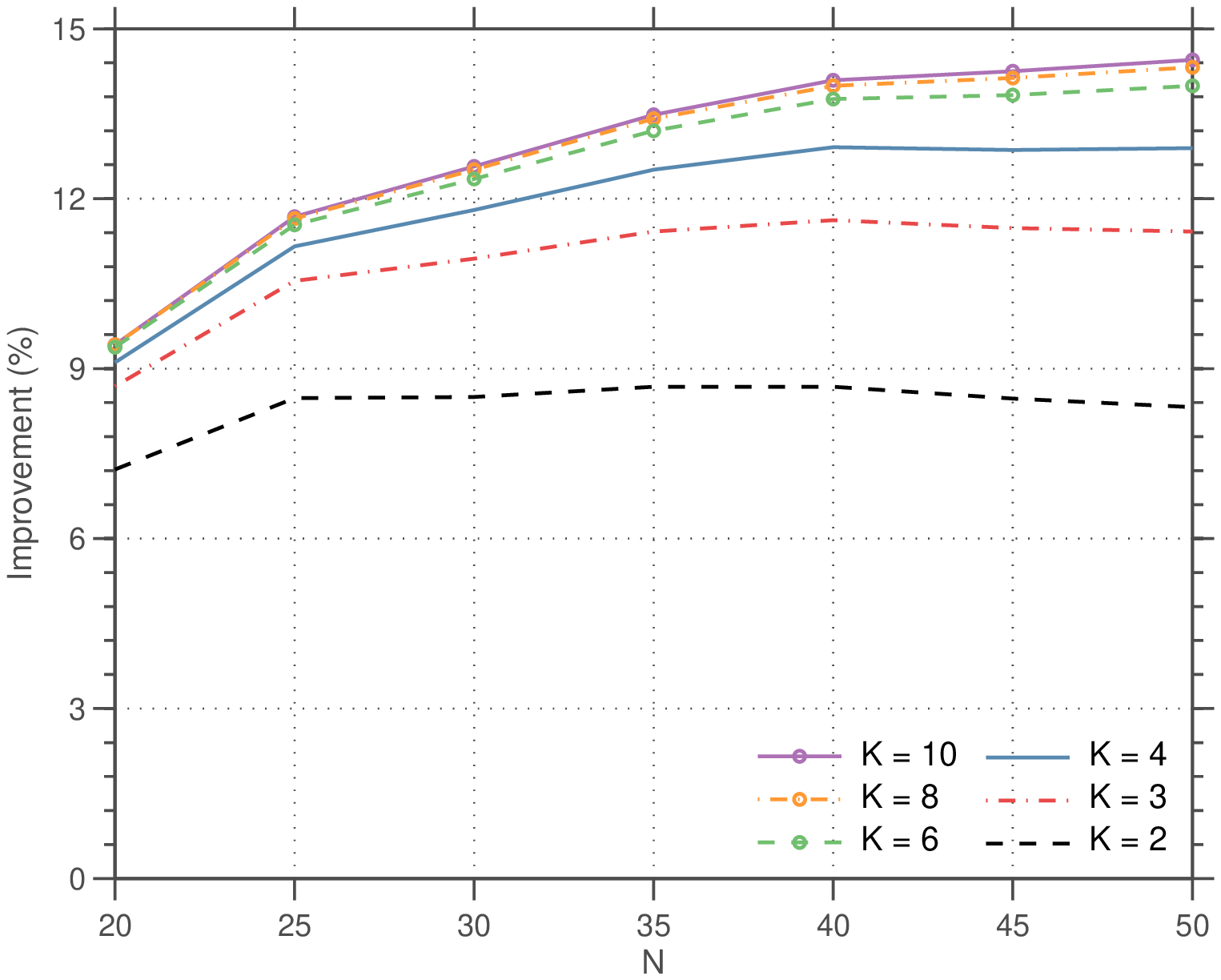}
\caption{}
\label{figure:spp_improvement_heuristic_bab}
\end{subfigure}~%
\begin{subfigure}[b]{0.48\textwidth}
\includegraphics[width=\textwidth]{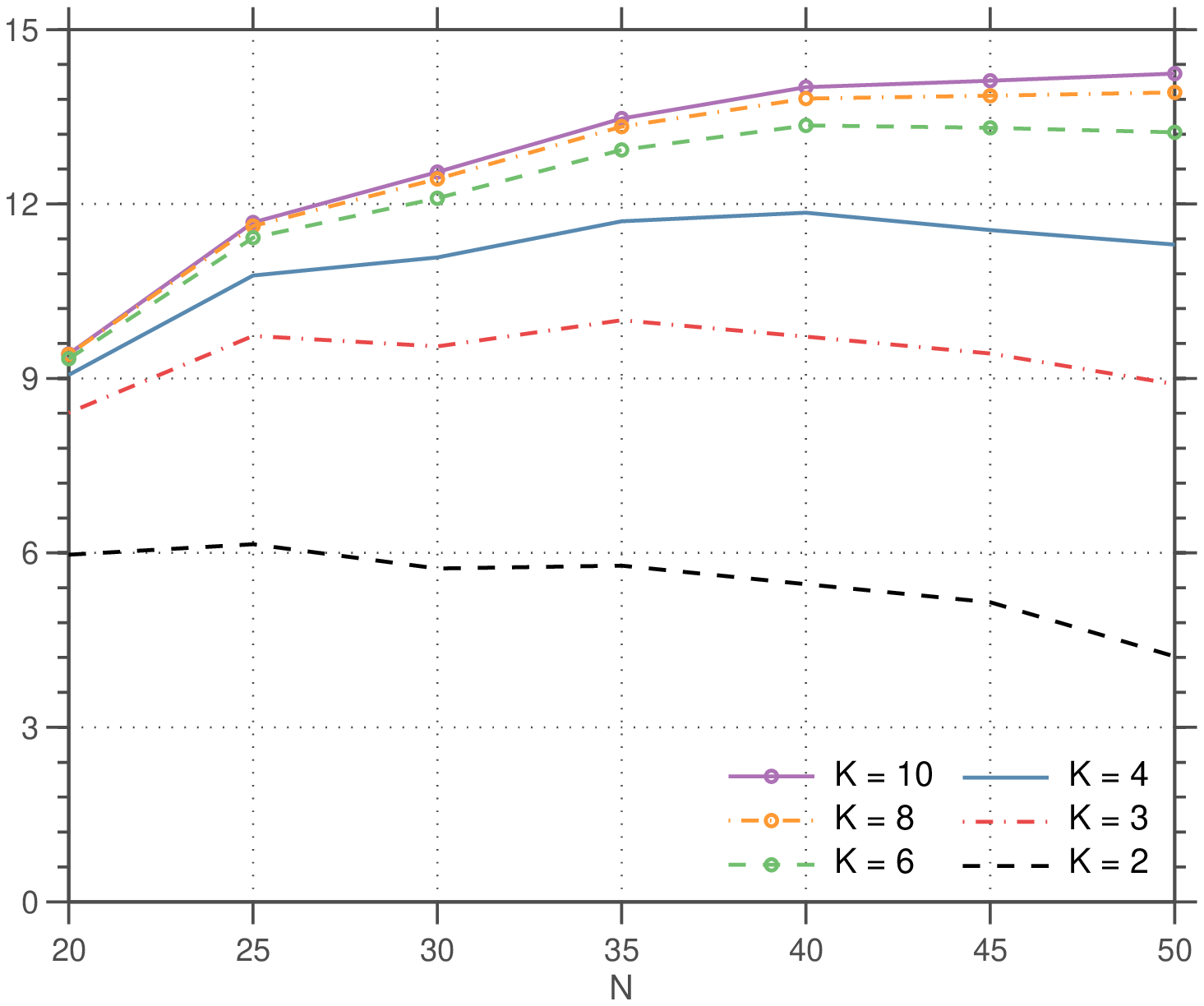}
\caption{}
\label{figure:spp_improvement_heuristic_mmm}
\end{subfigure}
\caption{Results for the shortest path problem using the heuristic algorithm. The graphs show the average improvement after 1~minute obtained using the heuristic variant of our algorithm (left) and the heuristic algorithm described in~\cite{BK16:min_max_max_MP} (right).} 
\label{figure:spp_improvement_heuristic}
\end{figure}

\subsection{Capital Budgeting}\label{sec:num_results_capital_budgeting}
We consider the capital budgeting problem from~\cite{HKW15:rip}, where a company wishes to invest in a subset of $N$ projects. Each project $i$ has an uncertain cost $c_i (\bm{\xi})$ and an uncertain profit $r_i (\bm{\xi})$ that are governed by factor models of the form
\begin{equation*}
c_i (\bm{\xi}) = \left(1 + \bm{\Phi}_i^\top \bm{\xi} / 2 \right) c_i^0 \;\; \text{ and } \;\; r_i (\bm{\xi}) = \left(1 + \bm{\Psi}_i^\top \bm{\xi} / 2 \right) r_i^0 \;\;\;\; \text{for} \; i = 1, \ldots, N.
\end{equation*}
In these models, $c_i^0$ and $r_i^0$ represent the nominal cost and the nominal profit of project $i$, respectively, while $\bm{\Phi}_i^\top$ and $\bm{\Psi}_i^\top$ represent the $i^\text{th}$ row vectors of the factor loading matrices $\bm{\Phi}, \bm{\Psi} \in \mathbb{R}^{N \times 4}$. The realizations of the uncertain vector of risk factors $\bm{\xi}$ belong to the uncertainty set $\Xi = [-1, 1]^4$.

The company can invest in a project either before or after observing the risk factors $\bm{\xi}$. In the latter case, the company generates only a fraction $\kappa$ of the profit (reflecting a penalty for postponement) but incurs the same cost as in the case of an early investment. Given an investment budget $B$, the capital budgeting problem can then be formulated as the following instance of the two-stage robust optimization problem~\eqref{eq:two_stage_ro}:
\begin{equation*}
\adjustlimits\sup_{\bm{x} \in \mathcal{X} \; } \inf_{\bm{\xi} \in \Xi \; } \sup_{\bm{y} \in \mathcal{Y}}
\left\{ \bm{r}(\bm{\xi})^\top (\bm{x} + \kappa \bm{y}) \, : \, \bm{c} (\bm{\xi})^\top (\bm{x} + \bm{y}) \leq B, \; \bm{x} + \bm{y} \leq \mathbf{e} \right\},
\end{equation*}
where $\mathcal{X} = \mathcal{Y} = \{0, 1\}^N$.

For our numerical experiments, we randomly generate 100 instances for each problem size $N \in \{5, 10, \ldots, 30\}$ as follows. The nominal costs $\bm{c}^0$ are chosen uniformly at random from the hyperrectangle $[0, 10]^N$. We then set $\bm{r}^0 = \bm{c}^0 / 5$, $B = \mathbf{e}^\top \bm{c}^0 / 2$ and $\kappa = 0.8$. The rows of the factor loading matrices $\bm{\Phi}$ and $\bm{\Psi}$ are sampled uniformly from the unit simplex in $\mathbb{R}^4$; that is, the $i^\text{th}$ row vector is sampled from $[0, 1]^4$ such that $\bm{\Phi}_{i}^\top \mathbf{e} = \bm{\Psi}_{i}^\top \mathbf{e} = 1$ is satisfied for all $i = 1, \ldots, N$.

Table~\ref{table:num_results_capital_budgeting} summarizes the numerical performance of our branch-and-bound scheme for $K \in \{2, 3, 4\}$. 
Table~\ref{table:num_results_capital_budgeting} demonstrates that our branch-and-bound scheme performs very well for this problem class since the majority of instances is solved to optimality for $K \in \{2, 3\}$. Moreover, the optimality gaps for the unsolved instances are less than 4\% for $K \in \{2, 3\}$ and less than 9\% for $K = 4$ on average. Additionally, Figure~\ref{figure:capital_budgeting_improvement} shows that even for the largest instances, high-quality incumbent solutions which significantly improve ($\approx$100\%) upon the static robust solutions are obtained within 1~minute of computation time. Our results compare favorably with those of~\cite{HKW15:rip} as well as those of the partition-and-bound approach for the corresponding two-stage robust optimization problem presented in~\cite{BD15:multistage_robust_mio}.

\begin{table}[!htb]
  \centering
  \caption{Results for the capital budgeting problem. The columns have the same interpretation as in Table~\ref{table:num_results_shortest_path}.}
    \begin{tabularx}{\textwidth}{cRRRRRRRRR}
    \toprule
          & \multicolumn{3}{c}{$K = 2$} & \multicolumn{3}{c}{$K = 3$} & \multicolumn{3}{c}{$K = 4$} \\
    \cmidrule(lr){2-4} \cmidrule(lr){5-7} \cmidrule(l){8-10}
    $N$     & Opt (\#) & Time (s) & Gap (\%) & Opt (\#) & Time (s) & Gap (\%) & Opt (\#) & Time (s) & Gap (\%) \\
    \midrule
    5     & 100   & 1     & --    & 100   & 1     & --    & 100   & 3     & -- \\
    10    & 100   & 1     & --    & 100   & 16    & --    & 100   & 149   & -- \\
    15    & 100   & 10    & --    & 99    & 566   & 0.33  & 69    & 2,245 & 1.42 \\
    20    & 100   & 419   & --    & 34    & 2,787 & 1.65  & 5     & 3,710 & 4.02 \\
    25    & 29    & 2,238 & 1.12  & 4     & 2,281 & 2.63  & 0     & --    & 6.22 \\
    30    & 1     & 188   & 3.01  & 1     & 6,687 & 3.35  & 0     & --    & 8.27 \\
    \bottomrule
    \end{tabularx}%
  \label{table:num_results_capital_budgeting}%
\end{table}%

\begin{figure}[!htb]
\captionsetup[subfigure]{belowskip=0pt}
\centering
\begin{subfigure}[b]{0.48\textwidth}
\includegraphics[width=\textwidth]{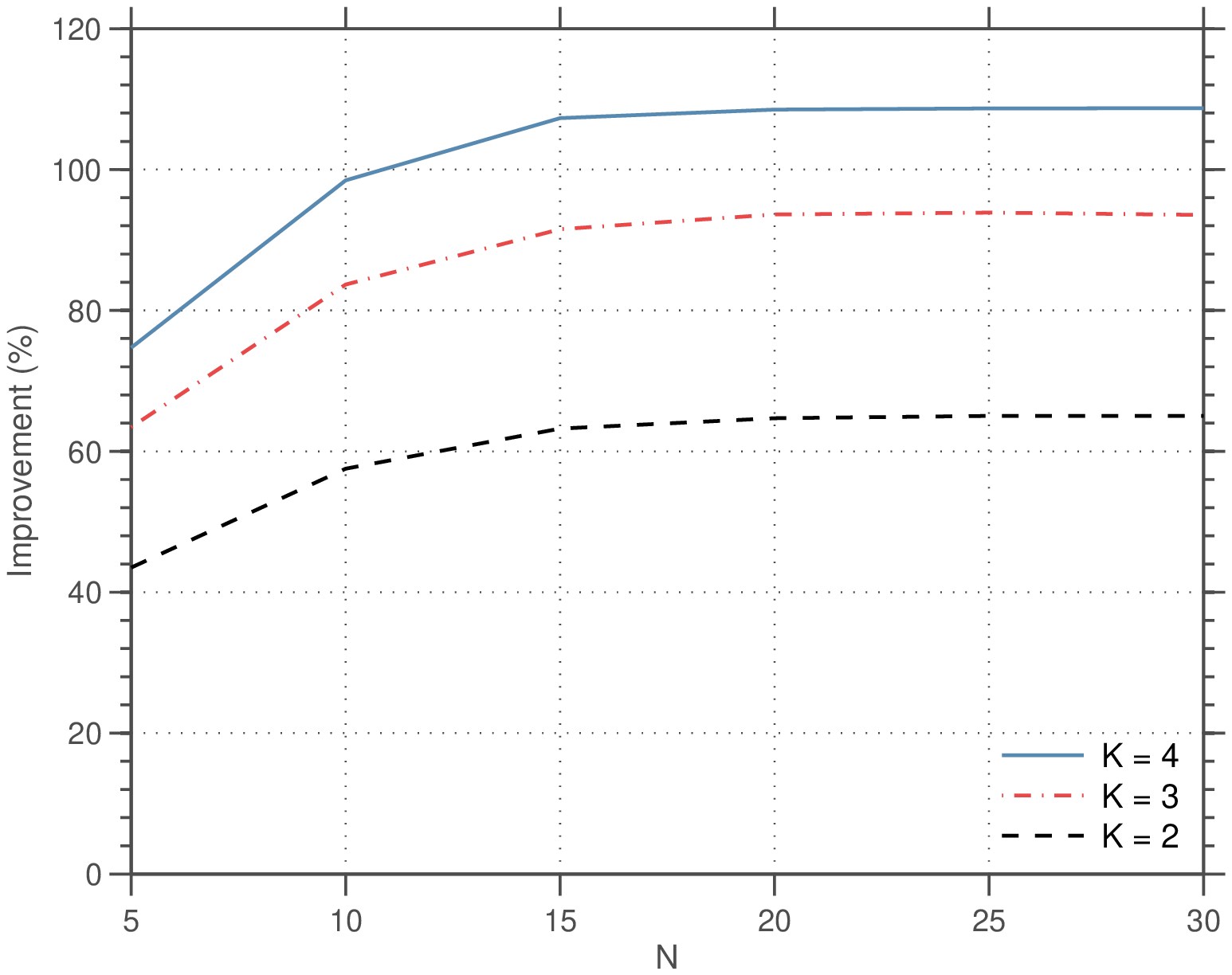}
\caption{}
\label{figure:capital_budgeting_improvement_overall}
\end{subfigure}~%
\begin{subfigure}[b]{0.48\textwidth}
\includegraphics[width=\textwidth]{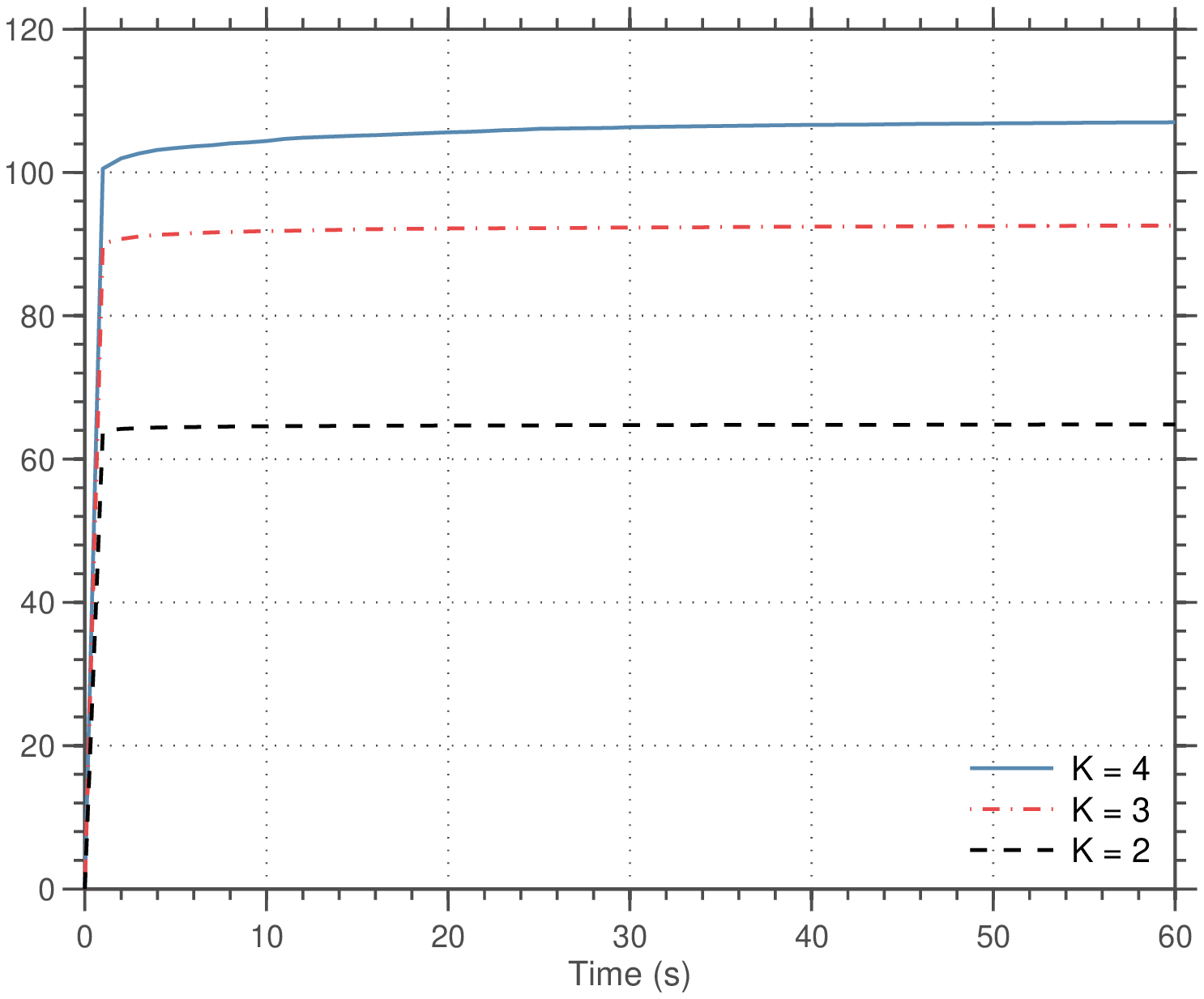}
\caption{}
\label{figure:capital_budgeting_improvement_profile}
\end{subfigure}

\begin{subfigure}[b]{0.48\textwidth}
\includegraphics[width=\textwidth]{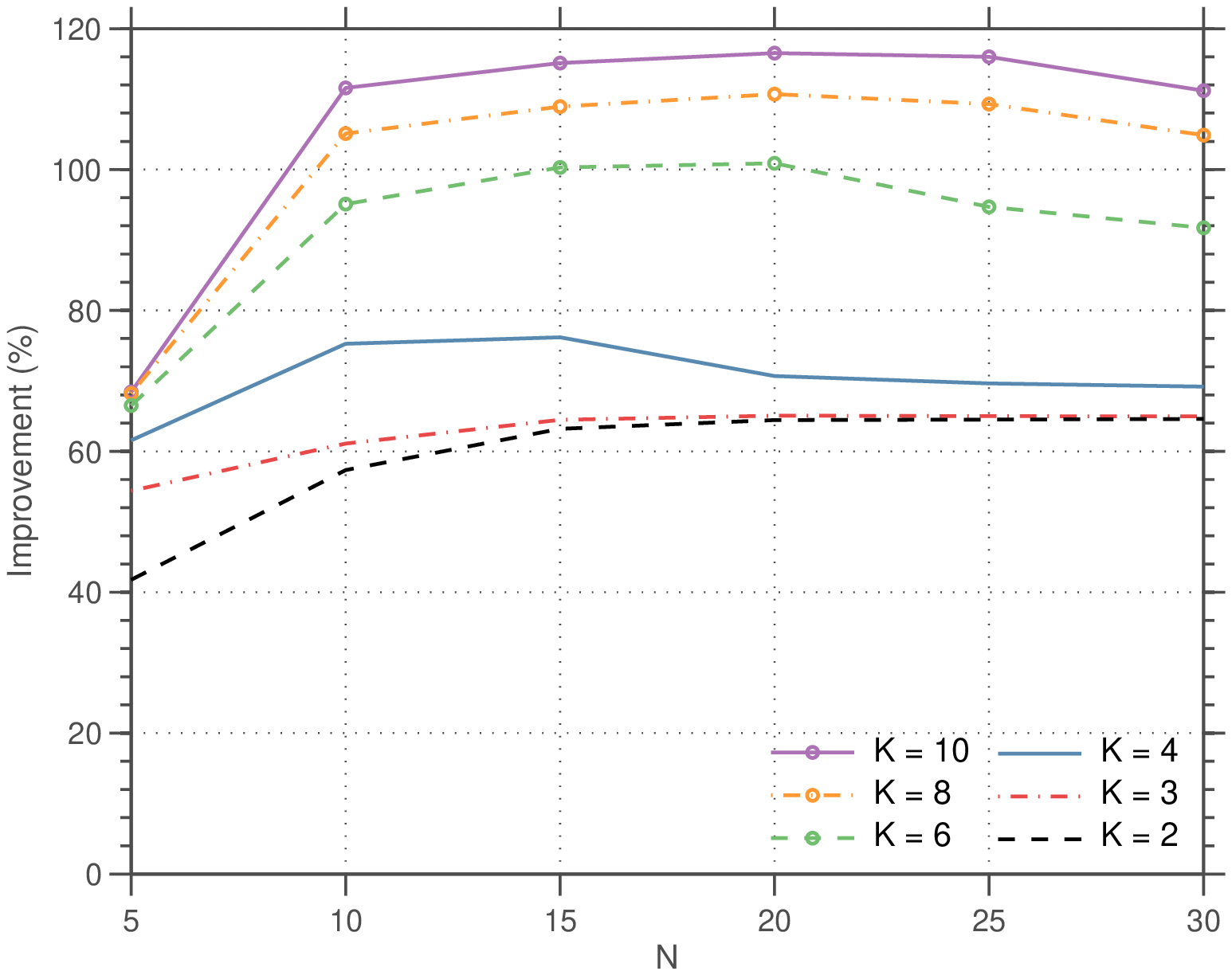}
\caption{}
\label{figure:capital_budgeting_improvement_overall_heuristic}
\end{subfigure}~%
\begin{subfigure}[b]{0.48\textwidth}
\includegraphics[width=\textwidth]{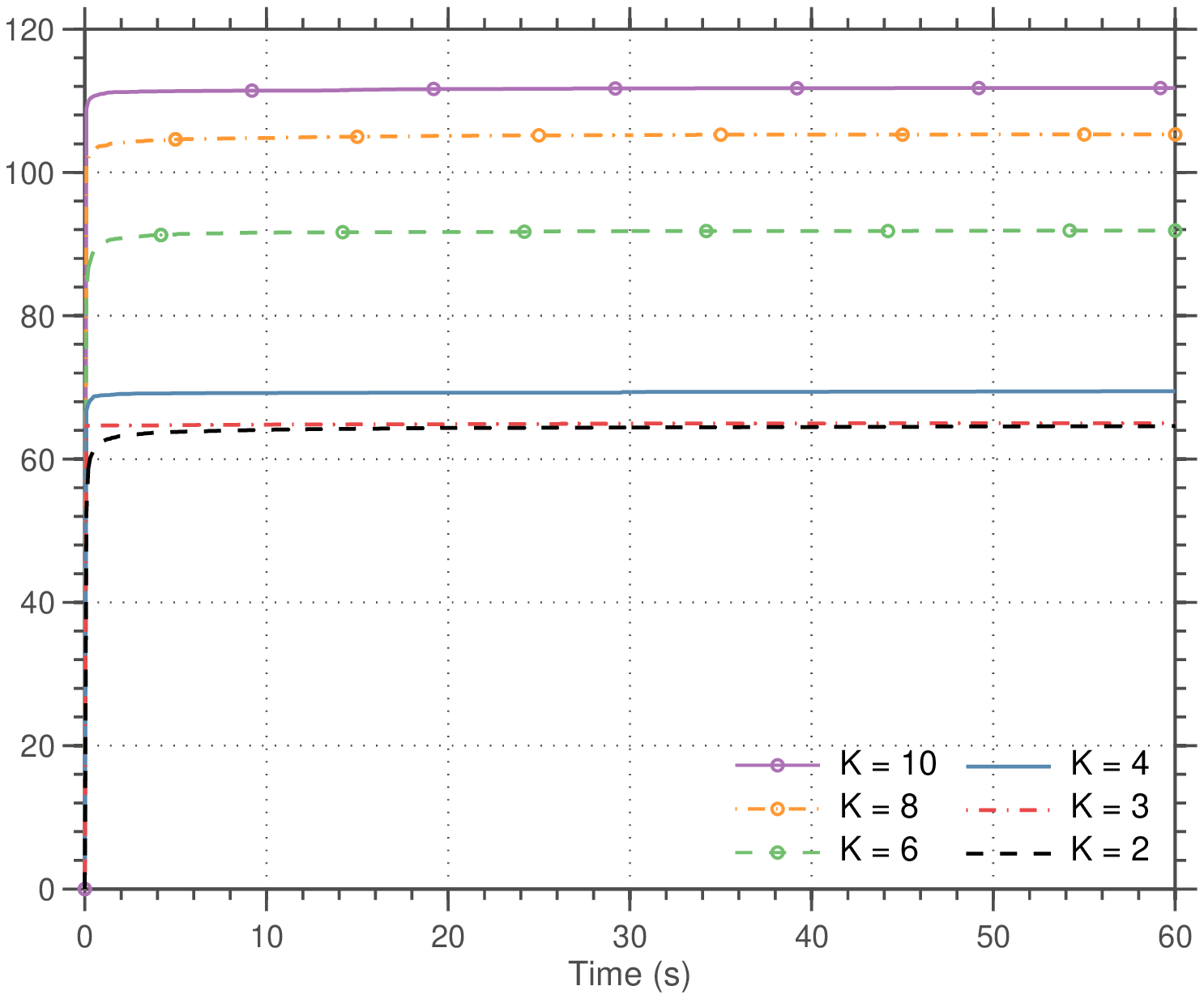}
\caption{}
\label{figure:capital_budgeting_improvement_profile_heuristic}
\end{subfigure}
\caption{Results for the capital budgeting problem. 
The top (bottom) graphs are for the exact (heuristic) algorithm under a time limit of 2~hours (1~minute).
The left graphs show the average improvement at the time limit (for increasing $N$), while the right graphs show the time profile of the improvement of the incumbent solution in the first 60~seconds (for $N = 30$).}
\label{figure:capital_budgeting_improvement}
\end{figure}

\subsection{Capital Budgeting with Loans}\label{sec:num_results_capital_budgeting_loans}
We consider a generalization of the capital budgeting problem from Section~\ref{sec:num_results_capital_budgeting} where the company can increase its investment budget by purchasing a loan from the bank at a unit cost of $\lambda > 0$ before the risk factors $\bm{\xi}$ are observed as well as purchasing a loan at a unit cost of $\mu \lambda$ (with $\mu > 1$) after the observation occurs. If the company does not purchase any loan, then the problem reduces to the one described in Section~\ref{sec:num_results_capital_budgeting}. Therefore, we expect the worst-case profits to be at least as large as in that setting. The generalized capital budgeting problem can be formulated as the following instance of problem~\eqref{eq:two_stage_ro}:
\begin{equation*}
\adjustlimits\sup_{(x_0, \bm{x}) \in \mathcal{X} \; } \inf_{\bm{\xi} \in \Xi \; } \sup_{(y_0, \bm{y}) \in \mathcal{Y}}
\left\{ \bm{r}(\bm{\xi})^\top (\bm{x} + \kappa \bm{y}) - \lambda(x_0 + \mu y_0) \, :
\left[
\begin{array}{l}
\bm{x} + \bm{y} \leq \mathbf{e} \\
\bm{c} (\bm{\xi})^\top \bm{x} \leq B + x_0 \\
\bm{c} (\bm{\xi})^\top (\bm{x} + \bm{y}) \leq B + x_0 + y_0
\end{array}
\right]
\right\}
\end{equation*}
Here, $\mathcal{X} = \mathcal{Y} = \mathbb{R}_{+} \times \{0, 1\}^N$. The constraint $\bm{c} (\bm{\xi})^\top \bm{x} \leq B + x_0$ ensures that the first-stage expenditures $\bm{c} (\bm{\xi})^\top \bm{x} $ are fully covered by the budget $B$ as well as the loan $x_0$ taken here-and-now.

We consider problems with $N \in \{5, 10, \ldots, 30\}$ projects. For each value of $N$, we solve the same 100 instances from Section~\ref{sec:num_results_capital_budgeting} with $\lambda = 0.12$ and $\mu = 1.2$. Table~\ref{table:num_results_capital_budgeting_loans} shows the computational performance of our branch-and-bound scheme for $K \in \{2, 3, 4\}$. As in the case of the problems discussed so far, the numerical tractability of our algorithm decreases as the value of $K$ increases. However, a comparison of Tables~\ref{table:num_results_capital_budgeting} and~\ref{table:num_results_capital_budgeting_loans} suggests that the numerical tractability is not significantly affected by the presence of the additional continuous variables $x_0$ and $y_0$. Indeed, the majority of instances for $K = 2$ are solved to optimality and the average gap across all unsolved instances is less than 5\% for $K=3$ and less than 9\% for $K=4$. Figure~\ref{figure:capital_budgeting_loans_improvement} shows that the $4$-adaptable solutions improve upon the static solutions by as much as 115\% in the largest instances. Although not shown in the figure, a comparison of the objective values of the final incumbent solutions with those of the capital budgeting problem without loans (Section~\ref{sec:num_results_capital_budgeting}) reveals that for $N \geq 15$, the option to purchase loans has no effect on the worst-case profit of the static solution and results in less than 1\% improvement in the worst-case profit of the $2$-adaptable solution. Indeed, the option to purchase loans results in significantly better worst-case profits only if $K \geq 3$. The average relative gain in objective value is 4.3\% for $K=3$ and 5.9\% for $K=4$.

\begin{table}[!hbt]
  \centering
  \caption{Results for the capital budgeting problem with loans. The columns have the same interpretation as in Table~\ref{table:num_results_shortest_path}.}
    \begin{tabularx}{\textwidth}{cRRRRRRRRR}
    \toprule
          & \multicolumn{3}{c}{$K = 2$} & \multicolumn{3}{c}{$K = 3$} & \multicolumn{3}{c}{$K = 4$} \\
    \cmidrule(lr){2-4} \cmidrule(lr){5-7} \cmidrule(l){8-10}
    $N$     & Opt (\#) & Time (s) & Gap (\%) & Opt (\#) & Time (s) & Gap (\%) & Opt (\#) & Time (s) & Gap (\%) \\
    \midrule
    5     & 100   & 1     & --    & 100   & 9     & --    & 98    & 80    & 3.14 \\
    10    & 100   & 3     & --    & 100   & 78    & --    & 98    & 938   & 1.92 \\
    15    & 100   & 62    & --    & 96    & 1,265 & 0.91  & 23    & 3,989 & 2.23 \\
    20    & 85    & 1,680 & 0.80  & 20    & 3,941 & 1.71  & 0     & --    & 4.94 \\
    25    & 12    & 3,363 & 2.29  & 1     & 2,693 & 3.34  & 0     & --    & 6.88 \\
    30    & 1     & 424   & 3.78  & 0     & --    & 4.73  & 0     & --    & 8.17 \\
    \bottomrule
    \end{tabularx}%
  \label{table:num_results_capital_budgeting_loans}%
\end{table}%

\begin{figure}[!htb]
\captionsetup[subfigure]{belowskip=0pt}
\centering
\begin{subfigure}[b]{0.48\textwidth}
\includegraphics[width=\textwidth]{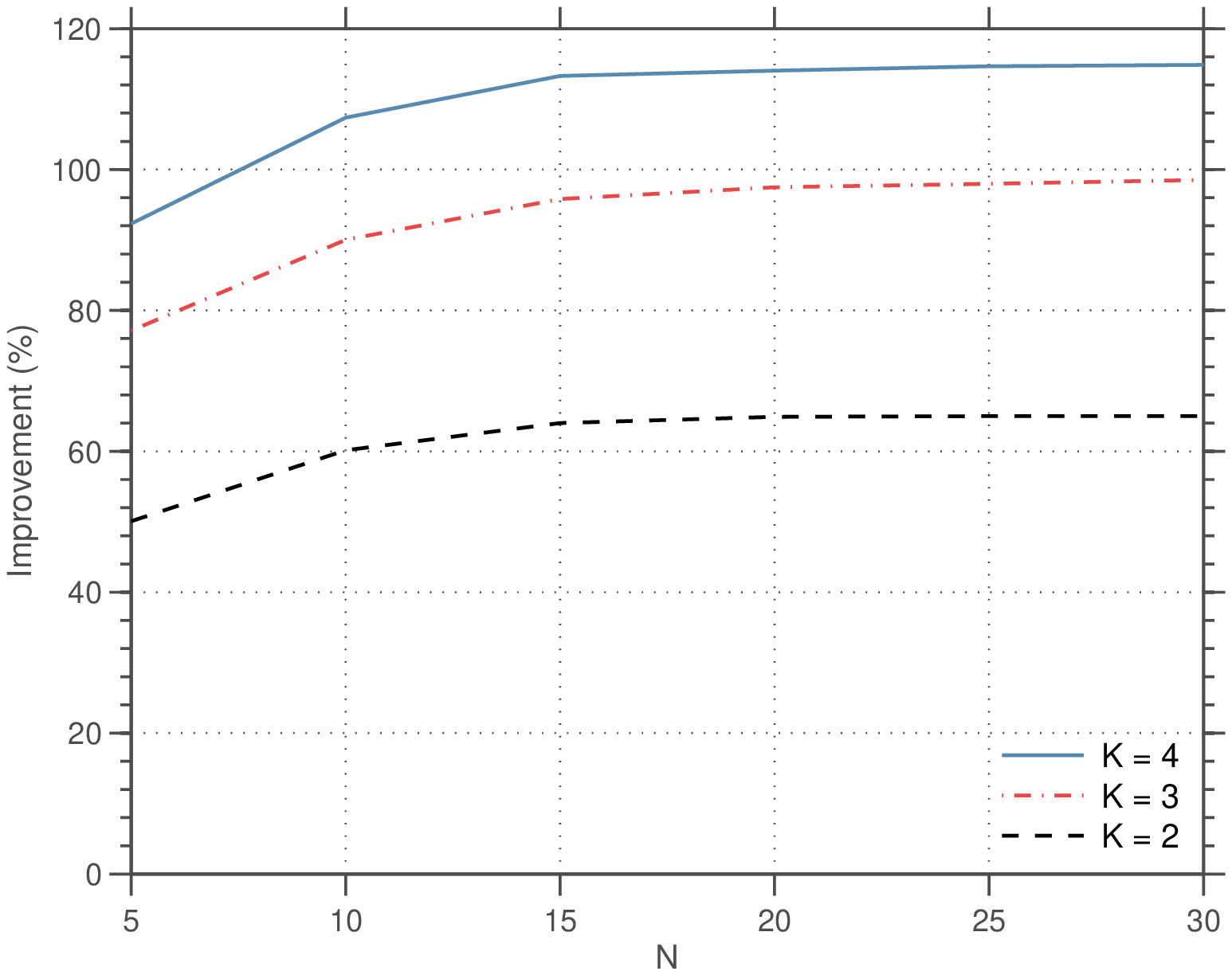}
\caption{}
\label{figure:capital_budgeting_loans_improvement_overall}
\end{subfigure}~%
\begin{subfigure}[b]{0.48\textwidth}
\includegraphics[width=\textwidth]{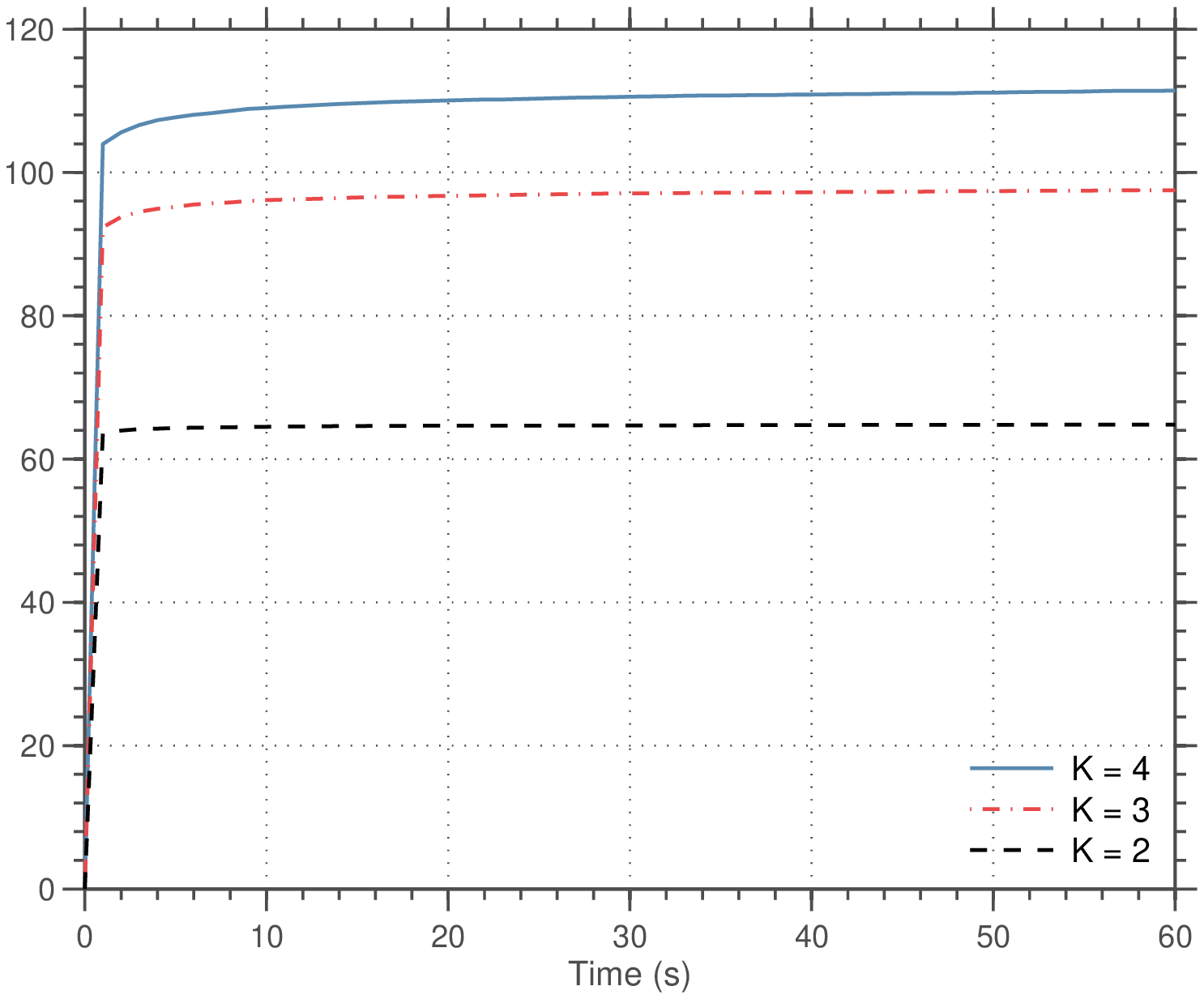}
\caption{}
\label{figure:capital_budgeting_loans_improvement_profile}
\end{subfigure}

\begin{subfigure}[b]{0.48\textwidth}
\includegraphics[width=\textwidth]{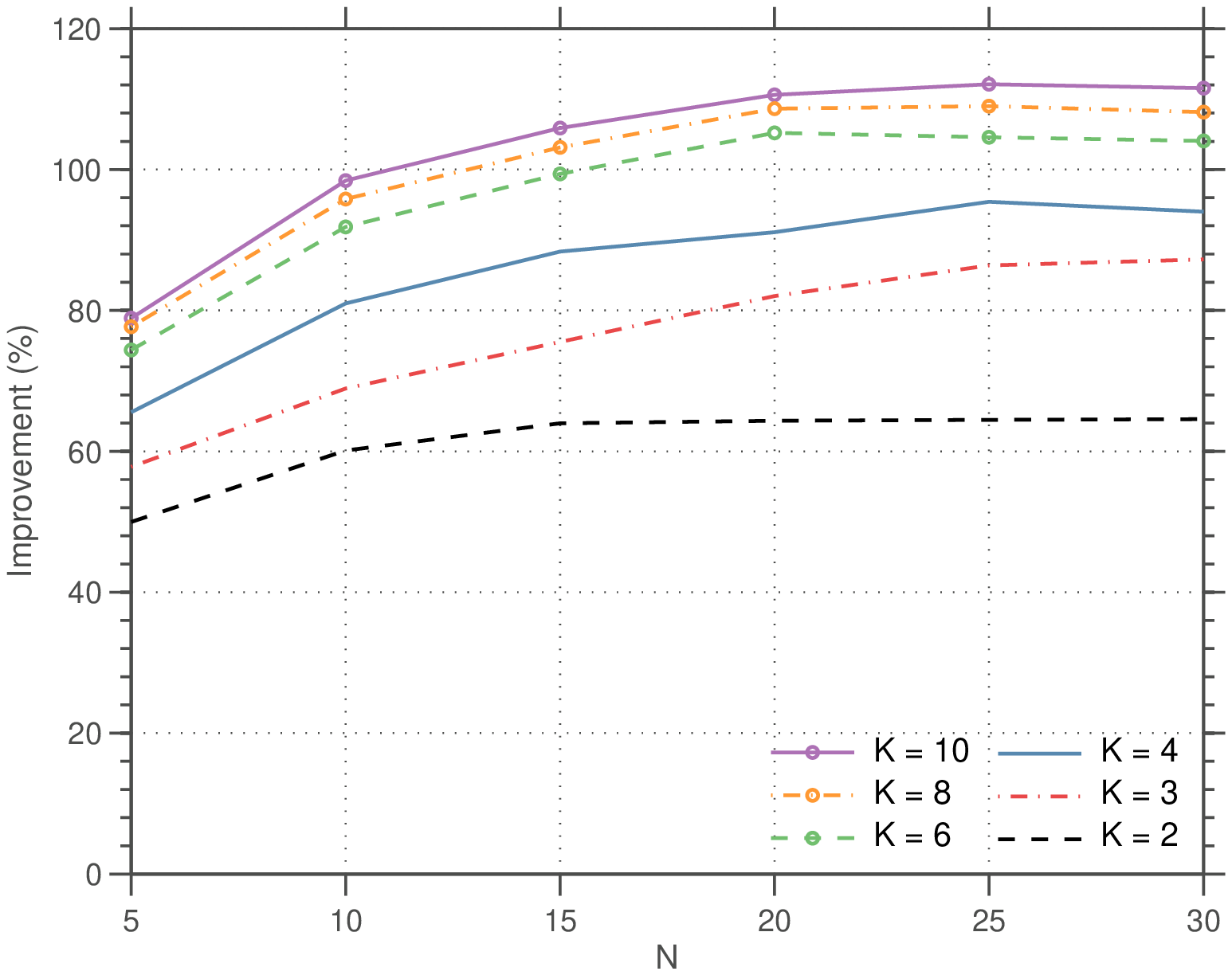}
\caption{}
\label{figure:capital_budgeting_loans_improvement_overall_heuristic}
\end{subfigure}~%
\begin{subfigure}[b]{0.48\textwidth}
\includegraphics[width=\textwidth]{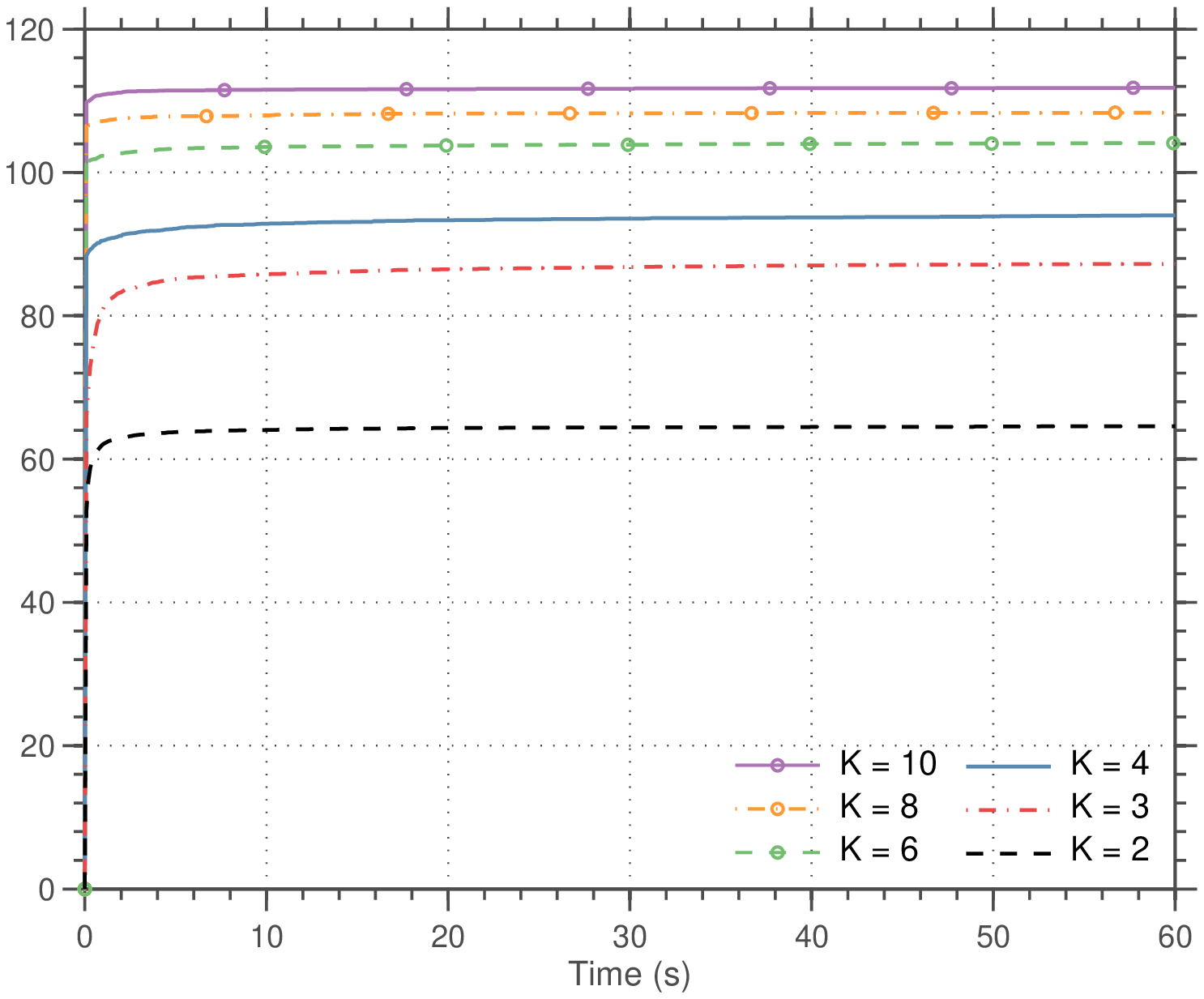}
\caption{}
\label{figure:capital_budgeting_loans_improvement_profile_heuristic}
\end{subfigure}
\caption{Results for the capital budgeting problem with loans. The graphs have the same interpretation as in Figure~\ref{figure:capital_budgeting_improvement}.}
\label{figure:capital_budgeting_loans_improvement}
\end{figure}

\subsection{Project Management}\label{sec:num_results_project_management}

We define a project as a directed acyclic graph $G = (V, A)$ whose nodes $V = \{ 1, \ldots, N \}$ represent the tasks (e.g., `build foundations' or `develop prototype') and whose arcs $A \subseteq V \times V$ denote the temporal precedences, i.e., $(i, j) \in A$ implies that task $j$ cannot be started before task $i$ has been completed. 
We assume that each task $i \in V$ has an uncertain duration $d_i (\bm{\xi})$ that depends on the realization of an uncertain parameter vector $\bm{\xi} \in \Xi$. Without loss of generality, we stipulate that the project graph $G$ has the unique sink $N \in V$, and that the last task $N$ has a duration of zero. This can always be achieved by introducing dummy nodes and/or arcs.

In the following, we want to calculate the worst-case makespan of the project, i.e, the smallest amount of time that is required to complete the project under the worst realization of the parameter vector $\bm{\xi} \in \Xi$. This problem can be cast as the following instance of problem~\eqref{eq:two_stage_ro}:
\begin{equation*}
\sup_{\bm{\xi} \in \Xi} \; \inf_{\bm{y} \in \mathcal{Y}} \;
\left\{ y_N \, : \, y_j - y_{i} \geq d_i (\bm{\xi}) \;\; \forall (i,j) \in A \right\}
\end{equation*}
Here $\mathcal{Y} = \mathbb{R}^N_+$, and $y_i$ denotes the start time of task $i$, $i = 1, \ldots, N$. This problem is known to be NP-hard \cite[Theorem~2.1]{WKR2012:robust_temporal_networks}, and we will employ affine decision rules as well as $K$-adaptable constant and affine decisions to approximate the optimal value of this problem. Note that the problem does not contain any first-stage decisions, but such decisions could be readily included, for example, to allow for resource allocations that affect the task durations.

For our numerical experiments, we consider the instance class presented in~\cite[Example~2.2]{WKR2012:robust_temporal_networks}. To this end, we set $N = 3m + 1$ and $A = \{ (3l + 1, 3l + p), (3l + p, 3l + 4) \, : \, l = 0, \ldots, m \text{ and } p = 2, 3 \}$, $d_{3l+2} (\bm{\xi}) = \xi_{l+1}$ and $d_{3l+3} (\bm{\xi}) = 1 - \xi_{l+1}$, $l = 0, \ldots, m - 1$, as well as $d_{3l+1} (\bm{\xi}) = 0$, $l = 0, \ldots, m$. Figure~\ref{figure:project_network_example} illustrates the project network corresponding to $m = 4$. Similar to~\cite{WKR2012:robust_temporal_networks}, we consider the uncertainty set $\Xi = \{ \bm{\xi} \in \mathbb{R}^m_+ \, : \, \lVert \bm{\xi} - \mathbf{e} / 2 \rVert_1 \leq 1/2 \}$.

\begin{figure}[!htbp]
\centering
\includegraphics[width=0.55\textwidth]{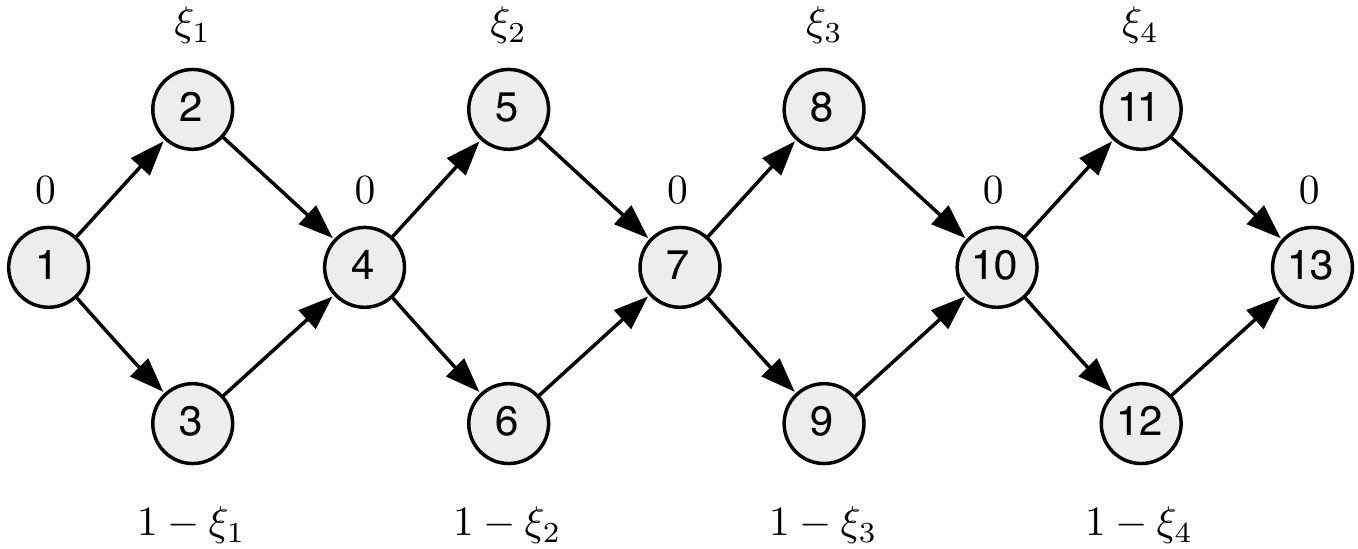}
\caption{Project network with $N = 3m + 1$ nodes for $m = 4$.}
\label{figure:project_network_example}
\end{figure}

We consider project networks of size $N \in \{ 3 m + 1 \, : \, m = 3, 4, \ldots, 8 \}$. One can show that for each network size $N$, the optimal value of the corresponding static robust optimization problem as well as the affine decision rule problem is $m$, see~\cite[Example~2.2]{WKR2012:robust_temporal_networks}. For $K \in \{2, 3, 4\}$, Figure~\ref{figure:project_management_results} summarizes the computational performance of the branch-and-bound scheme and the improvement in objective value of the resulting piecewise constant and piecewise affine decision rules with $K$ pieces over the corresponding $1$-adaptable solutions. Figures~\ref{figure:project_management_static_improvement} and~\ref{figure:project_management_affine_improvement} show that using only two pieces, piecewise constant decision rules can improve upon the affine approximation by more than 12\%, while a piecewise affine decision rule can improve by more than 15\%. Figures~\ref{figure:project_management_static_gap} and~\ref{figure:project_management_affine_gap} show that piecewise constant decision rules require smaller computation times than piecewise affine decision rules. This is not surprising since piecewise constant decision rules are parameterized by $\mathcal{O} (KN)$ variables, whereas piecewise affine decision rules are parameterized by $\mathcal{O} (KN^2)$ variables.

\begin{figure}[!htb]
\captionsetup[subfigure]{belowskip=0pt}
\centering
\begin{subfigure}[b]{0.48\textwidth}
\includegraphics[width=\textwidth]{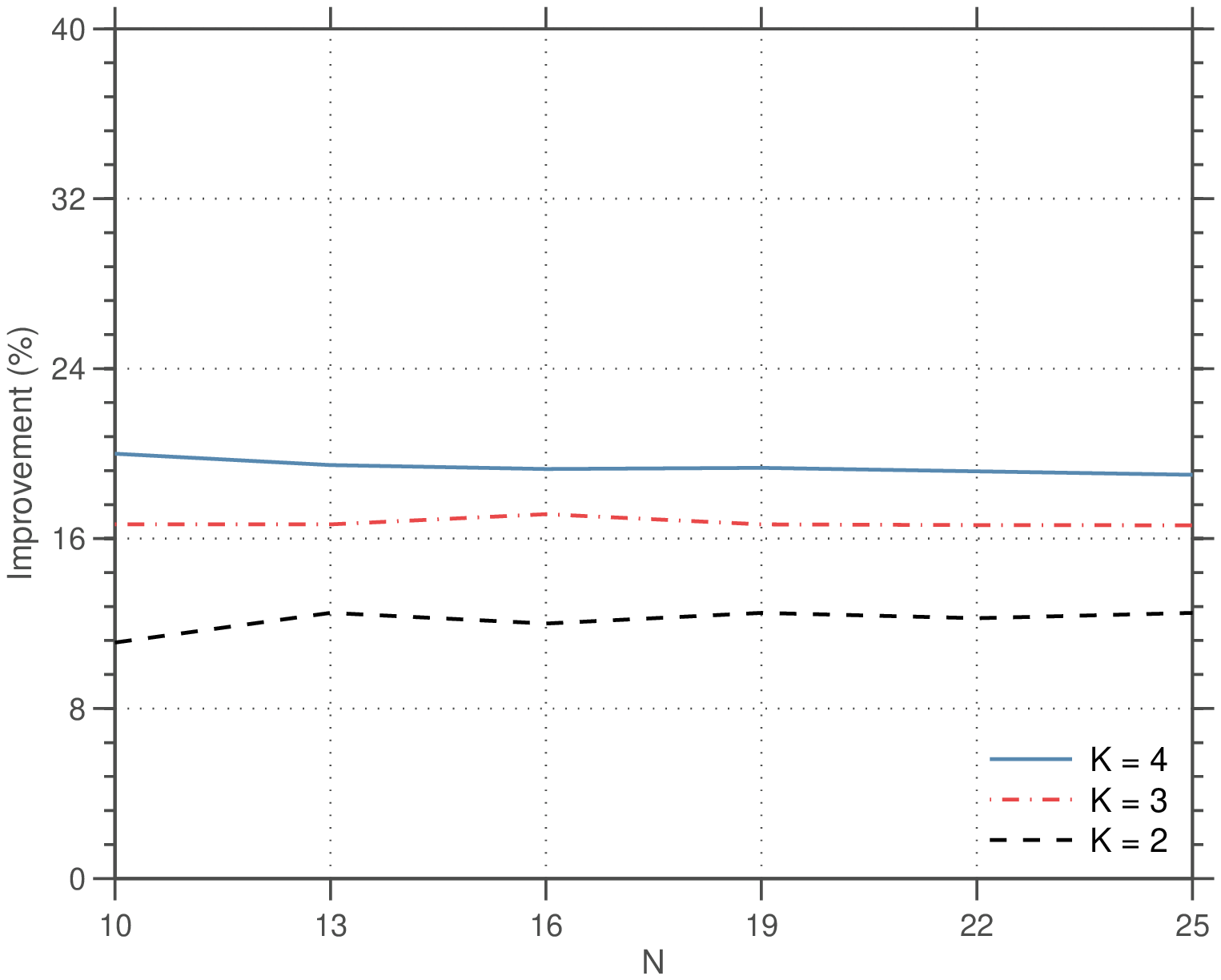}
\caption{}
\label{figure:project_management_static_improvement}
\end{subfigure}~%
\begin{subfigure}[b]{0.48\textwidth}
\includegraphics[width=\textwidth]{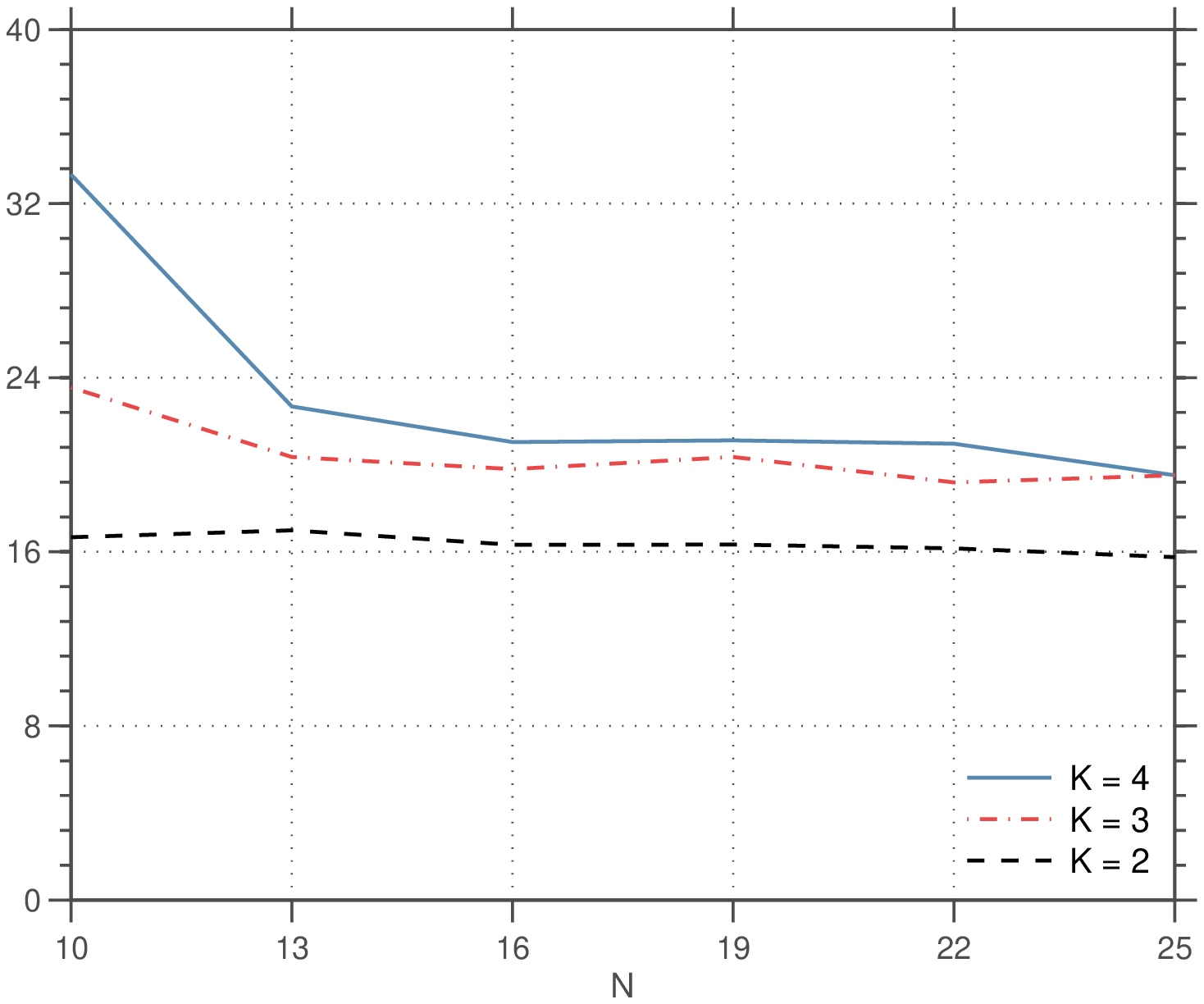}
\caption{}
\label{figure:project_management_affine_improvement}
\end{subfigure}

\begin{subfigure}[b]{0.48\textwidth}
\includegraphics[width=\textwidth]{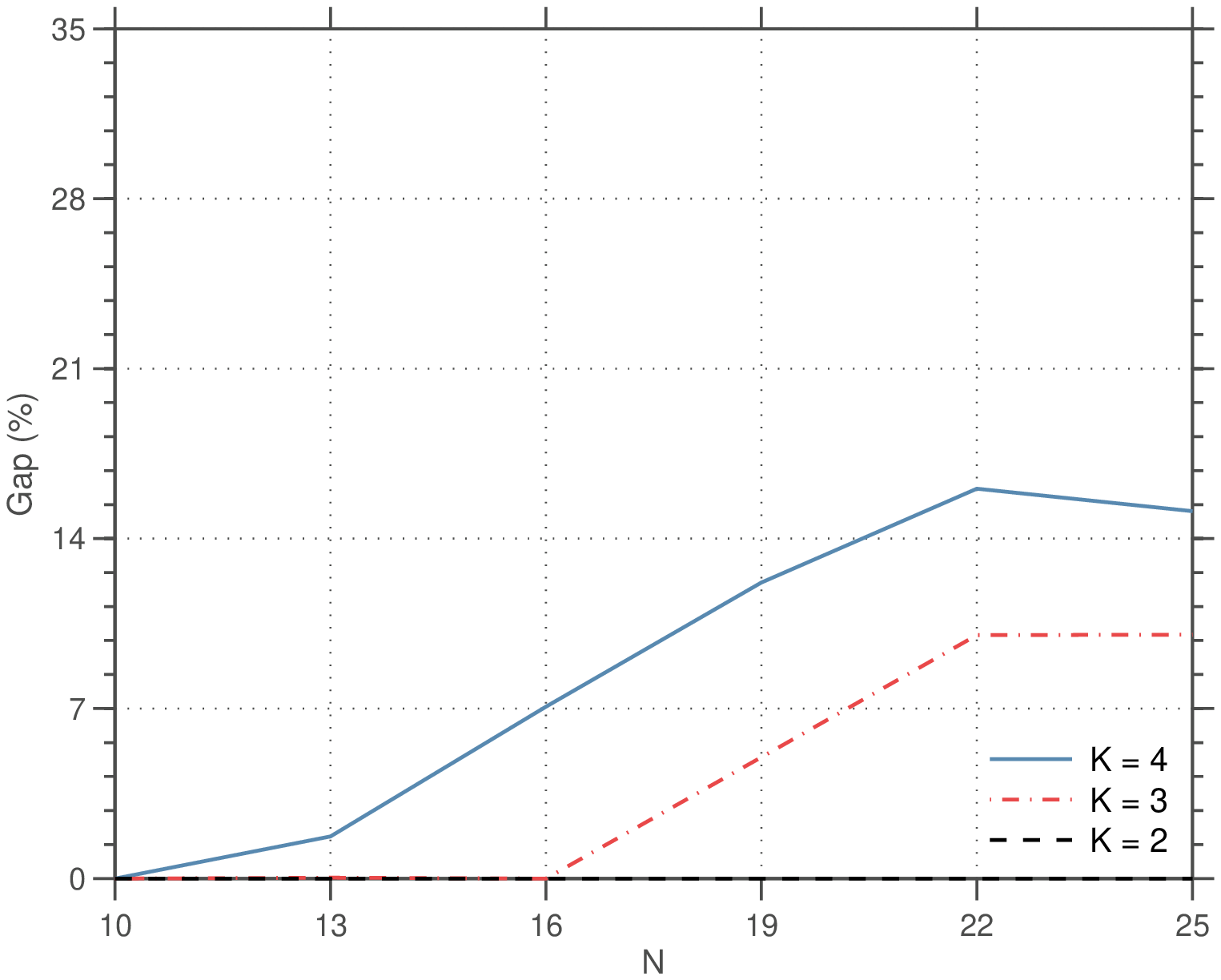}
\caption{}
\label{figure:project_management_static_gap}
\end{subfigure}~%
\begin{subfigure}[b]{0.48\textwidth}
\includegraphics[width=\textwidth]{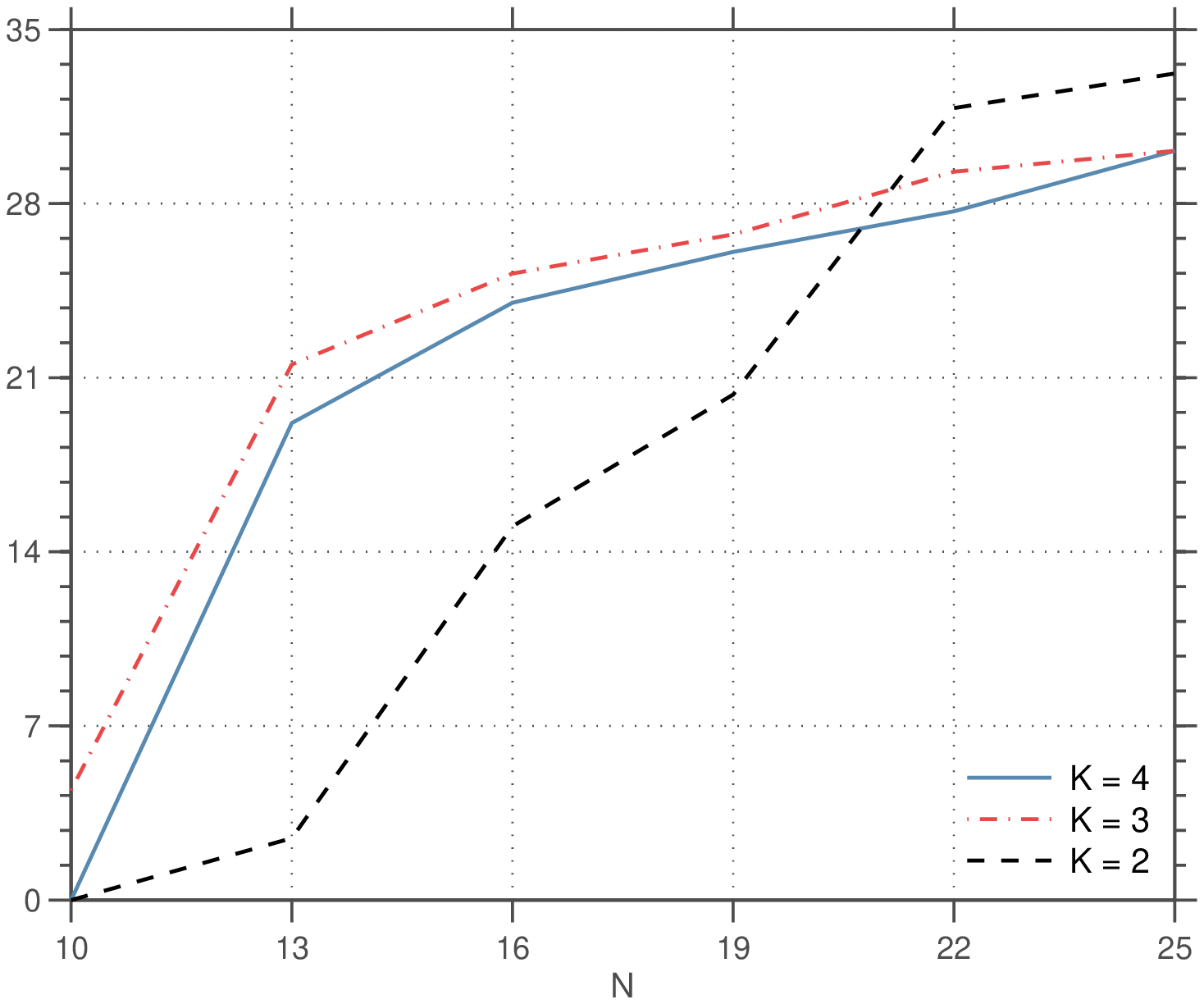}
\caption{}
\label{figure:project_management_affine_gap}
\end{subfigure}
\caption{Results for the project management problem. The left (right) graphs show results for the piecewise constant (affine) $K$-adaptability problems for increasing values of $N$. Graphs (a) and (b) depict improvements in objective value, while graphs (c) and (d) show optimality gaps after $2$ hours. The $y$-axes in graphs (a) and (b) have the same interpretation as those in Figure~\ref{figure:spp_improvement} while the $y$-axes in graphs (c) and (d) have the same interpretation as the column ``Gap (\%)'' in Table~\ref{table:num_results_shortest_path}.}
\label{figure:project_management_results}
\end{figure}

\subsection{Vehicle Routing}\label{sec:num_results_vehicle_routing}
We consider the classical capacitated vehicle routing problem~\cite{GRTWF16:robust_cvrp_demand_unc,GWF13:robust_cvrp_demand_unc,TV14:vehicle_routing_book} defined on a complete, undirected graph $G = (V, E)$ with nodes $V = \{0, 1, \ldots, N\}$ and edges $E = \{(i,j) \in V \times V: i < j \}$. Node $0$ represents the unique depot, while each node $i \in V_C = \{1, \ldots, N\}$ corresponds to a customer with demand $d_i \in \mathbb{R}_{+}$. The depot is equipped with $M$ homogeneous vehicles; each vehicle has capacity $C$ and it incurs an uncertain travel time $t_{ij}(\bm{\xi}) = (1 + \xi_{ij}/2) t_{ij}^0$ when it traverses the edge $(i, j) \in E$. Here, $t_{ij}^0 \in \mathbb{R}_{+}$ represents the nominal travel time along the edge $(i, j) \in E$, while $\xi_{ij}$ denotes the uncertain deviation from the nominal value. Similar to the shortest path problem from Section~\ref{sec:num_results_shortest_path}, the realizations of the uncertain vector $\bm{\xi}$ are known to belong to the set
\begin{equation*}
\Xi = \left\{ \bm{\xi} \in [0,1]^{\lvert E \rvert}: \sum_{(i,j)\in E} \xi_{ij} \leq \Gamma \right\},
\end{equation*}
which stipulates that at most $\Gamma$ travel times may maximally deviate from their nominal values.

A \emph{route plan} $(R_1, \ldots, R_M)$ corresponds to a partition of the customer set $V_C$ into $M$ vehicle routes, $R_m = (R_{m,1},\ldots, R_{m,N_m})$, where $R_{m,l}$ represents the $l^\text{th}$ customer and $N_m$ the number of customers served by the $m^\text{th}$ vehicle. This route plan is feasible if the total demand served on each route is less than the vehicle capacity; that is, if $\sum_{l = 1}^{N_m} d_{R_{m,l}} \leq C$ is satisfied for all $m \in \{1, \ldots, M\}$. The total travel time of a feasible route plan under the uncertainty realization $\bm{\xi}$ is given by $\sum_{m = 1}^M \sum_{l = 0}^{N_m} t_{R_{m,l} R_{m,l+1}} (\bm{\xi})$, where we define $R_{m,0} =  R_{m,N_m+1} = 0$; that is, each vehicle starts and ends at the depot.
The decision-maker aims to choose $K$ route plans here-and-now, i.e., before observing the actual travel times, such that the worst-case total travel time of the shortest among the chosen route plans is minimized. This problem can be formulated as an instance of the $K$-adaptability problem~\eqref{eq:k_adaptability}:
\begin{equation*}
\inf_{\bm{y} \in \mathcal{Y}^K} \; \sup_{\bm{\xi} \in \Xi} \; \inf_{k \in \mathcal{K}} \;
\bm{t} (\bm{\xi})^\top \bm{y}_k
\end{equation*}
Here, $\mathcal{Y}$ denotes the set of all feasible route plans in $G$; that is,
\begin{equation*}
\mathcal{Y} = \left\{ \bm{y} \in \mathbb{Z}^{\lvert E \rvert}_{+} :
\begin{array}{l}
\displaystyle 0 \leq y_{ij} \leq 1 \quad \forall (i,j) \in E: i, j \in V_C, \\
\displaystyle \sum_{j \in V_C} y_{0j} = 2M, \\
\displaystyle \sum_{j \in V : (i,j) \in E} y_{ij} = 2 \quad \forall i \in V_C, \\
\displaystyle \sum_{(i, j) \in E : i, j \in U} y_{ij} \leq \lvert U \rvert - \left\lceil \frac{1}{D} \sum_{i \in U} d_i \right\rceil \quad \forall U \subseteq V_C
\end{array}
\right\}.
\end{equation*}
Similar to the shortest path problem, the $K$-adaptability formulation of the vehicle routing problem only contains second-stage decisions, and as such, the corresponding two-stage robust optimization problem~\eqref{eq:two_stage_ro} is of limited interest in practice. However, the $K$-adaptability problem~\eqref{eq:k_adaptability} has important applications in logistics enterprises, where the time available between observing the travel times in a road network and determining the route plan is limited, or because the drivers must be trained to a small set of route plans that are to be executed daily over the course of a year.

We note that the set $\mathcal{Y}$ represents the so-called \emph{two-index vehicle flow} formulation of the vehicle routing problem, in which the first equation ensures that $M$ vehicles are used; the second set of equations ensures that each customer is visited by exactly one vehicle; while the third set of inequalities ensure that there are no \emph{subtours} disconnected from the depot and that all vehicle capacities are respected.
This formulation is known to be extremely challenging to solve because it consists of an exponential number of inequalities.
For $K > 1$, the corresponding $K$-adaptability problem is naturally even more challenging and it is practically intractable to solve it using the approach described in~\cite{HKW15:rip}. In contrast, the heuristic variant of our algorithm described in Section~\ref{sec:bab_extensions} as well as the heuristic approach of~\cite{BK16:min_max_max_MP} only require the solution of vehicle routing subproblems that are of similar complexity as the associated $1$-adaptability problems.
Therefore, in the following, we only present results using these algorithms.
In both cases, we solved all vehicle routing subproblems using the branch-and-cut algorithm described in~\cite{LLE04:vehicle_routing_branch_and_cut}.

For our numerical experiments, we consider all 49 instances from~\cite{LLE04:vehicle_routing_branch_and_cut} with $N \leq 50$, which are commonly used to benchmark vehicle routing algorithms. We set an overall time limit of 2~hours. For the heuristic variant of our algorithm, we further set a time limit of 10~minutes per vehicle routing subproblem. We note that the heuristic of~\cite{BK16:min_max_max_MP} requires the successful termination of an expensive preprocessing step to determine good $K$-adaptable solutions.
Therefore, to prevent bias in favor of our algorithm, Figure~\ref{figure:vrp_improvement_heuristic} compares the two algorithms only across the 39 instances for which this step terminated successfully.
The figure shows that when the number of policies is small, the $K$-adaptable solutions obtained using our algorithm are about 1\% better than those obtained using the heuristic algorithm of~\cite{BK16:min_max_max_MP}. 
Moreover, the differences in their objective values are relatively higher for larger instances.

\begin{figure}[!htb]
\captionsetup[subfigure]{belowskip=0pt}
\centering
\begin{subfigure}[b]{0.48\textwidth}
\includegraphics[width=\textwidth]{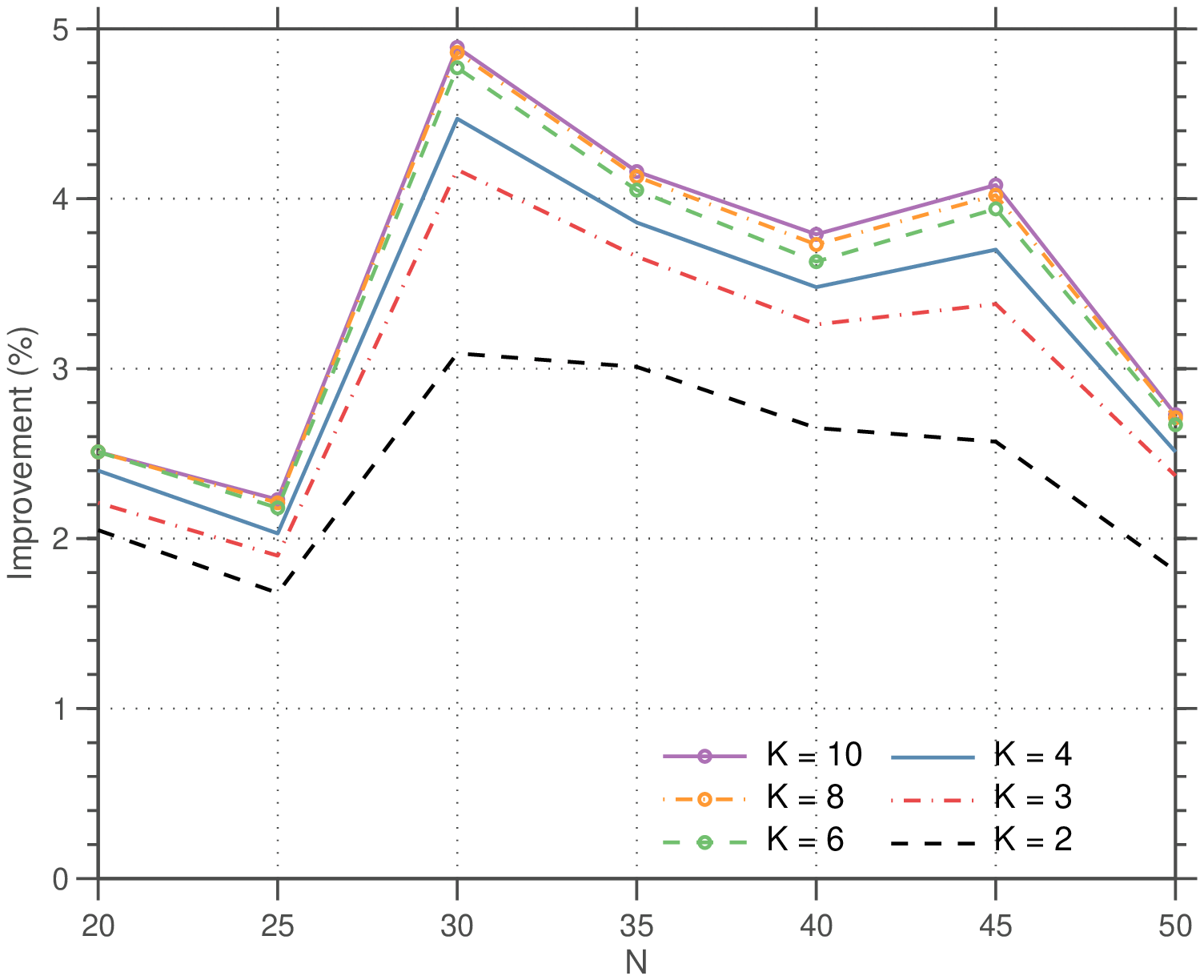}
\caption{}
\label{figure:vrp_improvement_heuristic_bab}
\end{subfigure}~%
\begin{subfigure}[b]{0.48\textwidth}
\includegraphics[width=\textwidth]{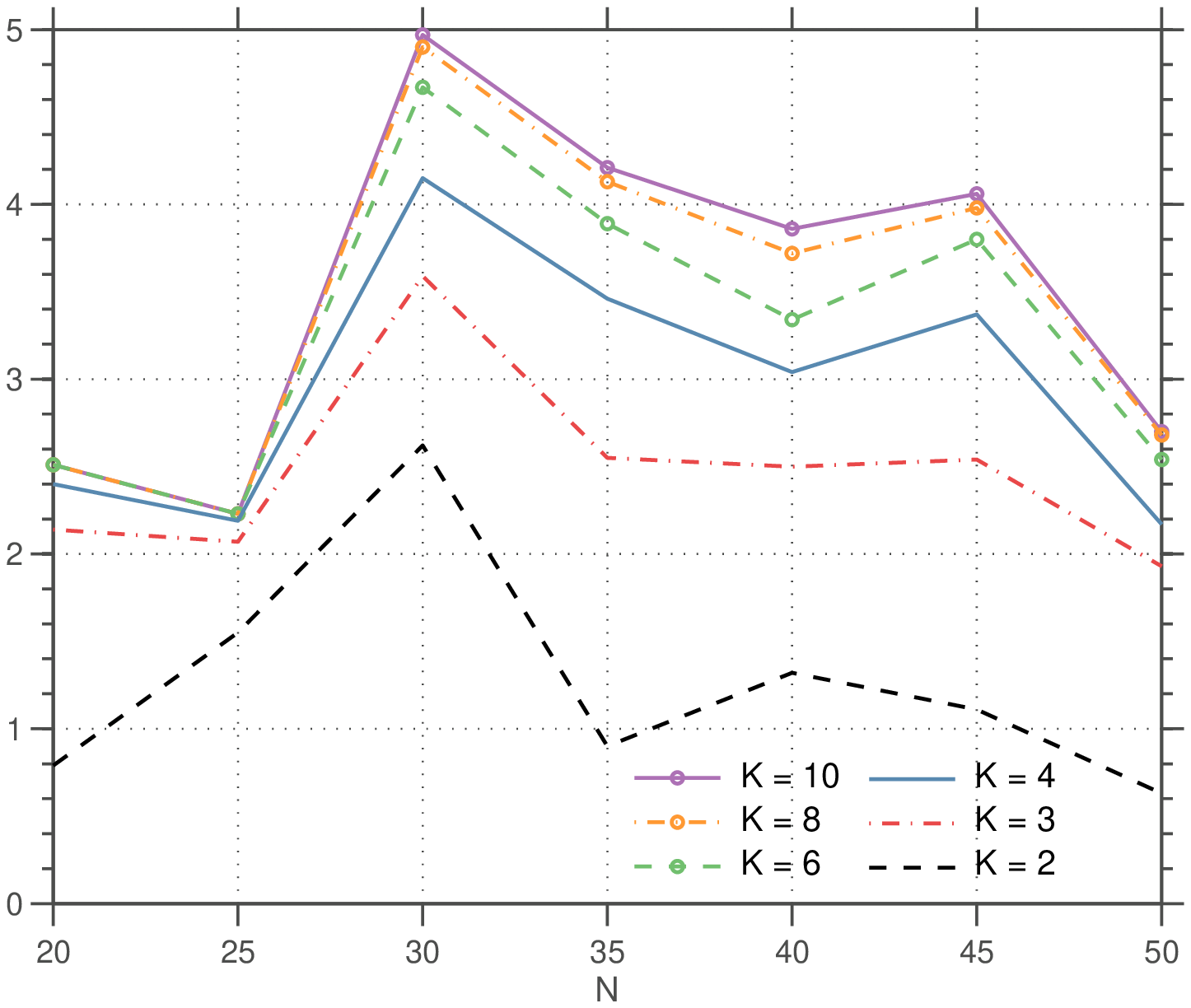}
\caption{}
\label{figure:vrp_improvement_heuristic_mmm}
\end{subfigure}
\caption{Results for the vehicle routing problem. The graphs show the average improvement after 2~hours obtained using the heuristic variant of our algorithm (left) and the heuristic algorithm in~\cite{BK16:min_max_max_MP} (right).} 
\label{figure:vrp_improvement_heuristic}
\end{figure}

\section{Conclusions}\label{sec:conclusions}
In contrast to single-stage robust optimization problems, which are typically solved via monolithic reformulations, there is growing evidence that two-stage and multi-stage robust optimization problems are best solved algorithmically~\cite{BD15:multistage_robust_mio,BG13:multistage_adaptive,Angele:BDRs,PdH15:multistage_adjustable_robust_mio,ZL13:column_and_constraint}. Our findings in this paper appear to confirm this observation, as our proposed branch-and-bound algorithm compares favorably with the reformulations proposed in~\cite{HKW15:rip}. In terms of modeling flexibility, our algorithm can accommodate mixed continuous and discrete decisions in both stages, can incorporate discrete uncertainty, and allows us to model flexible piecewise affine decision rules. At the same time, our numerical results indicate that the algorithmic approach is highly competitive in terms of computational performance as well. From a practical viewpoint, a notable feature of our algorithm is that it admits a lightweight implementation by integrating it into the branch-and-bound schemes of commercial solvers via branch callbacks, while allowing easy modification as a heuristic for large-scale instances.

Our results open up multiple fruitful avenues for future research. The scope of the presented algorithm can be broadened further by generalizing it to two-stage \emph{distributionally} robust optimization problems, where the uncertain problem parameters are modeled as random variables following a probability distribution that is only partially known. More ambitiously, it would be instructive to explore how the concept of $K$-adaptability, as well as our proposed branch-and-bound algorithm, extend to dynamic robust optimization problems with more than two stages.

\section*{Acknowledgments}
Anirudh Subramanyam and Chrysanthos E.~Gounaris gratefully acknowledge support from the National Science Foundation, award number CMMI-1434682. Wolfram Wiesemann gratefully acknowledges funding from the EPSRC grants EP/M028240/1, EP/M027856/1 and EP/N020030/1. Anirudh Subramanyam also acknowledges support from the John and Claire Bertucci Graduate Fellowship program at Carnegie Mellon University.

\bibliographystyle{plain}

\newpage

\begin{appendices}
\numberwithin{prop}{section}

\section{Analysis of the Two-Stage Robust Optimization Problem}

To analyze problem~\eqref{eq:two_stage_ro}, we equivalently rewrite it as
\begin{equation}\label{eq:two_stage_ro_explicit}
\begin{array}{r@{\quad}l@{\quad}r@{\; = \;}l}
\displaystyle \inf_{\bm{x} \in \mathcal{X}} \; \bm{c}^\top \bm{x} + \mathcal{Q} (\bm{x})
& \displaystyle \text{with}
& \displaystyle \mathcal{Q} (\bm{x}) & \displaystyle \sup_{\bm{\xi} \in \Xi} \; Q (\bm{x}, \bm{\xi}), \\
& \displaystyle \text{where}
& \displaystyle Q (\bm{x}, \bm{\xi}) & \displaystyle \inf_{\bm{y} \in \mathcal{Y}} \left\{ \bm{d} (\bm{\xi})^\top \bm{y} \, : \, \bm{T} (\bm{\xi}) \bm{x} + \bm{W} (\bm{\xi}) \bm{y} \leq \bm{h} (\bm{\xi}) \right\}
\end{array}
\tag{\ref{eq:two_stage_ro}$'$}
\end{equation}
for $\mathcal{Q} : \mathcal{X} \mapsto \mathbb{R} \cup \{ + \infty \}$ and $Q : \mathcal{X} \times \Xi \mapsto \mathbb{R} \cup \{ + \infty \}$.

We first investigate whether the infima and the supremum in problem~\eqref{eq:two_stage_ro_explicit} are attained.

\begin{prop}[Continuity]\label{prop:continuity_two_stage_ro}
Problem~\eqref{eq:two_stage_ro_explicit} satisfies the following properties.
\begin{enumerate}
\item[(i)] The problem $Q (\bm{x}, \bm{\xi})$ attains its infimum, if it is feasible.
\item[(ii)] The problem $\mathcal{Q} (\bm{x})$ attains its supremum, if only the objective function is uncertain. Otherwise, $\mathcal{Q} (\bm{x})$ does not necessarily attain its supremum, even if only the constraint right-hand sides are uncertain.
\item[(iii)] The problem~\eqref{eq:two_stage_ro_explicit} attains its infimum, if it is feasible.
\end{enumerate}
\end{prop}

\begin{proof}
The first statement holds since the problem $Q (\bm{x}, \bm{\xi})$ minimizes an affine function in $\bm{y}$ over the intersection of the compact set $\mathcal{Y}$ and the polyhedron $\{ \bm{y} \in \mathbb{R}^{N_2} \, : \, \bm{W} (\bm{\xi}) \bm{y} \leq \bm{h} (\bm{\xi}) - \bm{T} (\bm{\xi}) \bm{x} \}$.

In view of the second statement, assume first that only the objective function in problem~\eqref{eq:two_stage_ro_explicit} is uncertain; that is, $\bm{T} (\bm{\xi}) = \bm{T}$, $\bm{W} (\bm{\xi}) = \bm{W}$ and $\bm{h} (\bm{\xi}) = \bm{h}$ for all $\bm{\xi} \in \Xi$. Then the inner problem $Q (\bm{x}, \bm{\xi})$ simplifies to 
$Q (\bm{x}, \bm{\xi}) = \inf \{ \bm{d} (\bm{\xi})^\top \bm{y} \, : \, \bm{y} \in \mathcal{Y}_{\bm{W}} \}$ with $\mathcal{Y}_{\bm{W}} = \left\{ \bm{y} \in \mathcal{Y} \, : \, \bm{W} \bm{y} \leq \bm{h} - \bm{T} \bm{x} \right\}$. If $\mathcal{Y}_{\bm{W}} = \emptyset$, then $Q (\bm{x}, \bm{\xi}) = + \infty$ for all $\bm{\xi} \in \Xi$, and any $\bm{\xi} \in \Xi$ attains the supremum. Otherwise, we have $Q (\bm{x}, \bm{\xi}) = \inf \{ \bm{d} (\bm{\xi})^\top \bm{y} \, : \, \bm{y} \in \text{ext conv } \mathcal{Y}_{\bm{W}} \}$, where $\text{ext conv } \mathcal{Y}_{\bm{W}}$ denotes the set of (finitely many) extreme points of the convex hull of $\mathcal{Y}_{\bm{W}}$. We thus conclude from~\cite[Proposition~1.26]{RW09:variational_analysis} that $Q (\bm{x}, \bm{\xi})$ is upper semicontinuous in $\bm{\xi}$ for every fixed $\bm{x}$, and~\cite[Theorem~1.9]{RW09:variational_analysis} then implies that $\mathcal{Q} (\bm{x})$ attains its supremum since $\Xi$ is a compact set.

Assume now that only the constraint right-hand sides in problem~\eqref{eq:two_stage_ro_explicit} are allowed to be uncertain, and consider the following instance of problem $\mathcal{Q} (\bm{x})$:
\begin{equation*}
\sup_{\xi \in [0,1]} \; \inf_{\substack{\tau \in \mathbb{R}, \\ y \in \{ 0, 1 \}}} \left\{ \tau - y \, : \, y \leq \xi \leq \tau \right\}
\end{equation*}
The inner minimization problem is optimized by $\tau^\star = \xi$ and $y^\star = \left \lfloor \xi \right \rfloor$. The supremum is $1$, and it is approached by the sequence of feasible solutions $\xi \uparrow 1$. Note, however, that the limit point $\xi^\star = 1$ of this sequence results in an objective function value of $0$.

As for the third statement, we first note that the extended real-valued function
\begin{equation*}
Q (\bm{x}, \bm{\xi}, \bm{y}) =
\begin{cases}
\bm{d} (\bm{\xi})^\top \bm{y} & \text{if } \bm{y} \in \mathcal{Y} \text{ and } \bm{T} (\bm{\xi}) \bm{x} + \bm{W} (\bm{\xi}) \bm{y} \leq \bm{h} (\bm{\xi}), \\
+ \infty & \text{otherwise}
\end{cases}
\end{equation*}
is lower semicontinuous in $(\bm{x}, \bm{\xi}, \bm{y})$ since the set $\{ (\bm{x}, \bm{\xi}, \bm{y}) \in \mathbb{R}^{N_1} \times \mathbb{R}^{N_p} \times \mathcal{Y} \, : \, \bm{T} (\bm{\xi}) \bm{x} + \bm{W} (\bm{\xi}) \bm{y} \leq \bm{h} (\bm{\xi}) \}$ is closed. From~\cite[Theorem~1.17]{RW09:variational_analysis} and~\cite[Proposition~1.26]{RW09:variational_analysis} we conclude that the lower semicontinuity is preserved by the partial minimization over $\bm{y}$ and the partial maximization over $\bm{\xi}$. By~\cite[Theorem~1.9]{RW09:variational_analysis} and the fact that $\mathcal{X}$ is compact we can then conclude that the problem~\eqref{eq:two_stage_ro_explicit} attains its infimum whenever it is feasible.
\end{proof}

It is shown in~\cite[Example~1]{HKW15:rip} that $\mathcal{Q} (\bm{x})$ may not attain its supremum if both the objective function and the constraint right-hand sides in problem~\eqref{eq:two_stage_ro_explicit} are uncertain. Proposition~\ref{prop:continuity_two_stage_ro}~(ii) strengthens this observation to instances of problem~\eqref{eq:two_stage_ro_explicit} where only the constraint right-hand sides are uncertain.
We now consider the convexity properties of problem~\eqref{eq:two_stage_ro_explicit}.

\begin{prop}[Convexity]\label{prop:convexity_two_stage_ro}
Problem~\eqref{eq:two_stage_ro_explicit} satisfies the following properties.
\begin{enumerate}
\item[(i)] The problem $Q (\bm{x}, \bm{\xi})$ is convex, if $\mathcal{Y}$ is convex.
\item[(ii)] The problem $\mathcal{Q} (\bm{x})$ is convex, if $\Xi$ is convex and only the objective function is uncertain, irrespective of $\mathcal{Y}$. Otherwise, $\mathcal{Q} (\bm{x})$ is typically not convex, even if $\Xi$ and $\mathcal{Y}$ are convex and only the constraint right-hand sides are uncertain.
\item[(iii)] The problem~\eqref{eq:two_stage_ro_explicit} is convex, if $\mathcal{X}$ and $\mathcal{Y}$ are convex, irrespective of $\Xi$.
\end{enumerate}
\end{prop}

\begin{proof}
The first statement directly follows from the linearity of the objective function and the convexity of the feasible set.

In view of the second statement, assume first that $\Xi$ is convex and only the objective function in problem~\eqref{eq:two_stage_ro_explicit} is uncertain; that is, $\bm{T} (\bm{\xi}) = \bm{T}$, $\bm{W} (\bm{\xi}) = \bm{W}$ and $\bm{h} (\bm{\xi}) = \bm{h}$ for all $\bm{\xi} \in \Xi$. Then the inner problem $Q (\bm{x}, \bm{\xi})$ simplifies to $Q (\bm{x}, \bm{\xi}) = \inf \{ \bm{d} (\bm{\xi})^\top \bm{y} \, : \, \bm{y} \in \mathcal{Y}_{\bm{W}} \}$ with $\mathcal{Y}_{\bm{W}} = \left\{ \bm{y} \in \mathcal{Y} \, : \, \bm{W} \bm{y} \leq \bm{h} - \bm{T} \bm{x} \right\}$. Assume that $\mathcal{Y}_{\bm{W}} \neq \emptyset$; the other case is trivial. Then we have $Q (\bm{x}, \bm{\xi}) = \inf \{ \bm{d} (\bm{\xi})^\top \bm{y} \, : \, \bm{y} \in \text{ext conv } \mathcal{Y}_{\bm{W}} \}$, where $\text{ext conv } \mathcal{Y}_{\bm{W}}$ denotes the set of (finitely many) extreme points of the convex hull of $\mathcal{Y}_{\bm{W}}$. We thus conclude from~\cite[Proposition~2.9]{RW09:variational_analysis} that $Q (\bm{x}, \bm{\xi})$ is concave in $\bm{\xi}$ for every fixed $\bm{x}$, which implies that $\mathcal{Q} (\bm{x})$ is a convex optimization problem.

Assume now that the constraint right-hand sides in problem~\eqref{eq:two_stage_ro_explicit} are allowed to be uncertain, and consider the following instance of problem $\mathcal{Q} (\bm{x})$:
\begin{equation*}
\sup_{\xi \in [-1, 1]} \; \inf_{y \in \mathbb{R}} \left\{ y : \ y \geq \xi, \;\; y \geq -\xi \right\}
\end{equation*}
Since the inner minimization problem is optimized by $y^\star = | \xi |$, the problem maximizes the (convex) $1$-norm of $\xi$ over the interval $[-1, 1]$, which amounts to a non-convex optimization problem.

As for the third statement, assume that $\mathcal{X}$ and $\mathcal{Y}$ are convex, and consider the function
\begin{equation*}
Q (\bm{x}, \bm{\xi}, \bm{y}) =
\begin{cases}
\bm{d} (\bm{\xi})^\top \bm{y} & \text{if } \bm{y} \in \mathcal{Y} \text{ and } \bm{T} (\bm{\xi}) \bm{x} + \bm{W} (\bm{\xi}) \bm{y} \leq \bm{h} (\bm{\xi}), \\
+ \infty & \text{otherwise}.
\end{cases}
\end{equation*}
The function $Q (\bm{x}, \bm{\xi}, \bm{y})$ is convex in $(\bm{x}, \bm{y})$ for every fixed $\bm{\xi} \in \Xi$. From~\cite[Proposition~2.22]{RW09:variational_analysis} and~\cite[Proposition~2.9]{RW09:variational_analysis} we conclude that the convexity is preserved by the partial minimization over $\bm{y}$ and the partial maximization over $\bm{\xi}$. The problem~\eqref{eq:two_stage_ro_explicit} thus minimizes the sum of two convex functions $\bm{c}^\top \bm{x}$ and $\mathcal{Q} (\bm{x})$ over the convex set $\mathcal{X}$, which is a convex optimization problem.
\end{proof}

One readily verifies that the third statement in Proposition~\ref{prop:convexity_two_stage_ro} does not hold if $\mathcal{Y}$ is not convex.
%
%
We now investigate under which conditions we can solve problem~\eqref{eq:two_stage_ro_explicit} in polynomial time.

\begin{prop}[Tractability]\label{prop:tractability_two_stage_ro}
Problem~\eqref{eq:two_stage_ro_explicit} satisfies the following properties.
\begin{enumerate}
\item[(i)] The problem $Q (\bm{x}, \bm{\xi})$ can be solved in polynomial time, if $\mathcal{Y}$ is convex.
\item[(ii)] The problem $\mathcal{Q} (\bm{x})$ can be solved in polynomial time, if $\Xi$ and $\mathcal{Y}$ are convex and only the objective function is uncertain. Otherwise,  $\mathcal{Q} (\bm{x})$ is strongly NP-hard, even if $\Xi$ and $\mathcal{Y}$ are convex and only the constraint right-hand sides are uncertain.
\item[(iii)] The problem~\eqref{eq:two_stage_ro_explicit} can be solved in polynomial time, if $\mathcal{X}$, $\mathcal{Y}$ and $\Xi$ are convex and only the objective function is uncertain. Otherwise, the problem is strongly NP-hard, even if $\mathcal{X}$, $\mathcal{Y}$ and $\Xi$ are convex and only the constraint right-hand sides are uncertain.
\end{enumerate}
\end{prop}

\begin{proof}
If $\mathcal{Y}$ is convex, then the problem $Q (\bm{x}, \bm{\xi})$ amounts to a linear program that can be solved in polynomial time. This shows the first statement.

The first part of the second statement follows from the proof of the first part of the third statement below if we fix $\mathcal{X} = \{ \bm{x} \}$ and $\bm{c} = \bm{0}$ in problem~\eqref{eq:two_stage_ro_explicit}.
For the second part of the second statement, we recall the strongly NP-hard 0/1 Integer Programming (IP) feasibility problem~\cite{GJ79:ComputersIntractability}: \\[-2mm]

\fbox{\parbox{15cm}{ {\centering \textsc{0/1 Integer Programming Feasibility.}\\[-8mm]}
\vphantom{a}~~~\textbf{Instance.} Given are $\bm{A} \in \mathbb{Z}^{R \times N_p}$ and $\bm{b} \in \mathbb{Z}^R$. \\
\vphantom{a}~~~\textbf{Question.} Is there a vector $\bm{\xi} \in \left\{ 0, 1 \right\}^{N_p}$ such that $\bm{A} \bm{\xi} \leq \bm{b}$?
}} \\

We show that the IP feasibility problem has an affirmative answer if and only if the problem
\begin{equation}\label{eq:two_stage_ro_intractable_problem}
\sup \left\{ \inf \left\{ \mathbf{e}^\top \bm{y} \, : \, \bm{y} \in \mathbb{R}^{N_p}, \;\; \bm{y} \geq \bm{\xi} - \frac{1}{2} \mathbf{e}, \;\; \bm{y} \geq \frac{1}{2} \mathbf{e} - \bm{\xi} \right\} \, : \, \bm{\xi} \in [0, 1]^{N_p}, \;\; \bm{A} \bm{\xi} \leq \bm{b} \right\}
\end{equation}
has an optimal value of $N_p / 2$. Note that~\eqref{eq:two_stage_ro_intractable_problem} can be interpreted as an instance of $\mathcal{Q} (\bm{x})$. For any fixed $\bm{\xi}$, the infimum in~\eqref{eq:two_stage_ro_intractable_problem} evaluates to $\left \lVert \bm{\xi} - \mathbf{e} / 2 \right \rVert_1$. Thus, the optimal value of~\eqref{eq:two_stage_ro_intractable_problem} is equal to $N_p / 2$ if and only if there is $\bm{\xi} \in [0, 1]^{N_p}$ satisfying $\left \lVert \bm{\xi} - \mathbf{e} / 2 \right \rVert_1 = N_p / 2$ and $\bm{A} \bm{\xi} \leq \bm{b}$. The statement now follows from the fact that $\left \lVert \bm{\xi} - \mathbf{e} / 2 \right \rVert_1 = N_p / 2$ if and only if $\bm{\xi} \in \{ 0, 1 \}^{N_p}$.

In view of the first part of the third statement, the proof of Proposition~\ref{prop:static_two_stage_ro} below shows that, under the stated assumptions, problem~\eqref{eq:two_stage_ro_explicit} reduces to a single-stage robust optimization problem. Dualizing the inner maximization problem in that single-stage reformulation allows us to solve the overall problem in polynomial time as a linear program, see \cite{BTEGN09:rob_opt}. Finally, the second part of the third statement follows from the proof of the second part of the second statement if we amend problem~\eqref{eq:two_stage_ro_intractable_problem} by appending an outer infimum over the singleton set $\mathcal{X} = \{ 1 \}$.
\end{proof}

The NP-hardness of the two-stage robust optimization problem~\eqref{eq:two_stage_ro_explicit} has previously been established for the case where $\mathcal{X}$, $\mathcal{Y}$ and $\Xi$ are convex and only the constraint right-hand sides are uncertain \cite[Theorem~3.5]{Gus02:aro_thesis}. We provide an alternative proof in Proposition~\ref{prop:tractability_two_stage_ro}~(iii) to facilitate a self-contained comparison with the $K$-adaptability problem~\eqref{eq:k_adaptability} in Proposition~\ref{prop:tractability_k_adaptability}~(ii) below. Moreover, it is shown in~\cite[Theorem~3.3]{Gus02:aro_thesis} that problem~\eqref{eq:two_stage_ro_explicit} can be solved in polynomial time whenever $\mathcal{X}$ and $\mathcal{Y}$ are convex, $\bm{d}$ and $\bm{W}$ are deterministic and $\Xi$ is described in terms of its extreme points.
We close our analysis of the two-stage robust optimization problem~\eqref{eq:two_stage_ro_explicit} with a special case where it reduces to a single-stage robust optimization problem.

\begin{prop}[Reduction to Static Problem]\label{prop:static_two_stage_ro}
The problem~\eqref{eq:two_stage_ro_explicit} reduces to a static robust optimization problem where $\mathcal{Y}$ is replaced with its convex hull, if $\Xi$ is convex and only the objective function is uncertain, irrespective of $\mathcal{X}$ and $\mathcal{Y}$.
\end{prop}

\begin{proof}
Assume that $\Xi$ is convex and that only the objective function in problem~\eqref{eq:two_stage_ro_explicit} is uncertain; that is, $\bm{T} (\bm{\xi}) = \bm{T}$, $\bm{W} (\bm{\xi}) = \bm{W}$ and $\bm{h} (\bm{\xi}) = \bm{h}$ for all $\bm{\xi} \in \Xi$. Problem~\eqref{eq:two_stage_ro_explicit} then simplifies to
\begin{equation*}
\inf_{\bm{x} \in \mathcal{X}} \; \sup_{\bm{\xi} \in \Xi} \; \inf_{\bm{y} \in \text{conv}\, \mathcal{Y}}
\left\{ \bm{c}^\top \bm{x} + \bm{d} (\bm{\xi})^\top \bm{y} \, : \, \bm{T} \bm{x} + \bm{W} \bm{y} \leq \bm{h} \right\},
\end{equation*}
where the replacement of the second-stage feasible region $\mathcal{Y}$ with its convex hull, $\text{conv } \mathcal{Y}$, is justified since the inner minimization has a linear objective function. The classical minimax theorem now allows us to exchange the order of the inner two operators:
\begin{equation*}
\inf_{\substack{\bm{x} \in \mathcal{X}, \\ \bm{y} \in \text{conv}\, \mathcal{Y}}} \left\{ \sup_{\bm{\xi} \in \Xi}
\left\{ \bm{c}^\top \bm{x} + \bm{d} (\bm{\xi})^\top \bm{y} \right\} \, : \, \bm{T} \bm{x} + \bm{W} \bm{y} \leq \bm{h} \right\}.
\end{equation*}
This problem is readily recognized as a single-stage robust optimization problem.

We emphasize that the previous argument requires $\Xi$ to be convex. Indeed, we have
\begin{equation*}
-1
\; = \;
\sup_{\xi \in \{ -1, 1 \}} \; \inf_{y \in [-1, 1]} \, \xi y
\; \neq \;
\inf_{y \in [-1, 1]} \; \sup_{\xi \in \{ -1, 1 \}} \, \xi y
\; = \;
0,
\end{equation*}
and we cannot establish equivalence by replacing $\Xi$ with $\text{conv } \Xi = [-1, 1]$ in the second optimization problem either.
\end{proof}

A related result was established in~\cite[Theorem~2.1]{BTGGN04:adjustable}, where it is shown that the two-stage robust optimization problem~\eqref{eq:two_stage_ro_explicit} reduces to a single-stage robust optimization problem if $\mathcal{X}$, $\mathcal{Y}$ and $\Xi$ are convex and the uncertain parameters can be partitioned into subsets such that each constraint is only affected by the parameters of one such subset, and no two constraints are affected by parameters of the same subset.

\section{Analysis of the $K$-Adaptability Problem}

To analyze problem~\eqref{eq:k_adaptability}, we equivalently rewrite it as
\begin{equation}\label{eq:k_adaptability_explicit}
\begin{array}{r@{\quad}l@{\quad}r@{\; = \;}l}
\displaystyle \inf_{\substack{\bm{x} \in \mathcal{X}, \\ \bm{y} \in \mathcal{Y}^k}} \; \bm{c}^\top \bm{x} + \mathcal{Q} (\bm{x}, \bm{y})
& \displaystyle \text{with}
& \displaystyle \mathcal{Q} (\bm{x}, \bm{y}) & \displaystyle \sup_{\bm{\xi} \in \Xi} \; Q (\bm{x}, \bm{y}, \bm{\xi}), \\
& \displaystyle \text{where}
& \displaystyle Q (\bm{x}, \bm{y}, \bm{\xi}) & \displaystyle \inf_{k \in \mathcal{K}} \left\{ \bm{d} (\bm{\xi})^\top \bm{y}_k \, : \, \bm{T} (\bm{\xi}) \bm{x} + \bm{W} (\bm{\xi}) \bm{y}_k \leq \bm{h} (\bm{\xi}) \right\}
\end{array}
\tag{\ref{eq:k_adaptability}$'$}
\end{equation}
for $\mathcal{Q} : \mathcal{X} \times \mathcal{Y}^K \mapsto \mathbb{R} \cup \{ + \infty \}$ and $Q : \mathcal{X} \times \mathcal{Y}^K \times \Xi \mapsto \mathbb{R} \cup \{ + \infty \}$.

In analogy to Proposition~\ref{prop:continuity_two_stage_ro}, we first investigate whether~\eqref{eq:k_adaptability_explicit} attains its infima and supremum.

\begin{prop}[Continuity]\label{prop:continuity_k_adaptability}
Problem~\eqref{eq:k_adaptability_explicit} satisfies the following properties.
\begin{enumerate}
\item[(i)] The problem $\mathcal{Q} (\bm{x}, \bm{y})$ attains its supremum, if only the objective function or only the constraints are uncertain. Otherwise, $\mathcal{Q} (\bm{x}, \bm{y})$ does not necessarily attain its supremum, even if the constraint left-hand sides are not uncertain.
\item[(ii)] The problem~\eqref{eq:k_adaptability_explicit} attains its infimum, if it is feasible.
\end{enumerate}
\end{prop}

\begin{proof}
In view of the first statement, assume that only the objective function in problem~\eqref{eq:k_adaptability_explicit} is uncertain; that is, $\bm{T} (\bm{\xi}) = \bm{T}$, $\bm{W} (\bm{\xi}) = \bm{W}$ and $\bm{h} (\bm{\xi}) = \bm{h}$ for all $\bm{\xi} \in \Xi$. Then the inner problem $Q (\bm{x}, \bm{y}, \bm{\xi})$ simplifies to $Q (\bm{x}, \bm{y}, \bm{\xi}) = \inf \{ \bm{y}_k^\top \bm{d} (\bm{\xi}) \, : \, k \in \mathcal{K}_{\bm{W}} \}$, where $\mathcal{K}_{\bm{W}} = \{ k \in \mathcal{K} \, : \, \bm{W} \bm{y}_k \leq \bm{h} - \bm{T} \bm{x} \}$. If $\mathcal{K}_{\bm{W}} = \emptyset$, then $Q (\bm{x}, \bm{y}, \bm{\xi}) = + \infty$ for all $\bm{\xi} \in \Xi$, and any $\bm{\xi} \in \Xi$ attains the supremum. Otherwise, $Q (\bm{x}, \bm{y}, \bm{\xi})$ is readily verified to be continuous in $\bm{\xi}$ for every fixed $\bm{x}$ and $\bm{y}$, and~\cite[Theorem~1.9]{RW09:variational_analysis} then implies that $\mathcal{Q} (\bm{x}, \bm{y})$ attains its supremum since $\Xi$ is a compact set.

Assume now that only the constraints in problem~\eqref{eq:k_adaptability_explicit} are uncertain. Then $Q (\bm{x}, \bm{y}, \bm{\xi})$ simplifies to $Q (\bm{x}, \bm{y}, \bm{\xi}) = \inf \{ \bm{d}^\top \bm{y}_k \, : \, k \in \mathcal{K}, \; \bm{\xi} \in \Xi_k \}$, where $\Xi_k = \{ \bm{\xi} \in \Xi \, : \, \bm{T} (\bm{\xi}) \bm{x} + \bm{W} (\bm{\xi}) \bm{y}_k \leq \bm{h} (\bm{\xi}) \}$. If $\bigcup_{k \in \mathcal{K}} \Xi_k \neq \Xi$, then $Q (\bm{x}, \bm{y}, \bm{\xi}) = + \infty$ for any $\bm{\xi} \in \Xi \setminus \bigcup_{k \in \mathcal{K}} \Xi_k$ and $\mathcal{Q} (\bm{x}, \bm{y})$ attains its supremum. Otherwise, $\bigcup_{k \in \mathcal{K}} \Xi_k = \Xi$, and the supremum in $\mathcal{Q} (\bm{x}, \bm{y})$ is attained by any $\bm{\xi} \in \bigcap_{k \in \mathcal{K}^\star} \Xi_k$, where
\begin{equation*}
\mathcal{K}^\star \in \mathop{\arg \max}_{\mathcal{K}' \subseteq \mathcal{K}}
\left\{
\min_{k \in \mathcal{K}'} \bm{d}^\top \bm{y}_k \, : \,
\bigcap_{k \in \mathcal{K}'} \Xi_k \neq \emptyset
\right\}.
\end{equation*}

Finally, assume that both the objective function and the constraint right-hand sides in problem~\eqref{eq:k_adaptability_explicit} are allowed to be uncertain, and consider the following instance of problem $\mathcal{Q} (\bm{x}, \bm{y})$:
\begin{equation*}
\sup_{\xi \in [0,1]} \; \inf_{k \in \{ 1, 2 \}} \left\{ \xi - y_k \, : \, y_k \leq \xi \right\}
\qquad \text{with } y_1 = 1 \text{ and } y_2 = 0.
\end{equation*}
The inner minimization problem is optimized by $k^\star = 1$ if $\xi = 1$ and $k^\star = 2$ otherwise. The supremum is $1$, and it is approached by the sequence of feasible solutions $\xi \uparrow 1$. Note, however, that the limit point $\xi^\star = 1$ of this sequence results in an objective function value of $0$.

As for the second statement, we first note that each extended real-valued function
\begin{equation*}
Q_k (\bm{x}, \bm{y}, \bm{\xi}) =
\begin{cases}
\bm{d} (\bm{\xi})^\top \bm{y}_k & \text{if } \bm{y}_k \in \mathcal{Y} \text{ and } \bm{T} (\bm{\xi}) \bm{x} + \bm{W} (\bm{\xi}) \bm{y}_k \leq \bm{h} (\bm{\xi}), \\
+ \infty & \text{otherwise},
\end{cases}
\qquad k \in \mathcal{K},
\end{equation*}
is lower semicontinuous in $(\bm{x}, \bm{y}, \bm{\xi})$ since the sets $\{ (\bm{x}, \bm{y}, \bm{\xi}) \in \mathbb{R}^{N_1} \times (\mathbb{R}^{N_2})^K \times \mathbb{R}^{N_p} \, : \, \bm{y}_k \in \mathcal{Y}, \; \bm{T} (\bm{\xi}) \bm{x} + \bm{W} (\bm{\xi}) \bm{y}_k \leq \bm{h} (\bm{\xi}) \}$ are closed. From~\cite[Proposition~1.26]{RW09:variational_analysis} we conclude that the lower semicontinuity is preserved by the partial minimization over $k$ and the partial maximization over $\bm{\xi}$. Due to~\cite[Theorem~1.9]{RW09:variational_analysis} the problem~\eqref{eq:k_adaptability_explicit} then attains its infimum whenever it is feasible.
\end{proof}

Similar to Proposition~\ref{prop:convexity_two_stage_ro}, we now consider the convexity properties of problem~\eqref{eq:k_adaptability_explicit}.

\begin{prop}[Convexity]\label{prop:convexity_k_adaptability}
Problem~\eqref{eq:k_adaptability_explicit} satisfies the following properties.
\begin{enumerate}
\item[(i)] The problem $\mathcal{Q} (\bm{x}, \bm{y})$ is convex, if $\Xi$ is convex and only the objective function is uncertain. Otherwise, $\mathcal{Q} (\bm{x}, \bm{y})$ is typically not convex, even if $\Xi$ is convex and only the constraint right-hand sides are uncertain.
\item[(ii)] Problem~\eqref{eq:k_adaptability_explicit} is typically not convex, even if $\mathcal{X}$ and $\mathcal{Y}$ are convex and $\Xi$ is a singleton.
\end{enumerate}
\end{prop}

\begin{proof}
In view of the first statement, assume that $\Xi$ is convex and that only the objective function in problem~\eqref{eq:k_adaptability_explicit} is uncertain; that is, $\bm{T} (\bm{\xi}) = \bm{T}$, $\bm{W} (\bm{\xi}) = \bm{W}$ and $\bm{h} (\bm{\xi}) = \bm{h}$ for all $\bm{\xi} \in \Xi$. Then the inner problem $Q (\bm{x}, \bm{y}, \bm{\xi})$ simplifies to $Q (\bm{x}, \bm{y}, \bm{\xi}) = \inf \{ \bm{y}_k^\top \bm{d} (\bm{\xi}) \, : \, k \in \mathcal{K}_{\bm{W}} \}$, where $\mathcal{K}_{\bm{W}} = \{ k \in \mathcal{K} \, : \, \bm{W} \bm{y}_k \leq \bm{h} - \bm{T} \bm{x} \}$. Assume that $\mathcal{K}_{\bm{W}} \neq \emptyset$; the other case is trivial. Then $Q (\bm{x}, \bm{y}, \bm{\xi})$ is a piecewise-affine concave function in $\bm{\xi}$ for every fixed $\bm{x}$ and $\bm{y}$. We thus conclude that $\mathcal{Q} (\bm{x}, \bm{y})$ is a convex optimization problem as it maximizes a concave function over a convex set.

Assume now that the constraint right-hand sides in problem~\eqref{eq:k_adaptability_explicit} are allowed to be uncertain, and consider the following instance of problem $\mathcal{Q} (\bm{x}, \bm{y})$:
\begin{equation*}
\sup_{\xi \in [-1, 1]} \; \inf_{k \in \mathcal{K}} \left\{ y_k \, : \, y_k \geq \xi \right\}
\qquad
\text{with } y_1 = 1 \text{ and } y_2 = 0.
\end{equation*}
The inner minimization problem is optimized by $k^\star = 1$, if $\xi > 0$, and $k^\star = 2$, otherwise. The outer maximization problem thus maximizes the non-convex function $\mathbb{I}{[\xi > 0]}$ over the interval $[-1, 1]$, which amounts to a non-convex optimization problem.

In view of the second statement, consider the following problem instance:
\begin{equation*}
\inf_{y_1, y_2 \in [-1, 1]} \; \sup_{\xi \in \{1\}} \; \inf_{k \in \{ 1, 2 \}} \; \xi \, y_k
\end{equation*}
The problem attains the objective value $-1$ for $(y_1, y_2) \in \{ (-1, 1), \, (1, -1) \}$, but it attains the larger objective value of $0$ for $(y_1, y_2) = 1/2 \cdot (-1, 1) + 1/2 \cdot (1, -1) = (0, 0)$.
\end{proof}

In analogy to Proposition~\ref{prop:tractability_two_stage_ro}, we now investigate under which conditions problem~\eqref{eq:k_adaptability_explicit} is tractable.

\begin{prop}[Tractability]\label{prop:tractability_k_adaptability}
Problem~\eqref{eq:k_adaptability_explicit} satisfies the following properties.
\begin{enumerate}
\item[(i)] The problem $\mathcal{Q} (\bm{x}, \bm{y})$ can be solved in polynomial time, if $\Xi$ is convex and only the objective function is uncertain. Otherwise,  $\mathcal{Q} (\bm{x}, \bm{y})$ is strongly NP-hard, even if $\Xi$ is convex and only the constraint right-hand sides are uncertain.
\item[(ii)] The problem~\eqref{eq:k_adaptability_explicit} can be solved in polynomial time, if $\mathcal{X}$, $\mathcal{Y}$ and $\Xi$ are convex and only the objective function is uncertain. Otherwise, the problem is strongly NP-hard, even if $\mathcal{X}$, $\mathcal{Y}$ and $\Xi$ are convex and only the constraint right-hand sides are uncertain.
\end{enumerate}
\end{prop}

\begin{proof}
The first part of the first statement follows directly from \cite[Observation~2]{HKW15:rip}.

In view of the second part of the first statement, we consider the following variant of the IP feasibility problem: \\[-2mm]

\fbox{\parbox{15.5cm}{ {\centering \textsc{Approximate 0/1 Integer Programming Feasibility.}\\[-8mm]}
\vphantom{a}~~~\textbf{Instance.} Given are $\bm{A} \in \mathbb{Z}^{R \times N_p}$, $\bm{b} \in \mathbb{Z}^R$. \\
\vphantom{a}~~~\textbf{Question.} Is there $\bm{\xi} \in ([0, \epsilon) \cup (1 - \epsilon, 1])^{N_p}$, $\epsilon = \big( \min_r \sum_q | A_{rq}| \big)^{-1}$, such that $\bm{A} \bm{\xi} \leq \bm{b}$?
}} \\[4mm]
It follows from \cite[Lemma~2]{WKS14:DRCO} that the approximate IP feasibility problem is strongly NP-hard.
We claim that the approximate IP feasibility problem has an affirmative answer if and only if
\begin{equation}\label{eq:approx_ip_feasibility}
\mathcal{Q} (\bm{x}, \bm{y}) \; = \;
\sup_{\bm{\xi} \in \Xi} \; \inf_{k \in \mathcal{K}} \left\{ y_k^0 \, : \, \underline{\bm{y}}_k \leq \bm{\xi} \leq \overline{\bm{y}}_k \right\}
\; = \; 1,
\end{equation}
where $\Xi = \{ \bm{\xi} \in [0, 1]^{N_p} \, : \, \bm{A} \bm{\xi} \leq \bm{b} \}$ for $(\bm{A}, \bm{b})$ from the approximate IP feasibility instance, $K = N_p + 1$, $\bm{y}_k = (y_k^0, \underline{\bm{y}}_k, \overline{\bm{y}}_k) = (0, \epsilon \,\mathbf{e}_k, \mathbf{e} - \epsilon \,\mathbf{e}_k)$ for $k = 1, \ldots, N_p$ and $\bm{y}_{N_p+1} = (y_{N_p+1}^0, \underline{\bm{y}}_{N_p+1}, \overline{\bm{y}}_{N_p+1}) = (1, \bm{0}, \mathbf{e})$. Note that we do not specify the first-stage decision $\bm{x}$ as it is not required for the argument.

Assume first that the approximate IP feasibility problem has an affirmative answer, i.e., there is $\bm{\xi}^\star \in ([0, \epsilon) \cup (1 - \epsilon, 1])^{N_p}$ such that $\bm{A} \bm{\xi}^\star \leq \bm{b}$. Note that $\bm{\xi}^\star \in \Xi$, but for every $k = 1, \ldots, N_p$ we have $\bm{\xi}^\star \notin [\underline{\bm{y}}_k, \overline{\bm{y}}_k]$. Since $\bm{\xi}^\star \in [\bm{0}, \mathbf{e}]$, $\bm{y}_{N_p+1}$ is the only feasible candidate policy, and the optimal value of problem~\eqref{eq:approx_ip_feasibility} is indeed $1$.

Assume now that the optimal value of problem~\eqref{eq:approx_ip_feasibility} is $1$. In that case, there is $\bm{\xi}^\star \in \Xi$ such that $\bm{\xi}^\star \in [\bm{0}, \mathbf{e}] \setminus \bigcup_{k=1}^{N_p} [\epsilon \,\mathbf{e}_k, \mathbf{e} - \epsilon \,\mathbf{e}_k]$, i.e, $\bm{\xi}^\star \in ([0, \epsilon) \cup (1 - \epsilon, 1])^{N_p}$. Since $\bm{A} \bm{\xi}^\star \leq \bm{b}$ by construction of $\Xi$, we conclude that the approximate IP feasibility problem has an affirmative answer.

As for the first part of the second statement, assume first that $\mathcal{X}$, $\mathcal{Y}$ and $\Xi$ are convex and that only the objective function in problem~\eqref{eq:k_adaptability_explicit} is uncertain; that is, $\bm{T} (\bm{\xi}) = \bm{T}$, $\bm{W} (\bm{\xi}) = \bm{W}$ and $\bm{h} (\bm{\xi}) = \bm{h}$ for all $\bm{\xi} \in \Xi$. The classical minimax theorem then implies that
\begin{equation*}
\mspace{-25mu}
\inf_{\substack{\bm{x} \in \mathcal{X}, \\ \bm{y}_1 \in \mathcal{Y}}} \left\{ \sup_{\bm{\xi} \in \Xi} \left\{ \bm{c}^\top \bm{x} + \bm{d} (\bm{\xi})^\top \bm{y}_1 \right\} \, : \, \bm{T} \bm{x} + \bm{W} \bm{y}_1 \leq \bm{h} \right\}
\;\; = \;\;
\inf_{\bm{x} \in \mathcal{X}} \; \sup_{\bm{\xi} \in \Xi} \; \inf_{\bm{y} \in \mathcal{Y}}
\left\{ \bm{c}^\top \bm{x} + \bm{d} (\bm{\xi})^\top \bm{y} \, : \, \bm{T} \bm{x} + \bm{W} \bm{y} \leq \bm{h} \right\},
\end{equation*}
i.e., the $1$-adaptability problem achieves the same objective value as the two-stage robust optimization problem~\eqref{eq:two_stage_ro}. Since the $K$-adaptability problem is bounded from above by the $1$-adaptability problem and from below by the two-stage robust optimization problem, we thus conclude that, under the stated assumptions, the $1$-adaptability problem coincides with the $K$-adaptability problem. The statement now follows from the fact that we can reformulate the $1$-adaptability problem as a linear program by dualizing the inner maximization problem, see~\cite{BTEGN09:rob_opt}.
Note that the convexity of $\Xi$ is needed in this argument since
\begin{equation*}
-1
\; = \;
\max_{\xi \in \{ -1, 1 \}} \; \min_{y \in [-1, 1]} \, \xi y
\; \neq \;
\min_{y \in [-1, 1]} \; \max_{\xi \in \{ -1, 1 \}} \, \xi y
\; = \; 0.
\end{equation*}
A similar argument with $\Xi = [-1, 1]$ and $\mathcal{Y} = \{ -1, 1 \}$ shows that the convexity of $\mathcal{Y}$ is needed, too.

Finally, in view of the second part of the second statement, consider the following problem: 
\begin{equation}\label{eq:approx_ip_feasibility2}
\inf_{\bm{y} \in \mathcal{Y}^K} \sup_{\bm{\xi} \in \Xi} \; \inf_{k \in \mathcal{K}} \left\{ 0 \, : \, \bm{y}_k \leq \bm{\xi} \leq \mathbf{e} - \bm{y}_k \right\},
\end{equation}
where $K = N_p$, $\mathcal{Y} = \{ \bm{y} \in \mathbb{R}^{N_p}_+ \, : \, \mathbf{e}^\top \bm{y} = \epsilon \}$, $\epsilon < 1/2$, and $\Xi = \{ \bm{\xi} \in [0, 1]^{N_p} \, : \, \bm{A} \bm{\xi} \leq \bm{b} \}$ for $(\bm{A}, \bm{b})$ from an approximate IP feasibility instance. We claim that the optimal value of this problem is $+ \infty$, i.e., there is $\bm{\xi} \in \Xi$ for which no decision $\bm{y} \in \mathcal{Y}^K$ is feasible in the second stage, if and only if the approximate IP feasibility instance has an affirmative answer.

Assume first that the approximate IP feasibility problem has an affirmative answer. It then follows from \cite[Lemma~2]{WKS14:DRCO} that the \emph{exact} IP feasibility problem also has an affirmative answer; that is, there is $\bm{\xi}^\star \in \{ 0, 1 \}^{N_p}$ such that $\bm{A} \bm{\xi}^\star \leq \bm{b}$. Fix any feasible candidate policy $\bm{y}_k \in \mathcal{Y}$. By construction, there is $i \in \{ 1, \ldots, N_p \}$ such that $y_{ki} > 0$. Thus, the candidate policy is feasible in the second stage under the parameter realization $\bm{\xi}^\star$ only if $\xi^\star_i \in [y_{ki}, 1 - y_{ki}] \subseteq (0, 1)$, which is not the case since $\xi_i^\star \in \{ 0, 1 \}$. We thus conclude that the optimal value of problem~\eqref{eq:approx_ip_feasibility2} is indeed $+ \infty$ whenever the approximate IP feasibility problem has an affirmative answer.

Assume now that the optimal value of problem~\eqref{eq:approx_ip_feasibility2} is $+ \infty$ and consider the $K$ candidate policies $\bm{y}_k = \epsilon \,\mathbf{e}_k$, $k = 1, \ldots, N_p$. Since $\bm{y} = (\bm{y}_1, \ldots, \bm{y}_K) \in \mathcal{Y}^K$ and problem~\eqref{eq:approx_ip_feasibility2} evaluates to $+ \infty$ under this choice of $\bm{y}$, there must be $\bm{\xi}^\star \in \Xi$ such that $\bm{\xi}^\star \in [\bm{0}, \mathbf{e}] \setminus \bigcup_{k=1}^{N_p} [\epsilon \,\mathbf{e}_k, \mathbf{e} - \epsilon \mathbf{e}_k]$; that is, $\bm{\xi}^\star \in ([0, \epsilon) \cup (1 - \epsilon, 1])^{N_p}$. Since $\bm{A} \bm{\xi}^\star \leq \bm{b}$ by construction of $\Xi$, we conclude that the approximate IP feasibility problem has an affirmative answer.
\end{proof}

We note that $\mathcal{Q} (\bm{x}, \bm{y})$ can be solved in polynomial time if $\Xi$ is convex and the number of policies $K$ is fixed (even when the objective function, the constraint coefficients and the right-hand sides are uncertain), see \cite[Corollary~1]{HKW15:rip}.
The NP-hardness of $\mathcal{Q} (\bm{x}, \bm{y})$ has previously been established under the more restrictive assumption that both the objective function and the constraint right-hand sides in problem~\eqref{eq:k_adaptability_explicit} are uncertain, see \cite[Theorem~3]{HKW15:rip}.
For the special case where $K = 2$, the $K$-adaptability problem with objective and constraint uncertainty can be solved in polynomial time if any of $N_p$, $\max \{ N_1, N_2 \}$ or $L$ is fixed \cite[Proposition~5]{BC10:finite_adaptability}, while the problem becomes NP-hard otherwise~\cite[Proposition~6]{BC10:finite_adaptability}. Proposition~\ref{prop:tractability_k_adaptability}~(ii) provides an alternative proof of the NP-hardness of problem~\eqref{eq:k_adaptability_explicit}, which facilitates a direct comparison to the two-stage robust optimization problem~\eqref{eq:two_stage_ro_explicit}, see Proposition~\ref{prop:tractability_two_stage_ro}~(iii).

We also note that the $K$-adaptability problem~\eqref{eq:k_adaptability_explicit} simplifies in the absence of first-stage decisions $\bm{x}$. For this case, it has been shown in \cite[Theorem~2]{BK16:min_max_max_MP} that the problem can be solved in polynomial time, if only the objective function is uncertain, the deterministic second-stage problem is polynomial time solvable, the uncertainty set $\Xi$ has a tractable representation, and if any number of policies $K > N_2$ is acceptable. The same problem becomes NP-hard, however, when the number of policies $K$ is fixed \cite[Corollary~3]{BK16:min_max_max_MP} or when the uncertainty sets are discrete \cite[Corollaries~1--4]{BK16:min_max_max_discrete}.

Similar to Proposition~\ref{prop:static_two_stage_ro}, we next consider when the $K$-adaptability problem~\eqref{eq:k_adaptability_explicit} reduces to a single-stage robust optimization problem.

\begin{prop}[Reduction to Static Problem]\label{prop:static_k_adaptability}
The problem~\eqref{eq:k_adaptability_explicit} reduces to a static robust optimization problem where $\mathcal{Y}$ is replaced with its convex hull, if $\Xi$ is convex, only the objective function is uncertain and $K > \min \{ N_2, N_p \}$, irrespective of $\mathcal{X}$ and $\mathcal{Y}$.
\end{prop}

\begin{proof}
Assume that $\Xi$ is convex and that only the objective function in problem~\eqref{eq:k_adaptability_explicit} is uncertain; that is, $\bm{T} (\bm{\xi}) = \bm{T}$, $\bm{W} (\bm{\xi}) = \bm{W}$ and $\bm{h} (\bm{\xi}) = \bm{h}$ for all $\bm{\xi} \in \Xi$. Problem~\eqref{eq:k_adaptability_explicit} then simplifies to
\begin{equation*}
\inf_{\substack{\bm{x} \in \mathcal{X}, \\ \bm{y} \in \mathcal{Y}^K}} \; \sup_{\bm{\xi} \in \Xi} \; \inf_{k \in \mathcal{K}} \left\{ \bm{c}^\top \bm{x} + \bm{d} (\bm{\xi})^\top \bm{y}_k \, : \, \bm{T} \bm{x} + \bm{W} \bm{y}_k \leq \bm{h} \right\}.
\end{equation*}
The inner discrete minimization in this problem can be replaced by a continuous minimization over all convex combinations $\bm{\lambda} \in \Delta (\bm{x}, \bm{y}) = \left\{ \bm{\lambda} \in \mathbb{R}^K_+ \, : \, \mathbf{e}^\top \bm{\lambda} = 1, \;\; \lambda_k = 0 \text{ if } \bm{T} \bm{x} + \bm{W} \bm{y}_k \not \leq \bm{h} \right\}$:
\begin{equation*}
\inf_{\substack{\bm{x} \in \mathcal{X}, \\ \bm{y} \in \mathcal{Y}^K}} \; \sup_{\bm{\xi} \in \Xi} \; \inf_{\bm{\lambda} \in \Delta (\bm{x}, \bm{y})} \left\{ \bm{c}^\top \bm{x} + \sum_{k \in \mathcal{K}} \lambda_k \bm{d} (\bm{\xi})^\top \bm{y}_k \right\}
\end{equation*}
The classical minimax theorem then allows us to exchange the order of the inner two operators:
\begin{equation}\label{eq:pre_caratheodory}
\inf_{\substack{\bm{x} \in \mathcal{X}, \\ \bm{y} \in \mathcal{Y}^K}} \; \inf_{\bm{\lambda} \in \Delta (\bm{x}, \bm{y})} \; \sup_{\bm{\xi} \in \Xi} \left\{ \bm{c}^\top \bm{x} + \bm{d} (\bm{\xi})^\top \left[ \sum_{k \in \mathcal{K}} \lambda_k \bm{y}_k \right] \right\}
\end{equation}
If $K \geq N_2 + 1$, then we can use Carath{\'e}odory's theorem to replace the convex combinations of $K$ candidate decisions $\bm{y}_k \in \mathcal{Y}$ with a single decision from the convex hull, $\text{conv } \mathcal{Y}$, resulting in
\begin{equation}\label{eq:post_caratheodory}
\inf_{\substack{\bm{x} \in \mathcal{X}, \\ \bm{y} \in \text{conv}\, \mathcal{Y}}} \left\{ \sup_{\bm{\xi} \in \Xi} \left\{ \bm{c}^\top \bm{x} + \bm{d} (\bm{\xi})^\top \bm{y} \right\} \, : \, \bm{T} \bm{x} + \bm{W} \bm{y} \leq \bm{h} \right\},
\end{equation}
which is readily recognized as a single-stage robust optimization problem. Similarly, problem~\eqref{eq:pre_caratheodory} can be rewritten as
\begin{equation*}
\inf_{\substack{\bm{x} \in \mathcal{X}, \\ \bm{y} \in \mathcal{Y}^K}} \; \inf_{\bm{\lambda} \in \Delta (\bm{x}, \bm{y})} \; \sup_{\bm{\xi} \in \Xi} \left\{ \bm{c}^\top \bm{x} + \sum_{k \in \mathcal{K}} \lambda_k \left[ \bm{d} (\bm{\xi})^\top \bm{y}_k \right] \right\},
\end{equation*}
and if $K \geq N_p + 1$, then we can use Carath{\'e}odory's theorem to replace the convex combinations of $K$ candidate decisions $\bm{d} (\bm{\xi})^\top \bm{y}_k \in \left\{ \bm{d} (\bm{\xi})^\top \bm{y} : \bm{y} \in \mathcal{Y} \right\}$ with a single decision from the convex hull of their domain, $\text{conv} \left\{ \bm{d} (\bm{\xi})^\top \bm{y} : \bm{y} \in \mathcal{Y} \right\} = \left\{ \bm{d} (\bm{\xi})^\top \bm{y} : \bm{y} \in \text{conv } \mathcal{Y} \right\}$, resulting again in the single-stage robust optimization problem~\eqref{eq:post_caratheodory}.

As in Proposition~\ref{prop:static_two_stage_ro}, the previous arguments require $\Xi$ to be convex. Indeed, we have
\begin{equation*}
-1
\; = \;
\sup_{\xi \in \{ -1, 1 \}} \; \inf_{k \in \{ 1, 2 \}} \, \xi y_k
\; \neq \;
\inf_{y \in [-1, 1]} \; \sup_{\xi \in \{ -1, 1 \}} \, \xi y
\; = \;
0
\qquad
\text{with } y_1 = -1 \text{ and } y_2 = 1,
\end{equation*}
and we cannot establish equivalence by replacing $\Xi$ with $\text{conv } \Xi = [-1, 1]$ in the second optimization problem either. To see that the number of policies $K$ must exceed $\min \{ N_2, N_p \}$ in general in the previous arguments, we compare the two problems
\begin{equation*}
\inf_{\bm{y} \in (\{ -1, 1 \}^{N_2})^K} \; \sup_{\bm{\xi} \in [-1, 1]^{N_p}} \; \inf_{k \in \mathcal{K}} \, \xi_1 y_{k1}
\qquad \text{and} \qquad
\inf_{\bm{y} \in [-1, 1]^{N_2}} \; \sup_{\bm{\xi} \in [-1, 1]^{N_p}} \, \xi_1 y_1.
\end{equation*}
For every $N_2, N_p \geq 1$, the problem on the right attains its optimal value $0$ at any $\bm{y} \in [-1, 1]^{N_2}$ with $y_1 = 0$. Similarly, for every $N_2, N_p \geq 1$, the problem on the left attains an optimal value of $0$, if $K \geq 2$, and an optimal value of $1$, if $K = 1$. We thus verify that $K > \min \{ N_2, N_p \}$ is required for the optimal values to coincide by choosing $(N_2, N_p) \in \{ (1, 2), (2, 1) \}$.
\end{proof}

We close this section with several special cases where the optimal value of the $K$-adaptability problem~\eqref{eq:k_adaptability_explicit} coincides with the optimal value of the two-stage robust optimization problem~\eqref{eq:two_stage_ro_explicit}.

\begin{prop}[Optimality]\label{prop:optimality_k_adaptability}
The optimal values of problems~\eqref{eq:two_stage_ro_explicit} and~\eqref{eq:k_adaptability_explicit} coincide if \emph{(i)} $\mathcal{Y}$ and $\Xi$ are convex and only the objective function is uncertain, irrespective of $\mathcal{X}$ and $K$, or if \emph{(ii)} $\Xi$ is convex, only the objective function is uncertain and $K > \min \{ N_2, N_p \}$, irrespective of $\mathcal{X}$ and $\mathcal{Y}$, or if \emph{(iii)} $\mathcal{Y}$ has a finite cardinality and $K \geq | \mathcal{Y} |$, irrespective of $\mathcal{X}$ and $\Xi$. Otherwise, their optimal values may differ for any finite $K$, even if $\mathcal{X}$, $\mathcal{Y}$ and $\Xi$ are convex and only the constraint right-hand sides are uncertain.	
\end{prop}

\begin{proof}
The fact that the two optimal values coincide under the first set of conditions follows from the proof of Proposition~\ref{prop:tractability_k_adaptability}~(ii), where we have shown that, under the stated assumptions, the optimal values of the $1$-adaptability problem and the two-stage robust optimization problem~\eqref{eq:two_stage_ro_explicit} coincide. We have also shown there that the convexity of $\mathcal{Y}$ and $\Xi$ is crucial for the proof to hold.

The fact that the two optimal values coincide under the second set of conditions follows from Proposition~\ref{prop:static_k_adaptability}, which shows that, under the stated assumptions, the $K$-adaptability problem~\eqref{eq:k_adaptability_explicit} is equivalent to
\begin{equation*}
\inf_{\substack{\bm{x} \in \mathcal{X}, \\ \bm{y} \in \text{conv} \mathcal{Y}}} \left\{ \sup_{\bm{\xi} \in \Xi} \left\{ \bm{c}^\top \bm{x} + \bm{d} (\bm{\xi})^\top \bm{y} \right\} \, : \, \bm{T} \bm{x} + \bm{W} \bm{y} \leq \bm{h} \right\}.
\end{equation*}
The classical minimax theorem, which is applicable since $\text{conv} \, \mathcal{Y}$ and $\Xi$ are convex, then implies that this problem is equivalent to
\begin{equation*}
\inf_{\bm{x} \in \mathcal{X}} \; \sup_{\bm{\xi} \in \Xi} \; \inf_{\bm{y} \in \text{conv} \mathcal{Y}} \left\{ \bm{c}^\top \bm{x} + \bm{d} (\bm{\xi})^\top \bm{y} \, : \, \bm{T} \bm{x} + \bm{W} \bm{y} \leq \bm{h} \right\},
\end{equation*}
which provides a lower bound to the two-stage robust optimization problem~\eqref{eq:two_stage_ro_explicit} since $\text{conv } \mathcal{Y} \supseteq \mathcal{Y}$. On the other hand, we know that the $K$-adaptability problem~\eqref{eq:k_adaptability_explicit} by construction bounds~\eqref{eq:two_stage_ro_explicit} from above. We thus conclude that the optimal values of both problems must coincide.

In view of the third set of conditions, it is clear that $K \geq | \mathcal{Y} |$ policies are sufficient for the optimal values of problems~\eqref{eq:two_stage_ro_explicit} and~\eqref{eq:k_adaptability_explicit} to coincide. Moreover, \cite[Theorem~4]{HKW15:rip} presents a problem where the optimal values of problems~\eqref{eq:two_stage_ro_explicit} and~\eqref{eq:k_adaptability_explicit} differ for every $K < | \mathcal{Y} |$.

As for the second part of the statement, consider the problem
\begin{equation*}
\sup_{\xi \in [-1, 1]} \; \inf_{\substack{\tau \in \mathbb{R}, \\ y \in [-1, 1]}} \left\{ \tau \, : \, \tau \geq \xi - y, \;\; \tau \geq y - \xi \right\}.
\end{equation*}
The optimal second-stage decision to this two-stage robust optimization problem satisfies $\tau (\xi) = 0$ and $y (\xi) = \xi$, and any $\xi \in [-1, 1]$ attains the optimal objective value $0$. Consider now the $K$-adaptability problem
\begin{equation*}
\inf_{\substack{\bm{\tau} \in \mathbb{R}^K, \\ \bm{y} \in [-1, 1]^K}} \; \sup_{\xi \in [-1, 1]} \;  \inf_{k \in \mathcal{K}} \left\{ \tau_k \, : \, \tau_k \geq \xi - y_k, \;\; \tau_k \geq y_k - \xi \right\}.
\end{equation*}
By construction, the objective value of this problem evaluates to at least $\max_{\xi \in [-1, 1]} \; \min_{k \in \mathcal{K}} \, | \xi - y_k | > 0$ for any finite $K$ and any feasible solution $(\bm{\tau}, \bm{y}) \in \mathbb{R}^K \times [-1, 1]^K$.
\end{proof}

The validity of the first part of the statement under the second set of conditions in Proposition~\ref{prop:optimality_k_adaptability} was established for the subclass of purely binary $K$-adaptability problems in~\cite[Theorem~1]{HKW15:rip}. We also mention that the optimal value of the $K$-adaptability problem~\eqref{eq:k_adaptability_explicit} approaches the optimal value of the two-stage robust optimization problem~\eqref{eq:two_stage_ro_explicit} if $K \rightarrow \infty$ and a continuity assumption is satisfied, see \cite[Proposition 1]{BC10:finite_adaptability}.

\end{appendices}

\end{document}